\documentclass[11pt]{article}
\usepackage{setspace}
\usepackage[utf8]{inputenc}
\usepackage{geometry}
\usepackage{amsmath}
\usepackage{bbm}
\usepackage{amsthm}
\usepackage{amssymb}
\usepackage{mathtools}
\usepackage{relsize}
\usepackage{tikz-cd}
\usepackage{enumitem}
\usepackage{chngcntr}
\usepackage{mathrsfs}
\usepackage{thmtools}
\usepackage{centernot}
\usepackage{xcolor}
\usepackage{hyperref}
\usepackage{graphicx}
\usepackage{listings}
\usepackage{tikz}
\usepackage{bm}
\usepackage{listings}
\usepackage{cancel}
\usepackage{relsize}
\usepackage{authblk}
\usepackage{caption}
\usepackage{subcaption}
\usepackage{pgfplots}
\usepackage{float}

\usepackage{chngcntr}
\counterwithout{equation}{subsection}
\counterwithout{equation}{section}

\usepackage[small]{titlesec}

\titleformat*{\section}{\centering\bfseries}
\titleformat*{\subsection}{\bfseries}

\usepackage{cleveref}
\newtheorem{theorem}{Theorem}[section]
\newtheorem{thm}{Theorem}[section]
\newtheorem{propn}[thm]{Proposition}
\newtheorem{lemma}[thm]{Lemma}

\theoremstyle{remark}
\newtheorem{remark}[thm]{Remark}
\newtheorem{corollary}[thm]{Corollary}

\newtheorem{example}{Example}[section]

\theoremstyle{definition}

\theoremstyle{plain}

\newtheorem*{assumption*}{Assumptions}
\newtheorem*{assumption**}{Assumptions, continued}

\allowdisplaybreaks

\hypersetup{
    colorlinks=true,
    linkcolor=red,
    citecolor=blue,
    urlcolor=blue,
    pdfborder={0 0 0}
}

\DeclareMathOperator{\indset}{\mathbb{I}}

\DeclareMathOperator{\prob}{\mathbb{P}}
\DeclareMathOperator{\exptn}{\mathbb{E}}
\DeclareMathOperator*{\argmax}{arg\,max}

\newcommand{\bs}[1]{{\boldsymbol #1 }}

\newcommand{\vb}{\vspace{3.2mm}}

\title{\smaller \smaller \textbf{STATISTICAL INFERENCE FOR A SERVICE SYSTEM WITH NON-STATIONARY ARRIVALS AND UNOBSERVED BALKING}}

\author{Shreehari Anand Bodas, Michel Mandjes, and Liron Ravner}

\date{}

\vspace{-1cm}

\numberwithin{equation}{section}
\begin{document}

\maketitle

\begin{abstract}
    \small \noindent
    We study a multi-server queueing system with a periodic arrival rate and customers whose joining decision is based on their patience and a delay proxy. 
    Specifically, each customer has a patience level sampled from a common distribution. 
    Upon arrival, they receive an estimate of their delay before joining service and then join the system only if this delay is not more than their patience, otherwise they balk.  
    The main objective is to estimate the parameters pertaining to the arrival rate and patience distribution. 
    Here the complication factor is that this inference should be performed based on the observed process only, i.e., balking customers remain unobserved. 
    We set up a likelihood function of the state dependent effective arrival process (i.e., corresponding to the customers who join), establish strong consistency of the MLE, and derive the asymptotic distribution of the estimation error. 
    Due to the intrinsic non-stationarity of the Poisson arrival process, the proof techniques used in previous work become inapplicable. 
    The novelty of the proving mechanism in this paper lies in the procedure of constructing i.i.d.\ objects from dependent samples by decomposing the sample path into i.i.d.\ regeneration cycles. 
    The feasibility of the MLE-approach is discussed via a sequence of numerical experiments, for multiple choices of functions which provide delay estimates.
    In particular, it is observed that the arrival rate is best estimated at high service capacities, and the patience distribution is best estimated at lower service capacities.

\vb
    
\noindent
{\sc Keywords.} Service systems $\circ$ balking $\circ$ periodic arrivals $\circ$ estimation $\circ$ customer patience

\vb

\noindent
{\sc Affiliations.} 
SAB is with the Korteweg-de Vries Institute for Mathematics, University of Amsterdam, Science Park 904, 1098 XH Amsterdam, The Netherlands. \url{s.a.bodas@uva.nl}

\noindent 
MM is with the Mathematical Institute, Leiden University, P.O. Box 9512,
2300 RA Leiden,
The Netherlands. He is also affiliated with Korteweg-de Vries Institute for Mathematics, University of Amsterdam, Amsterdam, The Netherlands; E{\sc urandom}, Eindhoven University of Technology, Eindhoven, The Netherlands; Amsterdam Business School, Faculty of Economics and Business, University of Amsterdam, Amsterdam, The Netherlands. \url{m.r.h.mandjes@math.leidenuniv.nl}

\noindent
LR is with the Department of Statistics, University of Haifa, 199 Aba Khoushy Avenue, Mount Carmel, Haifa, Israel. \url{lravner@stat.haifa.ac.il}

\vb

\noindent
{\sc Acknowledgments.} 
This research was supported by the European Union’s Horizon 2020 research and innovation programme under the Marie Skłodowska-Curie grant agreement no.\ 945045, by the NWO Gravitation project NETWORKS under grant agreement no.\ 024.002.003 and by the Israel Science Foundation (ISF), grant no.\ 1361/23. \includegraphics[height=1em]{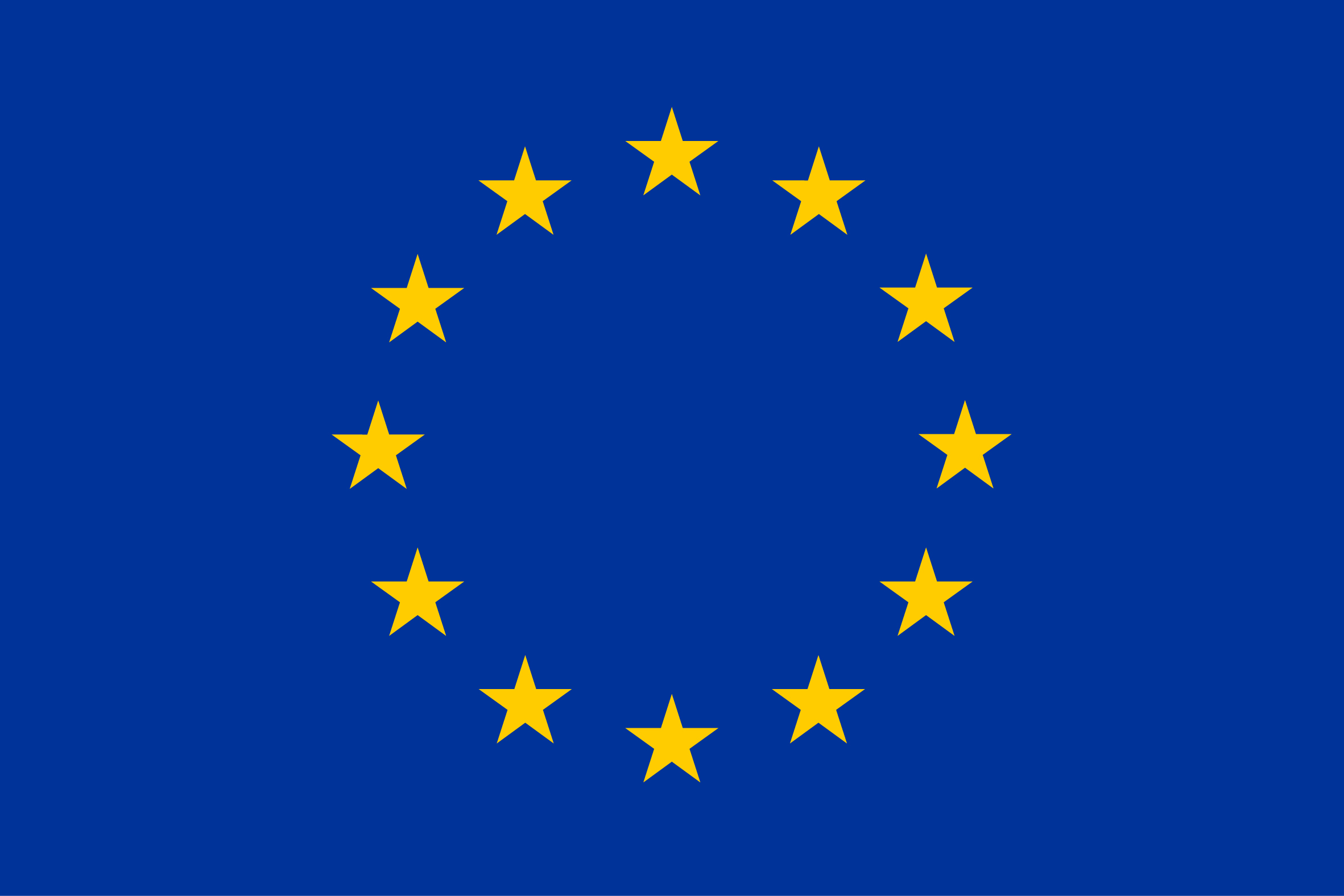} 
Date: {\it \today}.

\end{abstract}





\newpage

\section{INTRODUCTION}\label{section: Introduction}

In our modern world, customers interact with service systems on a daily basis. 
Examples abound across a broad range of economic sectors: besides conventional shops, one can think of online delivery services, communication and computing networks, call centers, and various types of healthcare systems. 
Seen from a conceptual perspective, in these systems the service capacity is a scarce resource, for which customers compete. 
When these customers are provided information on the current occupancy level, they may find the anticipated service level too poor, so that they decide not to join, a phenomenon known as `balking'. 
Such balking decisions are typically made based on an estimate of the delay the customer under consideration would be facing, or, more generally, the performance level they would receive.

When operating a service system, it is of crucial importance to have an accurate quantitative description of its underlying dynamics. 
Often one resorts to a queueing-theoretic framework, encompassing a customer arrival process, a specification of the service requirements, and a service mechanism.
In a practical context, this means that the `model primitives' are estimated, by performing statistical inference based on observations of the system over time. 
A frequently overlooked complication, however, is that often balking customers remain unobserved. 
Despite the fact that balking customers do not join the system, a system operator still wishes to reliably estimate the effective arrival rate, i.e., the rate of joining customers plus the rate of balking customers.
The reason why this total arrival rate is relevant, is that it reflects the demand for the service, which is of key importance in decisions concerning e.g.\ staffing, dimensioning or pricing. 
Healthcare service providers, for instance, could use knowledge of the total demand to decide whether extra doctors or beds are required.
Along the same lines, delivery services could consider a price adjustment so as to enlarge their market share. 

A second complication that is ignored in many studies, is that the demand for service often fluctuates with time, typically following daily or weekly patterns; see for instance \cite{ibrahim2013forecasting,ibrahim2016modeling} for analyses in the call center context. 
This means that while many modelling frameworks assume a time-independent Poissonian effective arrival rate, it would be more natural to work with a rate that follows some periodic pattern. 

When setting up an estimation approach, one would like to be sure it provides reliable output.
This concretely means that, ideally, the proposed estimator is consistent (i.e., converge to the true values when the sample size grows large), while in addition we would like to have insight into its asymptotic properties (for instance by proving that the estimator is asymptotically normally distributed).

The main conclusion from the above considerations, is that there is a clear need for a robust methodology to statistically estimate the service system's model parameters. 
To make sure it can be used in an operational context, it should fulfil the following three requirements: (a)~it should be based on observable data only, i.e., data corresponding to the joining customers, (b)~it should be able to deal with a non-constant effective arrival rate, and (c)~the proposed procedure should be equipped with provable performance guarantees. 

\vb

In this paper we choose to model the service system as a first-come-first-served queue of the M$_t$/G/$s$+H type.
We have selected this class of models because of its generality, capturing a broad range of relevant service systems.
It can be described as follows:\begin{itemize}
    \item[$\circ$]
In the first place, customers arrive according to a non-homogeneous periodic Poissonian arrival rate $\lambda_{\bs{\alpha}}(\cdot)\ge  0$.
Acknowledging that this rate typically follows a daily (or weekly) pattern, we throughout impose the natural assumption that we know the length of the corresponding period.
In our setup we develop a {\it parametric} inference procedure: we let the arrival rate be parametrized by an unknown (finite-dimensional) parameter vector $\bs{\alpha}$. 
\item[$\circ$] Each customer brings along a random amount of work. These service requirements for an i.i.d.\ sequence of non-negative random variable, with cumulative distribution function $G(\cdot)$, sampled independently of the arrival process. 
\item[$\circ$]
There are $s\in{\mathbb N}$ servers who serve in a first-come-first-serve manner. As it is a design parameter that is under full control of the service provider, throughout this paper we assume we know $s$.
\item[$\circ$]
Our system explicitly takes into account the option of balking by incorporating the following impatience mechanism. 
Each arriving customer is equipped with a patience level, sampled independently of everything else. The customers' patience levels are i.i.d., distributed as the non-negative random variable $Y$, characterized through the cumulative distribution function $H_{\bs{\theta}}(\cdot)$. 
Again, we work in a parametric context, with $\bs{\theta}$ representing the underlying (finite-dimensional) parameter vector.
At every time instant $t \geq 0$, the service operator announces $\Delta(t)$ - an estimate of the delay before service which a customer arriving at time $t$ can expect given the state of the system at time $t$.
By comparing this delay estimate with their patience, customers decide to join or to balk. 
\item[$\circ$]
In this paper, we consider a general delay announcement mechanism, unlike \cite{inoue2023estimating} which analyzes only the scenario $\Delta(t) = V(t)$ (where $V(t)$ is the virtual waiting time at time $t$). In particular, we consider the form $\Delta(t) = \psi\big(\mathcal{R}(t)\big)$, where $\psi$ is a function which maps the state of the system to a corresponding delay estimate.
For example, in some systems delay announcements may be estimated based on the number of customers in the system; think of a call center or a physical shop.
Once a customer joins the system, they do not leave until service completion.
\end{itemize}

In the context of this M$_t$/G/$s$+H setup, our concrete objective is to devise a procedure, endowed with provable performance guarantees, to estimate the parameter vector $(\bs{\alpha}, \bs{\theta})$ governing the arrival rate and the patience distribution, relying on information regarding the non-balking customers only. 
Our procedure is based on maximum likelihood estimation (MLE), exploiting that the system's key dynamics are dictated by the state-dependent effective arrival rate in the following way. 
Suppose that at a given point in time, which we associate for ease with time $0$, the virtual waiting time is $v$; this virtual waiting time is defined as the waiting time of a hypothetical customer joining at time $0$. 
If there are no effective arrivals in $(0,t]$, then the virtual waiting time at time $t$ is $(v-t)^{+}$, so that the effective arrival rate at time $t$ becomes \[\lambda_{\bs{\alpha}}(t)\big(1-H_{\bs{\theta}}((v-t)^{+})\big).\] 
Observe how this arrival rate depends both on time as well as the state of the system (or, more concretely, the current congestion level). 
Various other variants can be thought of, such as the above-mentioned mechanism in which the delay announcement is a function of the number of customers in the system.
Note that in the latter case, between successive events (which can be arrivals or departures), the arrival rate depends only on time. 
The idea is then that once we know the state-dependent effective arrival rate, we can calculate the distribution of the effective interarrival times conditioned on the system state, which we use to produce a closed-form expression for the conditional log-likelihood of the arrival process.

When attempting to estimate $(\bs{\alpha}, \bs{\theta})$, there are various additional complications to be dealt with. 
An important challenge is that when the non-homogeneous arrival rate is significantly greater than the service capacity, the system operator only observes the joining customers may incorrectly conclude that the arrival rate is a small constant, as explained in Example \ref{example: best and worst instances for estimation}. 
Another challenge concerns the  (possible lack of) identifiability of the arrival and patience distribution functions with respect to the parameters, as explained in Example \ref{example: identifiability issues}.   



\begin{example}\label{example: best and worst instances for estimation}
    \begin{figure}[t!]
    \begin{subfigure}{0.48\textwidth}
    \includegraphics[width=\linewidth]{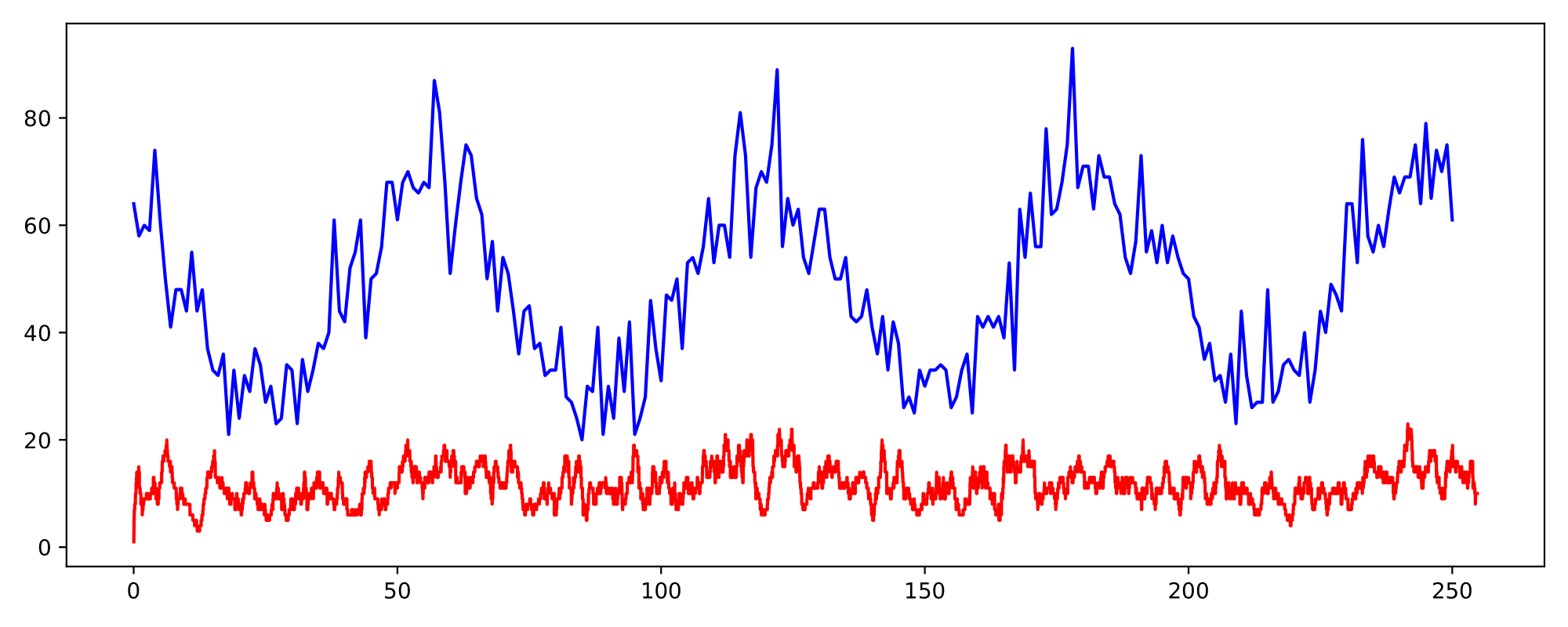}
    \caption{$s=1$, balk\:=\:90\%} \label{fig:a}
    \end{subfigure}\hspace*{\fill}
    \begin{subfigure}{0.48\textwidth}
    \includegraphics[width=\linewidth]{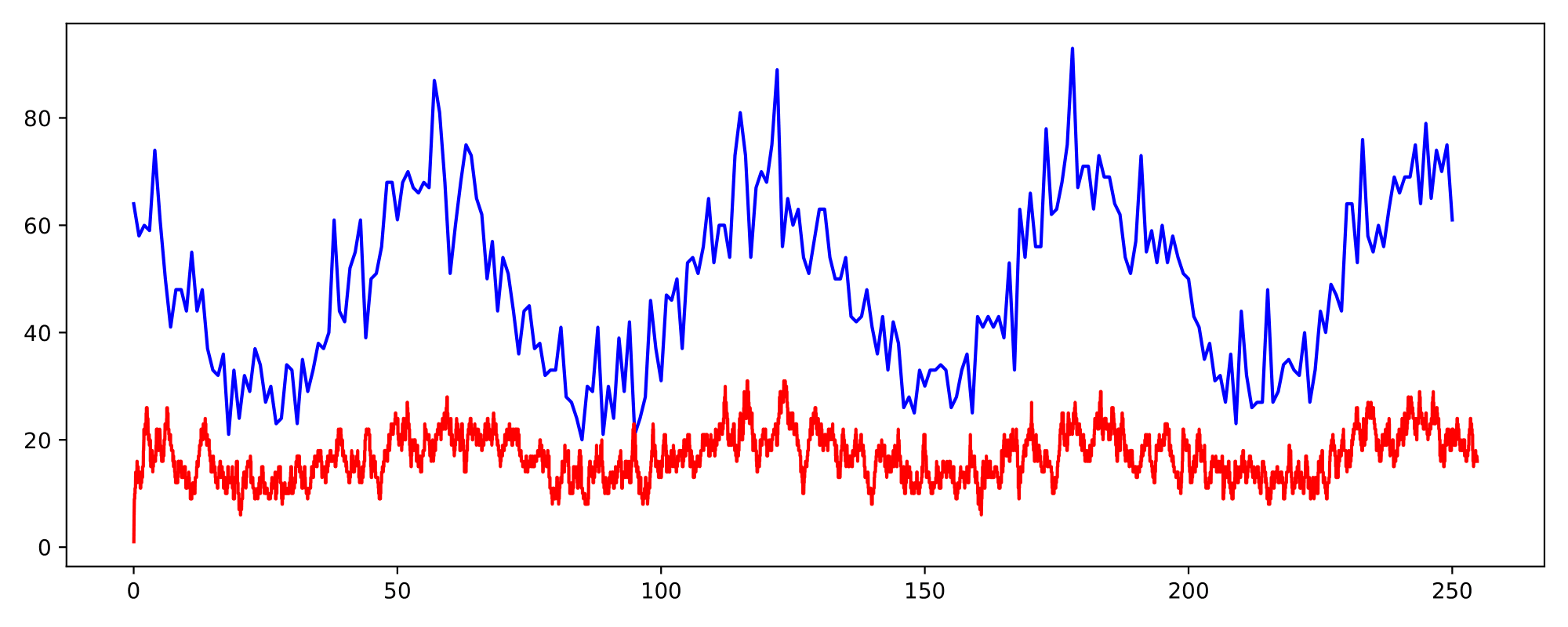}
    \caption{$s=2$, balk\:=\:80\%} \label{fig:b}
    \end{subfigure}
    
    \medskip
    \begin{subfigure}{0.48\textwidth}
    \includegraphics[width=\linewidth]{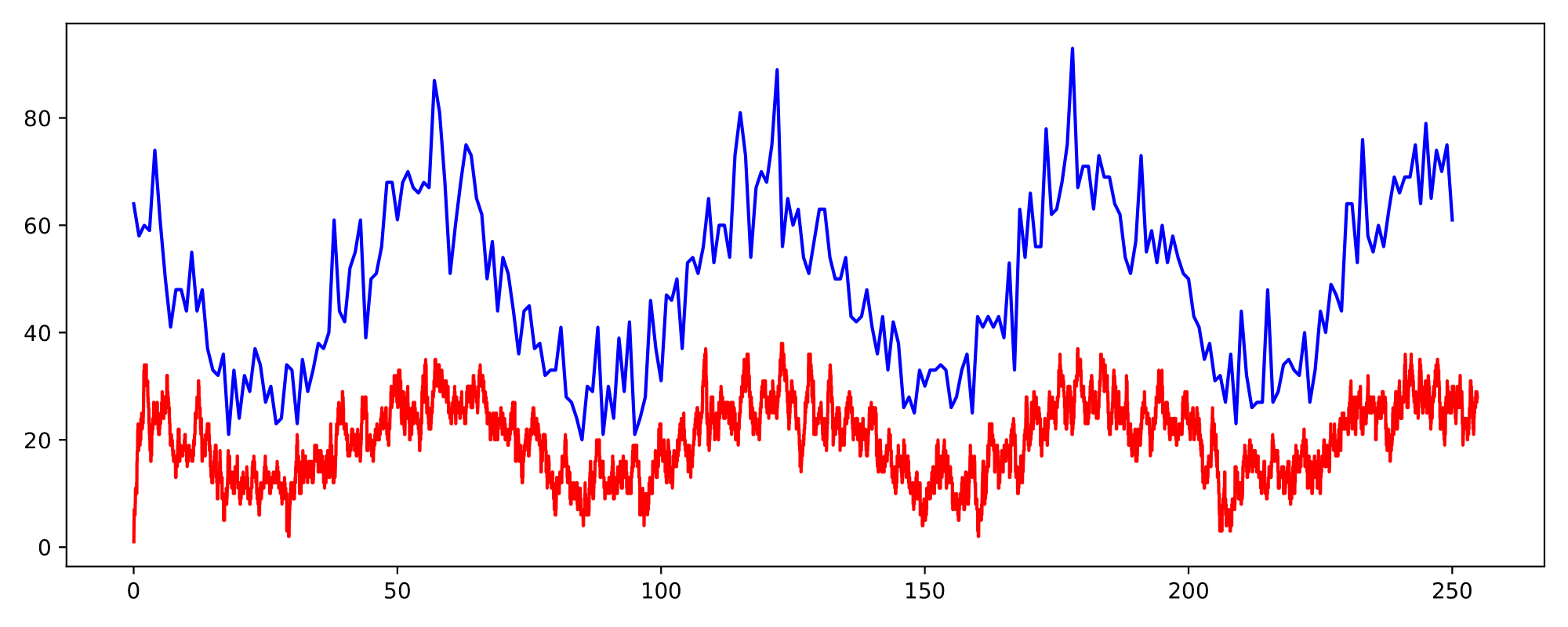}
    \caption{$s=4$, balk\:=\:60\%} \label{fig:c}
    \end{subfigure}\hspace*{\fill}
    \begin{subfigure}{0.48\textwidth}
    \includegraphics[width=\linewidth]{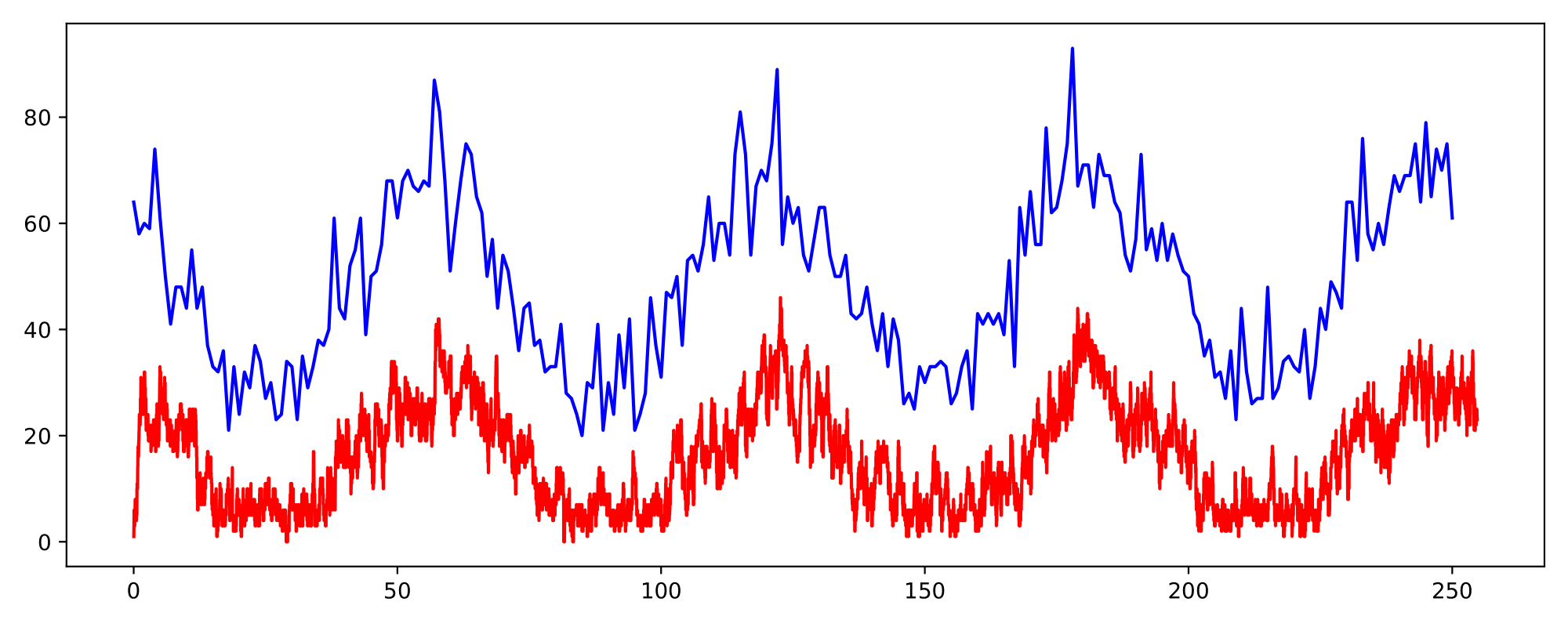}
    \caption{$s=8$, balk\:=\:25.7\%} \label{fig:d}
    \end{subfigure}
    
    \medskip
    \begin{subfigure}{0.48\textwidth}
    \includegraphics[width=\linewidth]{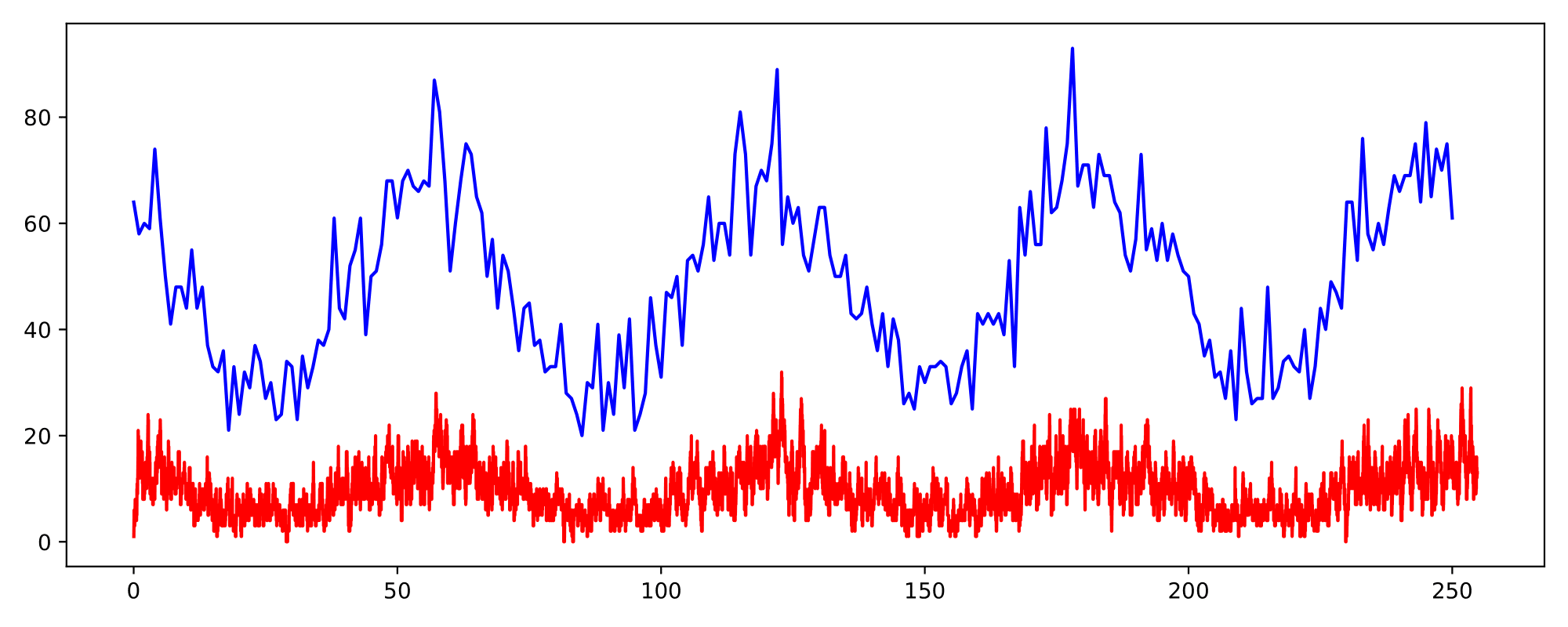}
    \caption{$s=16$, balk\:=\:0.8\%} \label{fig:e}
    \end{subfigure}\hspace*{\fill}
    \begin{subfigure}{0.48\textwidth}
    \includegraphics[width=\linewidth]{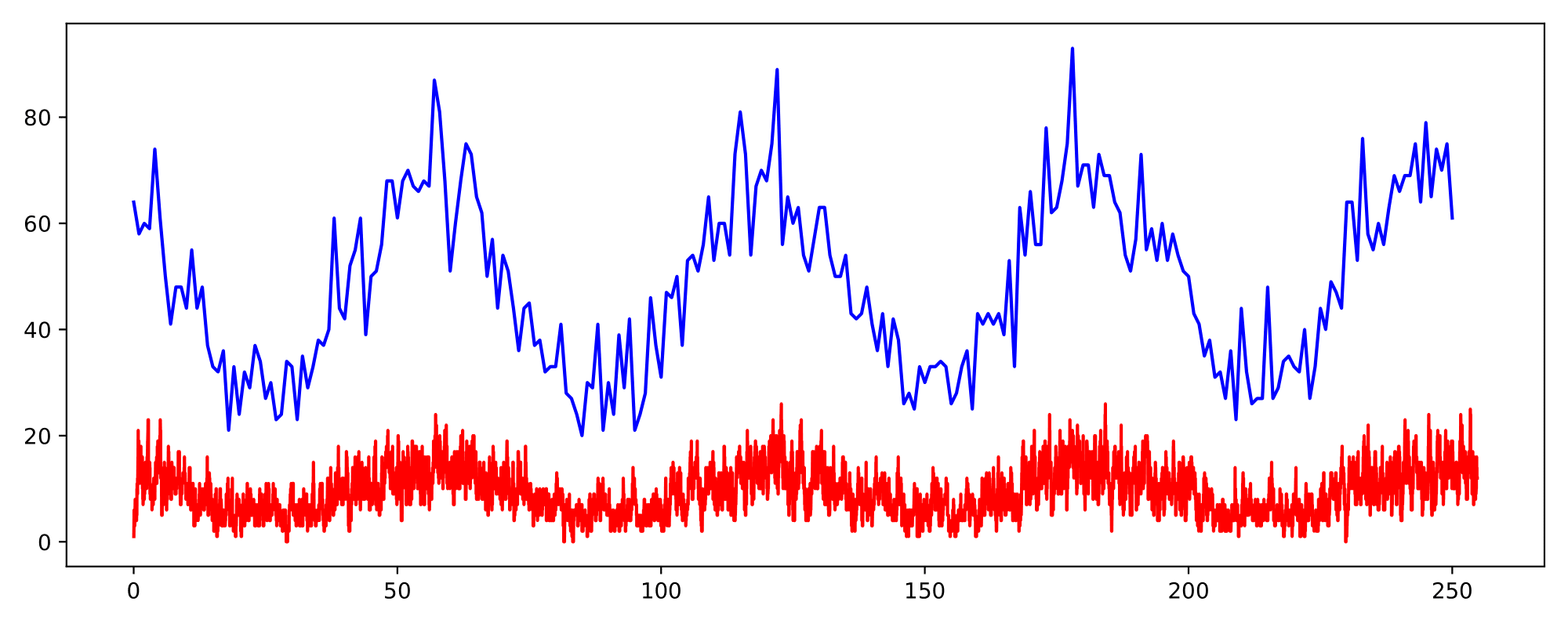}
    \caption{$s=32$, balk\:=\:0\%} \label{fig:f}
    \end{subfigure}

    
    \caption{Effect of impatience and service capacity on state of system. Here `balk' is the percentage of customers balking. Blue lines represents the arrival rate, and red lines the number of customers in the system} \label{fig:1}
    \end{figure}
    
    Figure \ref{fig:1} shows the time evolution of six M$_t$/G/$s$+H systems. 
    We throughout assume that the arrival rate is given through a sinusoidal with period $20\,\pi$, given by $\lambda_{\bs{\alpha}_0}(t) = 50+20\sin(1-0.1t)$. 
    Both the patience distribution function $H_{\bs{\theta}_0}(\cdot)$ and the service requirements distribution function $G(\cdot)$ are taken exponentially with mean 1. 
    The number of servers we vary: we consider $s \in \{1,2,4,8,16,32\}$. 
    The specific realizations of the arrival times, service requirements, and patience times are kept unchanged across all experiments.
    
    The blue line is a realization of the empirical arrival process with rate $\lambda_{\bs{\alpha}_0}(t)$. 
    For each integer-valued time point $n$, the y-coordinate of the blue line represents the total number of arrivals (balking plus non-balking, that is) that have occurred in $[n,n+1)$. 
    The red line plots the number of customers in the system as a function of time. 
    
    We proceed by examining how the red line behaves across the six instances. 
    For the cases $s=1$ and $s=2$, the waiting times are so high that a dominant fraction of customers balk. 
    By observing only the red line in Figures \ref{fig:a} and \ref{fig:b}, one cannot infer that that the arrival rate is actually large and non-homogeneous. 
    In fact, if the administrator is unaware of balking they may even conclude that the arrival process is stationary and not time-dependent.
    Furthermore, from a statistical perspective, balking customers help estimate the patience parameters $\bs{\theta}_0$, but not the arrival rate parameters $\bs{\alpha}_0$.
    On the other hand, for the cases $s=16$ and $s=32$, the service capacity is much higher than the arrival rate and therefore, so that only a negligible percentage of the arriving customers balk. 
    As a consequence, the estimates of $\bs{\alpha}_0$ are precise whereas those of $\bs{\theta}_0$ are poor. 
    In a way, the `best' instances are therefore the middle ones (i.e., $s=4$ and $s=8$), where both $\bs{\alpha}_0$ and $\bs{\theta}_0$ allow a reasonably accurate estimation. \hfill\mbox{$\Diamond$}
    \begin{figure}
        \centering

        \begin{tikzpicture}
          \begin{axis}[
            xlabel=$x$,
            ylabel=$\prob(Y \leq x)$,
            xmin=0,
            xmax=16,
            ymin=0,
            ymax=1,
            xtick={0,2,...,16},
            grid=both,
            axis lines=left,
            width=14cm,
            height=7cm,  
            legend style={
              at={(0.5,0.3)},  
              anchor=north,
            }
          ]
          
          \addplot[domain=0:16, samples=400, smooth, blue] {1 - 0.8*exp(-x) - 0.2*exp(-0.1*x)};
          \addplot[domain=0:16, samples=400, smooth, red] {1 - 0.85*exp(-0.91*x) - 0.15*exp(-0.05*x)};
          
          \legend{True distribution, Estimated distribution}
          
          \end{axis}
        \end{tikzpicture}
        
        \caption{Identifiability challenge while estimating patience distribution}
        \label{fig:identifiability_challenge_demonstration}
    \end{figure}
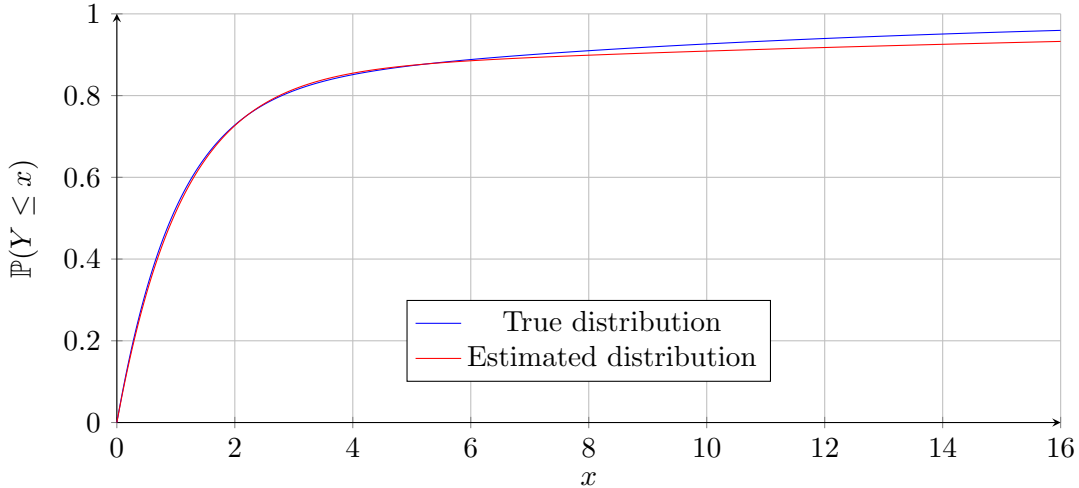
\end{example}

\begin{example}\label{example: identifiability issues}
    We simulate an M$_t$/G/$s$+H queue characterized by $\lambda_{\bs{\alpha}_0}(t) = 50 + 15\sin(4-0.1t) + 10\sin(1-0.5t)$, $H_{\bs{\theta}_0}(x) = 1 - 0.8e^{-x} - 0.2e^{-0.1x}$, $G \sim \text{Exp}(0.4)$, and $100\,000$ arrivals (balking plus non-balking, that is). 
    The true parameter vectors are therefore $\bs{\alpha}_0 = (50, 15, 10, 4, 1)$ and $\bs{\theta}_0 = (0.8, 1, 0.1)$. 
    By applying our estimation procedure, described later in this paper, the estimates that we obtain are $\hat{\bs\alpha} = (49.44, 14.95, 10.41, 4.00, 0.98)$ and $\hat{\bs\theta} = (0.85, 0.91, 0.05)$. 
    Observe that $\hat{\bs\alpha}$ is reasonably accurate, whereas $\hat{\bs\theta}$ is relatively far from $\bs{\theta}_0$. 
    However, now observe in Figure \ref{fig:identifiability_challenge_demonstration}, where the blue and red lines plot the distribution functions $H_{{\bs{\theta}_0}}(x)$ and $H_{\hat{\bs\theta}}(x)$ respectively, that the estimate $\hat{\bs\theta}$ being inaccurate does not mean that the corresponding estimate $H_{\hat{\bs\theta}}(x)$ is. 
    Put differently, the inaccuracy in the estimation of $\bs{\theta}_0$ is in a way harmless, because $H_{\bs{\theta}_0}(x)$ is estimated with great precision.
    One could say that we are facing a `quasi-unidentifiability' issue: multiple parameter vectors lead to `almost identical' models. 

    The system's congestion level plays a crucial role in the ability to accurately estimate the patience parameters.
    In this simulation the maximum virtual waiting time observed by a customer is 5.304.
    This is less than 6, which is, as revealed by Figure \ref{fig:identifiability_challenge_demonstration}, the value beyond which one can really differentiate between $H_{\bs{\theta}_0}(x)$ and $H_{\hat{\bs\theta}}(x)$.
    When increasing the number of arrivals, the maximum observed virtual waiting time will increase, thus reducing the gap between $\hat{\bs\theta}$ and $\bs{\theta}_0$. \hfill\mbox{$\Diamond$}

\end{example}

\subsection{Contributions}

This paper succeeds in devising an estimation methodology that fulfils the requirements~(a),~(b), and~(c) defined above. 
The proposed approach parametrically estimates the arrival rate and patience distribution parameter vectors $(\bs{\alpha}, \bs{\theta})$ in an M$_t$/G/$s$+H system by only observing information associated with the non-balking customers. 
It thus extends previous work \cite{inoue2023estimating} by incorporating a time-dependent customer arrival rate. 
Due to this time-dependence one cannot use a proof technique in which one has assumed that the system has reached stationarity. 
In greater detail, the contributions are the following:


\begin{itemize}
    \item[$\circ$] Our estimation procedure is based on maximum likelihood.
    One would initially think that such an approach should make use of closed-form (transient or stationary) results for the M$_t$/G/$s$+H queue, but those are hardly available; even the (simpler) M/G/$s$ system is essentially intractable. 
    We circumvent this problem, however, by constructing a closed-form expression for the conditional log-likelihood pertaining to the state-dependent arrival process.
    \vspace{-3mm}


    \item[$\circ$] As mentioned before, due to the periodic arrival rate, one cannot assume the system to be in stationarity. 
    We remedy this complication by decomposing the sample path of the queueing system variables into distinct regeneration cycles, which are periods of time after which the system probabilistically resets. 
    This allows us to introduce, out of sequences of dependent observations, an i.i.d.\ structure, thus facilitating the use of existing results (laws of large numbers, central limit theorems) so as to establish performance guarantees of the resulting estimator. 
    For arrival rate functions and patience distributions satisfying mild regularity conditions and under other natural assumptions, we prove the strong consistency of the maximum likelihood estimator.
    We also prove asymptotic normality of the corresponding estimation error (scaled by $\sqrt{n}$), which is the fastest convergence rate possible in a parametric setting. \vspace{-1.3mm}

    \noindent
    In addition, we have developed a number of verifiable assumptions, based solely on the model primitives, under which the above asymptotic analysis is valid.
    \vspace{-3mm}

    \item[$\circ$] Our model allows arriving customers to base their joining decision on a delay announcement which is any reasonable function of the complete state of the system. 
    In particular, let $\mathcal{R}(t) \in \Omega$ encode the full information of the system at time $t$ (for example, the residual service times and order of service for all customers present in the system at time $t$), and let $\psi: \Omega \mapsto \mathbb{R}$ be a Borel function. 
    Then $\Delta(t)$ takes the form $\psi\big(\mathcal{R}(t)\big)$.
    Concrete examples of $\psi(\cdot)$ include functions which calculate the the exact or noisy estimates of the expected waiting time given the state of the system at time $t$.
    An important example is that in which $\Delta(t)$ is estimated based on the number of customers (say $q$) in the system at time $t$.
    If $q < s$, then there is at least one idle server and hence $\Delta(t) = 0$. 
    If $q \geq s$, then a (biased, but probably reasonable) delay announcement could be $\Delta(t) = (q-s+1)g/s$, where $g$ is the mean service time.
    We can compile this into a single delay announcement function of the form $\Delta(t) = \psi(\mathcal{R}(t)) = (L(t)-s+1)^{+}g/s$, where $L(t)$ is the number of customers in the system at time $t$.
    \vspace{-3mm}

    
    \item[$\circ$] We discuss several interesting insights obtained from numerical experiments, such as the identification of instances where estimation of $(\bs{\alpha}_0, \bs{\theta}_0)$ is intrinsically hard (cf.\ the discussion in Example~\ref{example: best and worst instances for estimation}), and how `near-unidentifiability' issues play a role in estimation (cf.\ the discussion in Example~\ref{example: identifiability issues}). 
    We also assess the estimation accuracy in situations when exact virtual waiting times are not known.
\end{itemize}

\subsection{Relevant Literature}









This paper remedies two key shortcomings of the predecessor paper \cite{inoue2023estimating}, where estimation of constant arrival rate and patience distribution was carried out in the context of service system modelled by an M/G/$s$+H queue. 
In the first place, we lift the assumption of the customer arrival rate being constant over time by allowing a periodic arrival rate. 
Second, \cite{inoue2023estimating} assumes a complete information setting, in that each arriving customer knows their exact prospective waiting time, which is often infeasible in practical situations. 
Importantly, our paper demonstrates that our estimation technique works even when customers have incomplete information, and estimate their prospective waiting time based on a partially observed system state.

An intensively studied topic at the intersection of operations research, applied probability and statistics, concerns the estimation of model primitives based on the (potentially partially observed) evolution of a stochastic process. 
For instance, \cite{woodroofe1985estimating} estimates distribution functions with truncated data. 
A specific subfield focuses on inference problems for queues, where the objective is to estimate the input process based on queue length or workload observations; see for instance the detailed survey  \cite{asanjarani2021survey}. 
Each study in this survey is characterized by its own specific underlying model, observation process, and estimation objectives. 
For instance, \cite{bingham1999non} studies an M/G/$\infty$ system, observing specific aspects of the evolution of the number of customers present, with the objective to estimate the arrival rate and service-time distribution (making the problem semi-parametric). 
In \cite{ross2007estimation} the  M/M/$s$ queue length is observed in the regime that $c$ is large $c$, and employs a diffusion approximation for estimating the arrival and service rates. 
In \cite{larson1990queue} one observes the so-called transactional data (i.e., times of service initiations and completions), and based on these one estimates various queue-related statistics. 
In \cite{ravner2019estimating} the workload process of a Lévy-driven queue is sampled at Poisson instants, providing observations that can be used in a method-of-moments estimator for the characteristic exponent of the driving Lévy process. 
Finally, \cite{pickands1997estimation} studies a discrete time M/G/$\infty$ queue with the objective of estimating the arrival rate and service distribution. It uses an entropy-motivated geometric approximation to employ algorithms used for estimation of hidden Markov models.
In general one aims to develop estimation techniques with certain concrete performance guarantees, in particular consistency or asymptotic normality.

A complication we come across in our setting with impatient customers, is that we wish to estimate the {\it total} arrival rate, i.e., including the arrival rate corresponding to the customers that balk due to impatience. 
In the queueing literature, impatience has been studied from several perspectives; in line with the focus of the present paper, we discuss work in which balking is due to arriving customers facing a high prospective waiting time. 
In \cite{baccelli1984single} stability condition is established for M$_t$/G/1+H systems with a periodic arrival rate. 
Then, \cite{liu2006explicit} derives the steady-state workload distribution, and hence implicitly also a stability condition, for the M/Ph/1 queue (in which service requirements are of phase type) with a constant patience level across all customers. 
In \cite{liu2008busy} the Laplace-Stieltjes transform of the busy period of M/Ph/1 queues is derived, as a limiting case of an associated fluid model. 
In \cite{boxma2010busy} the busy period distribution in an M/G/1+H system is identified, for various choices of patience distribution $H(\cdot)$. 
The problem of estimating customer (im)patience in service systems has also been given significant attention; see e.g., \cite{mandelbaum2013data}, while further references can be found in \cite{inoue2023estimating}.

Time-varying multi-server queues also have attracted some attention in the queueing literature.
Where \cite{whitt2016queues} provides an extensive bibliography on queues time-varying arrival rates, we highlight some key works. 
In \cite{massey1994analysis} the focus is on the distribution of the number of busy servers in an M$_t$/G/$s$/$0$ loss system, which has $s$ servers and no extra waiting space (i.e. there can be at most $s$ customers in the system at any time) and a non-homogeneous Poisson arrival process, using a modified offered-load approximation. 
Reference \cite{massey1994unstable} studies the asymptotic behaviour of arrival, departure and waiting times in unstable queues with non-stationary arrival rates. 
A branch of research focuses on estimation; in this context, \cite{massey1996estimating} studies the approximation of a non-homogeneous arrival rate by a piecewise linear counterpart. 
In \cite{kim2013estimating} a time-varying Little's law is used to develop estimators for waiting times in systems with time-varying arrival rates. 
The main objective of \cite{jennings1996server} lies in staffing: it presents an algorithm to select the number of servers as a function of time, so as to meet given performance targets. In \cite{chen2023can} it is shown that periodic sinusoidal arrival rates have a good empirical fit with arrival data from various service systems.


Poisson processes play a pivotal role in queueing theory, primarily with a constant rate, but the case of a time-dependent rate has received substantial attention as well.
In the statistics literature, the case of non-homogeneous, potentially periodic, Poisson processes has been covered by for instance \cite{helmers2003estimating,helmers2009estimating,helmers2003consistent,kuhl2000least,yoshida1990robust}; these papers differ in terms of the underlying model (which can be either parametric or non-parametric), the nature of the observation process, and the estimation technique relied upon. 
A textbook treatment of inference techniques for broad classes of Poisson processes can be found in~\cite{kutoyants2012statistical}.

We conclude this account of the existing literature by mentioning a number of works that have been used in our consistency and asymptotic normality proofs. 
In the first place, we rely on results that provide insight into our M$_t$/G/$s$+H queue; in particular, \cite{heyman1984asymptotic} shows how the regeneration cycles can be dealt with in time-varying queues. 
In addition, we use results from \cite{andrews1992generic} in our strong consistency proof, and results from \cite{benaim2022markov} on establishing the recurrence of specific sets (by which we can derive properties which play a key role in our proofs).

\subsection{Paper Organization}

In Section \ref{section: Model and preliminaries}, we provide a model description and present preliminaries. 
Section \ref{section: Parametric estimation procedure} details the MLE-based estimator, along with theorems regarding their strong consistency and asymptotic normality.
From the methodological perspective, Section \ref{section: Constructing i.i.d regeneration cycles} can be seen as the heart of this paper: we describe our procedure to decompose the sample path into distinct regeneration cycles, which is a central element in the proofs in Section \ref{section: strong consistency}. Section \ref{section: strong consistency} also contains an exhaustive list of model assumptions under which the theorems and proofs hold.
A reader who is not interested in the asymptotic analysis of the estimator can skip Sections \ref{section: Constructing i.i.d regeneration cycles}, \ref{section: strong consistency}, and go directly to Section \ref{section: joining decisions based on incomplete information}, where we establish that the theorems on strong consistency and asymptotic normality of the estimators that are stated in Section \ref{section: Parametric estimation procedure} hold also for a general delay announcement, under an additional assumption.
Section \ref{section: numerical experiments} contains a series of simulation experiments, confirming our estimation procedure's efficacy.
In Section \ref{section: conclusion}, we discuss our findings, and provide directions for future research.
In Appendix \ref{section: analysis of Mt/G/s+H system}, we prove results related to a regeneration structure underlying the M$_t$/G/$s$+H system, which are required in our proofs. 
Appendix \ref{section: appendix} contains proofs some lemmas that are required in the proofs of consistency and asymptotic normality. 

\section{MODEL AND PRELIMINARIES} \label{section: Model and preliminaries}


We model our service system via a queue with $s\in{\mathbb N}$ servers. 
Potential customers arrive according to a Poisson process with a time-dependent arrival rate $\lambda_{\bs{\alpha}}(t)$, parametrized by the vector $\bs{\alpha} \equiv (\alpha_1, \alpha_2, \ldots, \alpha_k ) \in \mathbb{R}^k$ for some known $k\in{\mathbb N}$.
We consider the setting in which the arrival rate is periodic, where the period is known; by normalizing time we can assume without loss of generality the period is one, so that $\lambda_{\bs{\alpha}}(t+1) = \lambda_{\bs{\alpha}}(t)$ for all $t \geq 0$. 
Each customer is equipped with (1)~a service requirement that has cumulative distribution function $G(\cdot)$, and (2)~a patience level that has cumulative distribution function $H_{\bs{\theta}}(\cdot)$, where $\bs{\theta} \equiv (\theta_1, \theta_2, \ldots, \theta_p) \in \mathbb{R}^p$ (for some known $p\in{\mathbb N}$) represent the parameters of the patience distribution. 
The service requirements $\{B_i\}_{i\in{\mathbb N}}$ and patience levels $\{Y_i\}_{i\in{\mathbb N}}$ are independent sequences of i.i.d.\ random variables, and the customer interarrival times $\{T_i\}_{i\in{\mathbb N}}$ are independent of $\{B_i\}_{i\in{\mathbb N}}$ and $\{Y_i\}_{i\in{\mathbb N}}$. 
Note that the interarrival times are neither independent of each other, nor identically distributed because of the arrival rate varying with time.
At each time $t$, the service operator announces $\Delta(t)$, which is an estimate of the delay before service which an incoming customer at time $t$ can expect.
Based on this, they join the system if $Y \geq \Delta(t)$, otherwise they balk. 
We now discuss this delay announcement in detail.

\subsection{Delay Announcement}
The delay announcement at time $t$ is calculated based on the complete state of the system $\mathcal{R}(t)$ at time $t$. Let $\Delta(t) = \psi\big(\mathcal{R}(t)\big)$, where $\psi: \tilde{\Omega} \mapsto \mathbb{R}_{+}$ is a Borel-measurable function. 
By $\mathcal{R}(t)$, we mean a vector that represents the complete state of the system at time $t$. Define
\[\mathcal{R}(t) = (r_1, r_2, \ldots, r_k) \in \tilde{\Omega} = \bigcup_{k=0}^{\infty} \mathbb{R}_{+}^{k},\] where $k \in{\mathbb N}_0$ is the number of customers in the system at time $t$, and $r_1, \ldots, r_k$ represent their residual service times. If $k=0$, then there are obviously no residual service times. 
If $k \in\{1,\ldots,  s\}$, then all customers in the system are in service. If $k > s$, then customers corresponding to service times $r_1, \ldots, r_s$ are in service, and customers corresponding to $r_{s+1}, r_{s+2}, \ldots,r_k$ are ordered according to their service priority (in the first-come-first-serve setting this means in order of arrival) and their residual times equal their original service times. 

Under this construction, $\mathcal{R}(t)$ represents the complete state of the system at time $t$, available to the service operator at $t$, based on which, the delay announcement $\Delta(t)$ is constructed. 
One could formalize this by working with a function $\psi$ that maps the observed state of the system to a metric, which serve as a `delay proxy'. 
For a customer arriving at time $t$, the following list suggests some of the many possible information metrics based on which an arriving customer can decide whether or not to join the system: with $L(t) = \big\vert \mathcal{R}(t)\big\vert$ the number of customers present at time $t$, 
\begin{enumerate}
    \item $\psi\big(\mathcal{R}(t)\big)$ equals the virtual waiting time  $V(t)$.
    \item $\psi\big(\mathcal{R}(t)\big)$ equals some noisy estimate $\tilde V(t)$ of $V(t)$. For example, one reasonable estimate of the waiting time could be the delay proxy
    \begin{align*}
        \psi\big(\mathcal{R}(t)\big) &= \frac{1}{s}\big(L(t) -s + 1\big)^{+}\,{\mathbb E} B.
    \end{align*}
    \item $\psi\big(\mathcal{R}(t)\big)$ equals number of service completions that the customer has to wait for:
    \begin{align*}
        \psi\big(\mathcal{R}(t)\big) = \big(L(t)-s + 1\big)^{+}.
    \end{align*}
\end{enumerate}
From a theoretical point of view, the second and third information metric are the same, as they differ by a constant factor.
From a practical point of view, however, they reflect different messages, namely an expected delay message or simply the number customers in the queue.
It is noted that in applications one can encounter both information metrics.

\begin{remark}
The assumption that $\Delta(t)$ is a function of $\mathcal{R}(t)$ can be weakened further. Define $\sigma(t):= \big\{\mathcal{R}(u): u \in [T_s,t)\big\}$, where $T_s$ is the start time of the ongoing regeneration cycle, i.e. the cycle to which time $t$ belongs (We discuss the concept of a regeneration cycle for an M$_t$/G/s+H system later in this paper).
In other words, $\sigma(t)$ stores all information about the system since the start of the ongoing regeneration cycle.
We can then also allow $\Delta(t)$ to be a function of $\sigma(t)$.
A commonly used delay proxy for an incoming customer at time $t$, see \cite{ibrahim2009real,ibrahim2011wait}, is the largest of the times spent waiting so far by the customers present in the system at time $t$. 
To formally define this delay proxy, let $\{\mathcal{W}(t): t \geq 0\}$ store the amount of time that the customers with residual service times corresponding to $\mathcal{R}(t)$ have waited in the system until time $t$, i.e., \[\mathcal{W}(t) = (w_1, \cdots, w_k),\] is the time spent waiting so far by each of the $k$ customers present in the system at time $t$. 
Then, observe that $\mathcal{W}(t) \subseteq \sigma(t)$. We therefore define another delay proxy as follows:
\begin{enumerate}
    \item[4.] $\psi(\sigma(t))$ is the maximum amount of time spent waiting amongst all customers present in the system at time $t$, i.e.
    \begin{align*}
        \psi(\sigma(t)) = \max_{i=1,\cdots,L(t)} w_i.
    \end{align*}
\end{enumerate}
Observe that this delay proxy increases linearly between departure times, and can only decrease at departure times. In the rest of the paper however, for ease of understanding, we continue with our original assumption $\Delta(t) = \psi\big(\mathcal{R}(t)\big)$.
\end{remark}

We now proceed to explain a sample delay mechanism and its implications. 
As an example, we assume that $\Delta(t) = V(t)$, the virtual waiting time at time $t$, which is the exact waiting time experienced by a hypothetical customer who joins the system immediately after time $t$. 
The quantity $V(t)$ can alternatively be understood as the time it would take before at least one server becomes idle if no customers join the system after time $t$. 
That is, for any $t \geq 0$, the virtual waiting time can be defined by
\begin{align}\label{defn:VWT}
    V(t) = \inf \big\{u \geq 0: L(t) - \big(N_D(t+u) - N_D(t) \big) \leq s-1 \big\}
\end{align}
where $L(t)$ is the number of customers in the system (including those in service) at time $t$, and $N_{D}(t)$ the total number of customers that have finished their service by time $t$. 
The virtual waiting time process $\{V(t)\}_{t\in{\mathbb R}}$ is a càdlàg (`right continuous with left limits') process. Customer $i\in {\mathbb N}$ joins the system if their patience is at least equally large as their virtual delay:
\begin{align}
    Y_i \geq V\Biggl(\sum_{j=1}^{i} T_j -\Biggr).
\end{align}
Figure \ref{fig:V(t)} provides an illustration of the virtual waiting time process of an M/G/2+H queueing system with (for ease) a constant arrival rate $\lambda = 10$, service requirements with unit mean, and patience levels that are deterministically equal to 1. 
The green (red) dots represent the non-balking (balking, respectively) customers. 
In the graph the virtual waiting time just prior to the arrival of non-balking (balking) customers is smaller (larger, respectively) than 1.
\begin{figure}
    \centering
    \vspace{-1cm}
    \includegraphics[width = 17cm, height = 9cm]{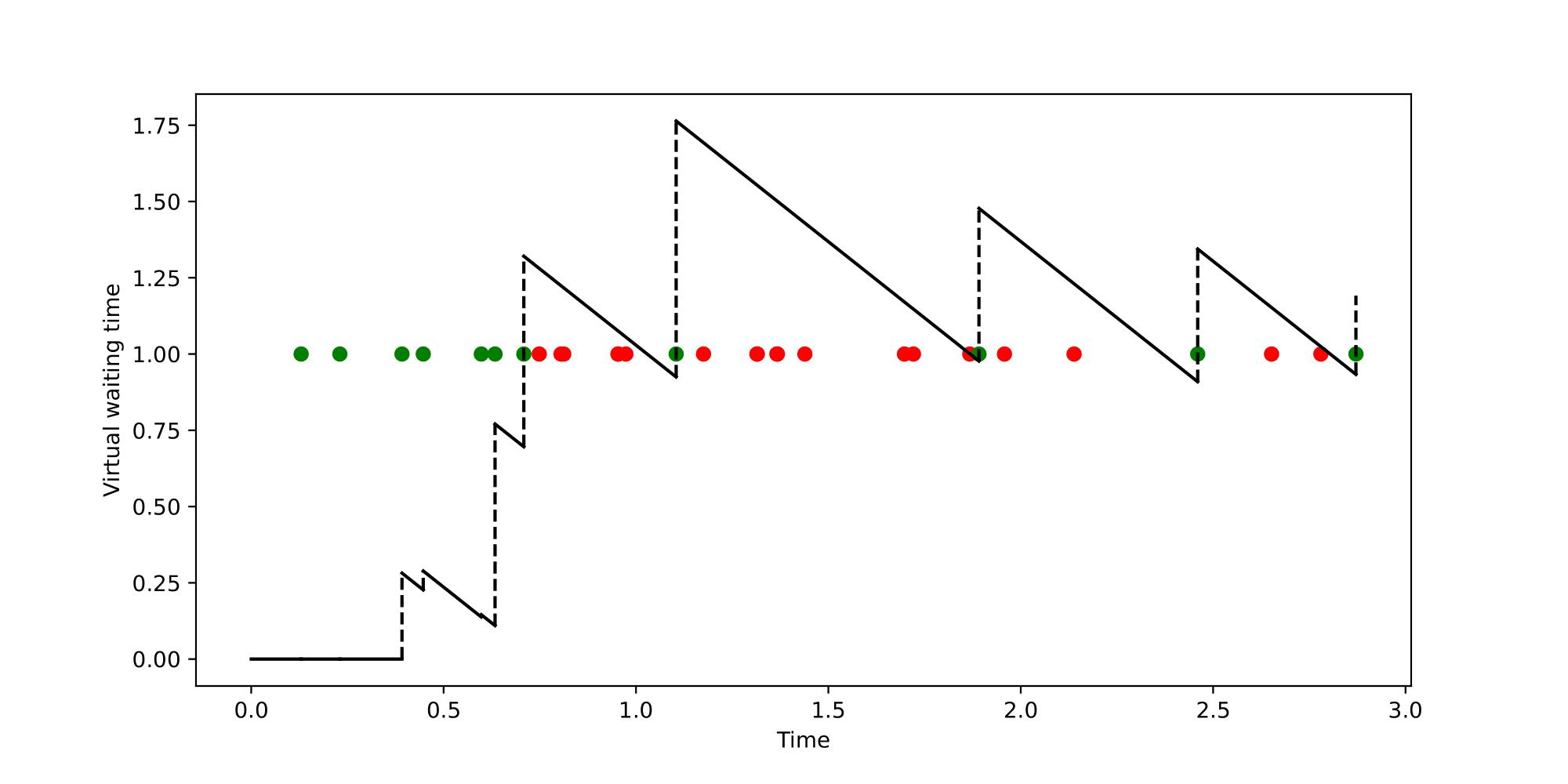}
    \vspace{-0.5cm}
    \caption{An example of virtual waiting time process $(s=2)$ and joining decisions of customers}
    \label{fig:V(t)}
\end{figure}

We now discuss the inference problem addressed in this paper.
The crucial feature is that only non-balking customers can be observed, i.e., the customers corresponding to the green dots in Figure \ref{fig:V(t)}.
Our estimation procedure relies on the state dependent distribution of the {\it effective} interarrival times, i.e., the times between two subsequent non-balking arrivals. 
In the sequel, we denote the sequence of effective interarrival times by $\{A_i\}_{i\in{\mathbb N}}$, and the corresponding effective arrival times by $\{\tilde{A}_i\}_{i\in{\mathbb N}}$, i.e., $\tilde{A}_i = \sum_{j=1}^{i} A_j$. 
Each non-balking arrival causes a non-negative increase in the virtual waiting time, which we can be seen as an `upward jump' in the virtual waiting time.
We denote the sequence of such upward jumps by $\{X_i\}_{i\in {\mathbb N}}$. 
We write
\begin{align*}
    X_i := V(\tilde{A}_i) - V(\tilde{A}_i-) = V(\tilde{A}_i) - W_{i},
\end{align*}
i.e., $W_i$ denotes the virtual waiting time immediately before the $i$-th effective arrival; realize that the waiting time corresponding to the $i$-th effectively arriving customer coincides with the virtual waiting time immediately before this arrival.

In this paper we aim to develop a procedure for estimating the parameter vector $\bs{\mu}_0 \equiv \big({\bs{\alpha}_0}, {\bs{\theta}_0}\big) \in \Theta\subseteq \mathbb{R}^{k+p}$ corresponding the arrival rate function and the patience distribution. 
Here it is assumed that some system data is observed over time. The exact observables depend on the specific choice of $\Delta(t)$. For instance, when $\Delta(t) = V(t)$, we observe the full queueing process over time, i.e., all effective arrivals and service times. However, when $\Delta(t) = \psi(L(t))$, i.e. the delay is some function of the number of customers in the system, then we observe only the process $\big\{L(t)\big\}_{t \geq 0}$.
The objective is to use this information to somehow learn the parameters of the true arrival rate and the patience-level distribution. More concretely, we wish to devise statistical procedures for estimating ${\bs{\alpha}_0}$ and ${\bs{\theta}_0}$ (both corresponding to non-observable quantities) with provable performance guarantees.

\section{PARAMETRIC ESTIMATION PROCEDURE} \label{section: Parametric estimation procedure}
In the previous section, we specified and explained all the necessary variables for the demonstration of the estimation procedure in this section.
Our estimation procedure relies on constructing a conditional likelihood of the state-dependent arrival process. 
Finally, we make claims regarding asymptotic performance of the estimator.

Let the random variable\[ \left(N\big(\tilde{A}_{i-1},\tilde{A}_{i-1}+t \big] \ \big\vert \:\Delta\big(\tilde{A}_{i-1}\big) = \delta_{i-1}\right)\] denote the the number of effective (i.e., non-balking) arrivals in time interval $\big(\tilde{A}_{i-1},\tilde{A}_{i-1}+t \big]$, conditioned on the delay announcement immediately after the $(i-1)$-st effective arrival being $\delta_{i-1}$. Define the tail probability $\tilde{H}_{\bs{\theta}}(x) = \prob(Y \geq x)$. Then, relying on standard properties of time-dependent Poisson processes,
\begin{align*}
   \left(N\big(\tilde{A}_{i-1},\tilde{A}_{i-1}+t \big] \ \big\vert \:\Delta\big(\tilde{A}_{i-1}\big) = \delta_{i-1}\right) \sim \text{Poisson}\bigg(\int_{0}^{t} \lambda_{\bs{\alpha}}\big(u+\tilde{A}_{i-1}\big) \tilde{H}_{\bs{\theta}}\big(\Delta\big(\tilde{A}_{i-1}+u\big)\big) \ \mathrm{d}u \bigg),
\end{align*}
entailing that
\begin{align}
    \prob\Bigl(A_i >t \ \Big\vert \ \Delta\big(\tilde{A}_{i-1}\big)  = \delta_{i-1}\Bigr) &= \prob \bigg(N\big(\tilde{A}_{i-1},\tilde{A}_{i-1}+t\big] = 0 \ \Big\vert \ \Delta\big(\tilde{A}_{i-1}\big) = \delta_{i-1} \bigg)\notag
    \\
    &= \exp\biggl(-\int_{0}^{t} \lambda_{\bs{\alpha}}\big(u+\tilde{A}_{i-1}\big)\tilde{H}_{\bs{\theta}}\big(\Delta\big(\tilde{A}_{i-1}+u\big)\big) \ \mathrm{d}u \biggr).
\end{align}
The density of the $i$-th effective interarrival time, conditioned on the delay announcement just after the $(i-1)$-st effective arrival, thus reads
\begin{align}\label{eqn:density effective interarrival times}
    f_{A_i \,\vert\, \delta_{i-1}}(t) &= - \frac{\rm d}{{\rm d}t} \prob\big(A_i >t \ \big\vert \ \Delta\big(\tilde{A}_{i-1}\big) = \delta_{i-1} \big) \notag
    \\
    &= \lambda_{\bs{\alpha}}\big(t+\tilde{A}_{i-1}\big) \tilde{H}_{\bs{\theta}}\big(\Delta\big(\tilde{A}_{i-1}+t\big)\big)\exp\biggl(-\int_{0}^{t} \lambda_{\bs{\alpha}}\big(u+\tilde{A}_{i-1}\big)\tilde{H}_{\bs{\theta}}\big(\Delta\big(\tilde{A}_{i-1}+u\big)\big) \, \mathrm{d}u \biggr).
\end{align}
We now arrive at the following conditional likelihood function:
\begin{align}
    L_n&\big(\bs{\alpha}, \bs{\theta};{\bs A} \,\big\vert \,{\bs \Delta}\big) = \prod_{i=1}^{n} f_{A_i \,\vert\, \delta_{i-1}}(A_i) \notag
    \\
    = &\prod_{i=1}^{n} 
    \lambda_{\bs{\alpha}}\big(\tilde{A}_i\big) \tilde{H}_{\bs{\theta}}\big(\Delta\bigl(\tilde{A}_i^{-}\bigr)\big)\exp\biggl(-\int_{0}^{A_i} \lambda_{\bs{\alpha}}\big(u+\tilde{A}_{i-1}\big)\tilde{H}_{\bs{\theta}}\big(\Delta\big(\tilde{A}_{i-1}+u\big)\big) \, \mathrm{d}u \biggr),
\end{align}
so that the log-likelihood of the effective arrival process is given by
\begin{align}
    \ell_n \equiv \ell_n\big(\bs{\alpha}, \bs{\theta};{\bs A}\, \big\vert\, {\bs \Delta}\big) &= \log L_n\big(\bs{\alpha}, \bs{\theta};{\bs A}\, \big\vert \,{\bs \Delta}\big) \notag
    \\
    &= \sum_{i=1}^{n} \biggl(\log\lambda_{\bs{\alpha}}\big(\tilde{A}_i\big) + \log\tilde{H}_{\bs{\theta}}\big(\Delta\bigl(\tilde{A}_i^{-}\bigr)\big)\: -\notag \\
    &\hspace{0.42cm}\int_{0}^{A_i} \lambda_{\bs{\alpha}}\big(u+\tilde{A}_{i-1}\big)\tilde{H}_{\bs{\theta}}\big(\Delta\big(\tilde{A}_{i-1}+u\big)\big) \, \mathrm{d}u \biggr). \label{eqn:log_likelihood:general}
\end{align}
For any $n\in{\mathbb N}$, the maximum likelihood estimator is defined by
\begin{align}
    \hat{\bs\mu}_n := \big(\hat{\bs\alpha}_n, \hat{\bs\theta}_n\big) = \argmax_{(\bs{\alpha}, \bs{\theta}) \in \Theta} \ \ell_n\big(\bs{\alpha}, \bs{\theta};{\bs A} \,\big\vert \,{\bs \Delta}\big).\label{eq:MLE}
\end{align}

We now explicitly show how \eqref{eqn:log_likelihood:general} would look like, for a couple of delay announcement functions
\begin{itemize}
    \item $\Delta(t) = V(t)$. In this case, as discussed in Section \ref{section: Model and preliminaries}, $\Delta\big(\tilde{A}_{i-1}\big) = \delta_{i-1} = W_{i-1}$, where $W_{i-1}$ is the waiting time of the $(i-1)$-st effective arrival. Furthermore, assuming there are no effective arrivals in $\big(\tilde{A}_{i-1}, \tilde{A}_{i-1}+u\big]$, we have $\Delta\big(\tilde{A}_{i-1} + u\big) = \max\big\{0, \delta_{i-1}+X_{i-1}-u\big\}$, and consequently, $\Delta\bigl(\tilde{A}_i^{-}\bigr) = \max\big\{0, \delta_{i-1}+X_{i-1}-A_i\big\}$. As a result, \eqref{eqn:log_likelihood:general} becomes
    \begin{align}
        \ell_n^{\text{exact}} = \sum_{i=1}^{n} \biggl(\log\lambda_{\bs{\alpha}}\big(\tilde{A}_i\big) + &\log\tilde{H}_{\bs{\theta}}\big(W_{i-1}+X_{i-1}-A_i \big) \notag
        \\
        &- \int_{0}^{A_i} \lambda_{\bs{\alpha}}\big(u+\tilde{A}_{i-1}\big)\tilde{H}_{\bs{\theta}}(W_{i-1}+X_{i-1}-u) \, \mathrm{d}u \biggr). \label{eqn:log_likelihood:exact}
    \end{align}

    \item $\Delta(t) = \psi\big(L(t)\big)$, i.e. the delay at time $t$ is a function of the number of customers in the system at time $t$. The arrival rate therefore depends only on time between 2 successive arrival or departure events. Then, $\Delta\big(\tilde{A}_{i-1}\big) = \delta_{i-1} = \psi(q_{i-1})$, where $q_{i-1}$ is the number of customers in the system just after the $(i-1)$-st effective arrival.
    Let $\big\{M(t): t \geq 0\big\}$ denote the departure process, i.e., $M\big(t_1,t_2\big]$ represents the number of departures in $(t_1,t_2]$. 
    For $i=1,\cdots,n$, let \[\big\{D(i,k)\big\}_{k=1,2,\cdots, M(\tilde{A}_{i}, \tilde{A}_{i+1}]}\] denote the departure times in $\big(\tilde{A}_{i}, \tilde{A}_{i+1}\big]$, where $M(\tilde{A}_{i}, \tilde{A}_{i+1}] = \max\{0, q_i - q_{i+1}+1\}$.
    We can then write \eqref{eqn:log_likelihood:exact} as
    \begin{align*}
        \ell^{\text{est}}_n\big(&\bs{\alpha}, \bs{\theta}\big) = \sum_{i=1}^{n} \Bigg\{\log\lambda\big(\bs{\alpha},\tilde{A}_i\big) + \log\tilde{H}_{\bs{\theta}}\big(\psi(q_i-1)\big) - \int_{\tilde{A}_{i-1}}^{D(i-1,1)} \lambda(\bs{\alpha},u)\tilde{H}_{\bs{\theta}}\big(\psi(q_{i-1})\big) \mathrm{d}u
        \\
        & - \int_{D(i-1,1)}^{D(i-1,2)} \lambda(\bs{\alpha}, u)\tilde{H}_{\bs{\theta}}\big(\psi(q_{i-1}-1)\big) \mathrm{d}u - \cdots - \int_{D(i-1,M(\tilde{A}_{i-1}, \tilde{A}_i])}^{\tilde{A}_i} \lambda(\bs{\alpha}, u)\tilde{H}_{\bs{\theta}}\big(\psi(q_i-1)\big)\mathrm{d}u \Bigg\}.
\end{align*}
\end{itemize}

We now state the two main theorems of this paper, which claim strong consistency and asymptotic normality of the estimator \eqref{eq:MLE}. For the sake of brevity, the following theorems and the construction leading up to their proofs are stated for $\Delta(t) = V(t)$.
The theorems rely on some mild and natural regularity assumptions on the arrival rate and patience distribution.
The extensive list of assumptions has been provided in Section \ref{section: strong consistency}.
In the asymptotic variance term we use an elaborate construction, to be described in Section \ref{section: Constructing i.i.d regeneration cycles}, featuring regeneration cycles. 
In particular, we will point out in detail what the objects $\bs{Z}_1$, $q(\bs Z_1, \bs\mu_0)$, and ${\mathbb E}C_1$ mean in Section \ref{section: Constructing i.i.d regeneration cycles}.

\begin{theorem}\label{theorem:consistency} 
    If Assumptions {\em (A1)--(A6)} are satisfied, then, as $n \rightarrow \infty$,
    \begin{align}
        \hat{\bs{\mu}}_n \overset{\text{\rm a.s.}}{\longrightarrow} {\bs{\mu}_0}. \label{consistency}
    \end{align}
\end{theorem}
\begin{proof}
    Provided in Section \ref{section: strong consistency}.
\end{proof}

\begin{theorem}\label{theorem:asymptotic_normality}
    If Assumptions {\em (A1)--(A10)} are satisfied, then, as $n \rightarrow \infty$,
    \begin{align}
        \sqrt{n}\big(\hat{\bs{\mu}}_n - {\bs{\mu}_0}\big) \overset{\text{\rm a.s.}}{\longrightarrow} \mathcal{N}\big(0, \,{\mathbb E} C_1 \,I({\bs{\mu}_0})^{-1} \big), \label{asymptotic_normality}
    \end{align}
    where \begin{equation}\label{eq:defI}I\big({\bs{\mu}_0}\big) := -\exptn\Big[\nabla^2 q\big({\bs{Z}_1}, {\bs{\mu}_0}\big)\Big],\end{equation} and $C_1$ is the number of customers served during a regeneration cycle.
\end{theorem}
\begin{proof}
    Provided in Section \ref{section: strong consistency}.
\end{proof}

For a general delay announcement $\Delta(t) = \psi\big(\mathcal{R}(t)\big)$, Theorems \ref{theorem:consistency} and \ref{theorem:asymptotic_normality} are generalized in Proposition \ref{propn} under an additional assumption (A11). 
We direct the reader to Section \ref{section: joining decisions based on incomplete information} for more details.
Sections \ref{section: Constructing i.i.d regeneration cycles} and \ref{section: strong consistency} work towards proving the theorems stated above, and therefore work with $\Delta(t) = V(t)$.


\section{CONSTRUCTING i.i.d.\ REGENERATION CYCLES}\label{section: Constructing i.i.d regeneration cycles}


In Section \ref{section: Parametric estimation procedure}, we defined the estimator $\hat{\bs\mu}_n\equiv (\hat{\bs{\alpha}}_n,\hat{\bs{\theta}}_n)$, and stated its consistency and asymptotic normality properties. 
In the proofs of these statements, a crucial role is played by `regeneration cycles' into which our queueing process can be decomposed.
For regeneration cycle $j$, where $j\in {\mathbb N}$, we store all relevant observable information (effective arrival times, waiting times, jump sizes) pertaining to only that regeneration cycle into a random vector ${\bs Z}_j$.
The log-likelihood is then written as a sum of log-likelihoods of i.i.d.\ random variables $q(\bs{Z}_j, {\bs \mu})$, which are functions of the observation ${\bs Z}_j$ and the parameter ${\bs \mu}$.
In doing so, we create a set of i.i.d.\ objects, which we use in Section \ref{section: strong consistency} in order to prove Theorems \ref{theorem:consistency}--\ref{theorem:asymptotic_normality}.

We proceed by detailing the way how we decompose the sample path of our M$_t$/G/$s$+H system into regeneration cycles.
Recalling that $\lambda(\cdot)$ has period 1, we define $\zeta_j$ as the $j$-th integer-valued point in time that the system is empty. Formally, with $L(t)$ denoting the number of customers in the system at time $t$, and supposing $L(0)=0$, we define the start of the $j$-th regeneration cycle recursively via
\begin{align*}
    \zeta_j &:= \inf \big\{t > \zeta_{j-1}: t \in \mathbb{N}, \ L(t) = 0 \big\}.
\end{align*}
It is clear that the trajectories of the queueing process within distinct regeneration cycles are independent and identically distributed. (Clearly, subsequent time epochs $t\in{\mathbb R}$ such that $L(t)=0$ do not induce regeneration cycles, due to the variable arrival rate: one should have that both the system is empty and the arrival rate resets.)
We define the duration of the $j$-th cycle by $R_j = \zeta_j - \zeta_{j-1}$.
Let $C_j$ represent the number of customers that are serviced during cycle $j$ and let $\eta_{j-1}+1, \eta_{j-1}+2, \cdots, \eta_{j-1}+C_j$ denote their indices, with $\eta_0 = 0$ and $\eta_j - \eta_{j-1} = C_j$ (so that $\eta_j$ is the cumulative total number of customers that have arrived by the end of the $j$-th regeneration cycle). 
Finally, denote by $N_r$ the number of regeneration cycles in the observed sample path.

We now express the data we will be working with in terms of the regeneration structure we just defined. 
Let $\{\tilde{A}_{j,i}\}_{i=1,2,\cdots,C_j}$ represent the arrival times of customers arriving in the $j$-th regeneration cycle, relative to $\zeta_{j-1}$ (i.e., the starting point of the regeneration cycle), and let $\{A_{j,i} \}_{i=1,2,\cdots,C_j}$ denote the corresponding interarrival times. 
Define $\{W_{j,i}\}_{i=1,2,\cdots,C_j}$ and $\{X_{j,i}\}_{i=1,2,\cdots,C_j}$ as the waiting times and upward jumps in the virtual waiting time, respectively. 
More formally, for any $j \in \{1,2,\ldots,N_r\}$, and $i = \eta_{j-1}+1, \eta_{j-1}+2, \cdots, \eta_{j-1}+C_j$, we have
\begin{align}
    \tilde{A}_i &= \zeta_{j-1} + \tilde{A}_{j, i-\eta_{j-1}},\:\:\:\:\:\:
    A_i = A_{j,i-\eta_{j-1}},\:\:\:\:\:\:
    W_i= W_{j,i-\eta_{j-1}},\:\:\:\:\:\:
    X_i = X_{j,i-\eta_{j-1}}.\label{eq:conversion}
\end{align}
In essence, we make this transformation because we wish to specify the position of any customer relative to the start time of the regeneration cycle that they belong to. 

Figure \ref{fig:tikz:demo regeneration cycle} shows a hypothetical sample path of the number of customers as a function of time in our M$_t$/G/$s$+H system. 
We observe one full regeneration cycle, starting at time $0$ and ending at time $3$ (when the next regeneration cycle begins). 
Notice that the system becomes empty somewhere between time 2 and 3, but time $3$ is the first instant when the system is empty \textit{and} the arrival rate resets. 
Note that any regeneration cycle can consist of multiple busy periods of the underlying queueing process.

Theorem \ref{theorem:properties_of_system}, stated below, is crucial to the proofs of consistency and asymptotic normality of the estimators $\hat{\bs{\mu}}_n$.

\begin{figure}
    \centering
    \begin{tikzpicture}
        \draw[->] (0,0) -- (16,0) node[right] {$t$};
        \draw[->] (0,0) -- (0,5) node[above] {$L(t)$};
    
        \filldraw (4.1,0) circle[radius=1.5pt, color=red];
        \node[draw=none,below,color=red] at (4.1,0) {\smaller{1}};
    
        \filldraw (8.2,0) circle[radius=1.5pt, color=red];
        \node[draw=none,below,color=red] at (8.2,0) {\smaller{2}};
    
        \filldraw (12.3,0) circle[radius=1.5pt, color=red];
        \node[draw=none,below,color=red] at (12.2,0) {\smaller{3}};
    
        \filldraw (0,1) circle[radius=1.5pt, color=blue];
        \node[draw=none,left] at (0,1) {1};
    
        \filldraw (0,2) circle[radius=1.5pt, color=blue];
        \node[draw=none,left] at (0,2) {2};
    
        \filldraw (0,3) circle[radius=1.5pt, color=blue];
        \node[draw=none,left] at (0,3) {3};
    
        \filldraw (0,4) circle[radius=1.5pt, color=blue];
        \node[draw=none,left,] at (0,4) {4};
        
        \filldraw (0.5,0) circle[radius=1.5pt, color=blue];
        \node[draw=none,below,color=blue] at (0.5,0) {\small{$\tilde{A}_{1,1}$}};
    
        \filldraw (1.5,0) circle[radius=1.5pt, color=blue];
        \node[draw=none,below,color=blue] at (1.5,0) {\small{$\tilde{A}_{1,2}$}};
    
        \filldraw (2.5,0) circle[radius=1.5pt, color=blue];
        \node[draw=none,below,color=blue] at (2.5,0) {\small{$\tilde{A}_{1,3}$}};
        \node[draw=none,below,color=blue] at (3.5,0) {\small{$\cdots$}};
    
        \filldraw (11,0) circle[radius=1.5pt, color=blue];
        \node[draw=none, below,color=blue] at (11,0) {\small{$\tilde{A}_{1,C_1}$}};
    
        \filldraw (13,0) circle[radius=1.5pt, color=blue];
        \node[draw=none, below,color=blue] at (13,0) {\small{$\tilde{A}_{2,1}$}};
    
        \filldraw (14,0) circle[radius=1.5pt, color=blue];
        \node[draw=none, below,color=blue] at (13.8,0) {\small{$\tilde{A}_{2,2}$}};
    
        \filldraw (14.5,0) circle[radius=1.5pt, color=blue];
        \node[draw=none, below,color=blue] at (14.6,0) {\small{$\tilde{A}_{2,3}$}};
        
        \draw[line width=0.5mm, black, dashed]{} (0.5, 0) -- (0.5, 1);
        \draw[line width=0.5mm, black]{} (0.5, 1) -- (1.5, 1);
        \draw[line width=0.5mm, black, dashed]{} (1.5, 1) -- (1.5, 2);
        \draw[line width=0.5mm, black]{} (1.5, 2) -- (2.5, 2);
        \draw[line width=0.5mm, black, dashed]{} (2.5, 2) -- (2.5, 3);
        \draw[line width=0.5mm, black]{} (2.5, 3) -- (3.5, 3);
        \draw[line width=0.5mm, black, dashed]{} (3.5, 3) -- (3.5, 4);
        \draw[line width=0.5mm, black]{} (3.5,4) -- (4.8, 4);
        \draw[line width=0.5mm, black, dashed]{} (4.8, 4) -- (4.8, 3);
        \draw[line width=0.5mm, black]{} (4.8,3) -- (5,3);
        \draw[line width=0.5mm, black, dashed]{} (5,3) -- (5, 4);
        \draw[line width=0.5mm, black]{} (5,4) -- (6.5,4);
        \draw[line width=0.5mm, black, dashed]{} (6.5,4) -- (6.5, 3);
        \draw[line width=0.5mm, black]{} (6.5,3) -- (7.1,3);
        \draw[line width=0.5mm, black, dashed]{} (7.1,3) -- (7.1, 2);
        \draw[line width=0.5mm, black]{} (7.1,2) -- (7.5,2);
        \draw[line width=0.5mm, black, dashed]{} (7.5,2) -- (7.5, 1);
        \draw[line width=0.5mm, black]{} (7.5,1) -- (8.5,1);
        \draw[line width=0.5mm, black, dashed]{} (8.5,1) -- (8.5, 2);
        \draw[line width=0.5mm, black]{} (8.5,2) -- (9,2);
        \draw[line width=0.5mm, black, dashed]{} (9,2) -- (9, 1);
        \draw[line width=0.5mm, black]{} (9,1) -- (9.6,1);
        \draw[line width=0.5mm, black, dashed]{} (9.6, 1) -- (9.6, 0);
    
        \draw[line width=0.5mm, black, dashed]{} (10.5, 0) -- (10.5, 1);
        \draw[line width=0.5mm, black]{} (10.5,1) -- (11,1);    
        \draw[line width=0.5mm, black, dashed]{} (11, 1) -- (11, 2);
        \draw[line width=0.5mm, black]{} (11,2) -- (11.5,2);
        \draw[line width=0.5mm, black, dashed]{} (11.5, 2) -- (11.5, 1);
        \draw[line width=0.5mm, black]{} (11.5,1) -- (12,1);
        \draw[line width=0.5mm, black, dashed]{} (12,1) -- (12, 0);
        
        \draw[line width=0.5mm, black, dashed]{} (13,0) -- (13,1);
        \draw[line width=0.5mm, black]{} (13,1) -- (14,1);
        \draw[line width=0.5mm, black, dashed]{} (14,1) -- (14,2);
        \draw[line width=0.5mm, black]{} (14,2) -- (14.5,2);
        \draw[line width=0.5mm, black, dashed]{} (14.5,2) -- (14.5,3);
        \draw[line width=0.5mm, black]{} (14.5,3) -- (15,3);
        \draw[line width=0.5mm, black, dotted]{} (15,3) -- (15.5,3);
    \end{tikzpicture}
    \caption{Demonstration of regeneration cycle, in which $R_1=3.$}
    \label{fig:tikz:demo regeneration cycle}
\end{figure}
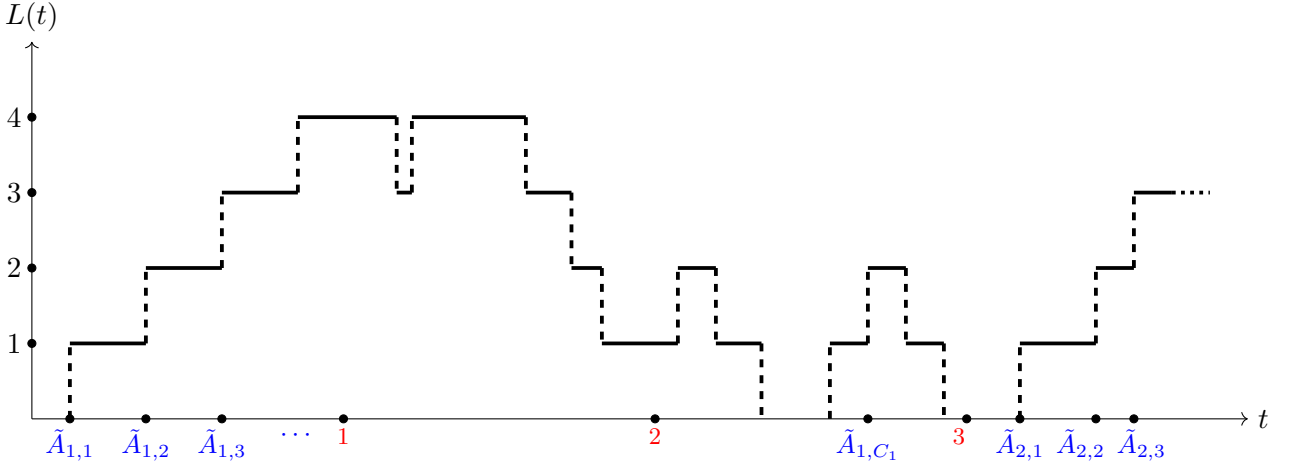

\begin{theorem} \label{theorem:properties_of_system}
    For the {\rm M$_t$/G/$s$+H} system, under Assumptions {\em (A2)} and {\em (A4)}, the following properties hold:
    \begin{enumerate}
        \item[(a)] The mean length of a regeneration cycle is finite, i.e., $\exptn R_1 < \infty$.
        \item[(b)] The expected number of non-balking customers in a regeneration cycle is finite, i.e., $\exptn C_1 < \infty$.
        \item[(c)] The expected sum of waiting times of non-balking customers in a regeneration cycle is finite, i.e.
        \begin{align*}
            \exptn\Bigg[\sum_{i=1}^{C_1} W_{1,i} \Bigg] < \infty.
        \end{align*}
    \end{enumerate}
\end{theorem}
\begin{proof}
    Deferred to Appendix \ref{section: analysis of Mt/G/s+H system}.
\end{proof}

In the rest of this section we present a decomposition of the log-likelihood $\ell_n^{\text{exact}}$ in terms of $N_r$ contributions $\big\{\ell^\circ_{j}\big\}_{j=1}^{N_r}$ by the individual regeneration cycles $1,2,\cdots,N_r$, in such a way that they are i.i.d., and that for each $j \in \{1,2,\cdots, N_r\}$, $\ell^{\circ}_{j}$ is a function only of the information pertaining to regeneration cycle $j$.
In this exposition we introduce a number of objects that feature in the assumptions we impose in the next section, under which Theorems~\ref{theorem:consistency}  and~\ref{theorem:asymptotic_normality} hold. 
The most intuitive first step towards obtaining $\ell^{\circ}_{j}$ would be to select the following portion of the summation from \eqref{eqn:log_likelihood:exact}:
\begin{align}\label{eqn:reg cycles:intuitive expression}
    \sum_{i= \eta_{j-1}+1}^{\eta_j} \bigg(\underbrace{\log\lambda_{\bs{\alpha}}\big(\tilde{A}_i\big)}_{\text{(I)}} + \underbrace{\log\tilde{H}_{\bs{\theta}}\big(W_{i-1}+X_{i-1}-A_i \big)}_{\text{(II)}} -\underbrace{\int_{0}^{A_i} \lambda_{\bs{\alpha}}\big(u+\tilde{A}_{i-1}\big)\tilde{H}_{\bs{\theta}}(W_{i-1}+X_{i-1}-u)}_{\text{(III)}} \ \mathrm{d}u \bigg),
\end{align}
i.e., the portion of the log-likelihood from \eqref{eqn:log_likelihood:exact} corresponding to the contribution by customers $\eta_{j-1}+1, \ldots, \eta_j=\eta_{j-1}+C_j$, all of which belong to cycle $j$. However, we will soon discover that in the above summation, we have included some portion which depends on cycle $j-1$, call it $T^{\text{exc}}_j$ (which we must exclude), and excluded some portion which depends on cycle $j$, call it $T^{\text{inc}}_j$ (which we must include). In order to understand what $T^{\text{exc}}_{j}$ and $T^{\text{inc}}_{j}$ are, we must express (I), (II), and (III) in terms of cycle $j$ related quantities, through the aid of \eqref{eq:conversion}.


(\textbf{I}): For $i = \eta_{j-1}+1, \eta_{j-1}+2, \cdots, \eta_j$, due to \eqref{eq:conversion}, we have, due to the periodic nature of $\lambda(\cdot)$ and bearing in mind that $\zeta_{j-1}$ is integer, 
\begin{align}
    \log\lambda_{\bs{\alpha}}\big(\tilde{A}_i\big) = \log\lambda_{\bs{\alpha}}\big(\zeta_{j-1} + \tilde{A}_{j, i-\eta_{j-1}}\big) = \log\lambda_{\bs{\alpha}}\big(\tilde{A}_{j, i-\eta_{j-1}}\big).
\end{align}
We conclude that all terms in (I) are dependent only on cycle $j$ variables, and must therefore be included in $\ell^{\circ}_j$.

\vspace{1mm}

(\textbf{II}):  For $i = \eta_{j-1}+2, \eta_{j-1}+3, \cdots, \eta_j$, due to \eqref{eq:conversion}, we obtain
\begin{align}\label{eq: reg cycle: term 2}
    \log\tilde{H}_{\bs{\theta}}\big(W_{i-1}+X_{i-1}-A_i \big) = \log\tilde{H}_{\bs{\theta}}\big(W_{j,i-1-\eta_{j-1}} + X_{j, i-1-\eta_{j-1}} - A_{j,i-\eta_{j-1}} \big).
\end{align}
In the case of $i = \eta_{j-1}+1$, on an apparent level, it feels as if the LHS of \eqref{eq: reg cycle: term 2} depends on $W_{i-1}$ and $X_{i-1}$, which are the waiting time and the upward jump in virtual waiting time caused by the last customer of the previous cycle, i.e. cycle $j-1$. 
Upon closer inspection however, we realize that $W_{\eta_{j-1}+1}$ is the waiting time of the first customer in the $j$-th regeneration cycle (who by definition finds the system empty). 
Therefore, by the Lindley recursion, $W_{\eta_{j-1}}+X_{\eta_{j-1}}-A_{\eta_{j-1}+1} \leq 0$ which implies $\tilde{H}_{\bs{\theta}}\big(W_{\eta_{j-1}}+X_{\eta_{j-1}}-A_{\eta_{j-1}+1}\big) = 1$ and hence 
\begin{align}\log\tilde{H}_{\bs{\theta}}\big(W_{\eta_{j-1}}+X_{\eta_{j-1}}-A_{\eta_{j-1}+1}\big) = 0.\end{align}
Again we conclude that all terms in (II) are dependent only on cycle $j$ variables, and must therefore be included in $\ell^{\circ}_j$.

\vspace{1mm}

(\textbf{III}): For $i = \eta_{j-1}+2, \eta_{j-1}+3, \cdots, \eta_j$, again by \eqref{eq:conversion} in combination with the periodicity of $\lambda_{\bs\alpha}(\cdot)$,
\begin{align}\label{eq:reg cycle:term 3}
    \int_{0}^{A_i} \lambda_{\bs{\alpha}}\big(u+\tilde{A}_{i-1}\big)&\tilde{H}_{\bs{\theta}}\big(W_{i-1}+X_{i-1}-u\big) \, \mathrm{d}u \notag
    \\
    &= \int_{0}^{A_{j,i-\eta_{j-1}}} \lambda_{\bs{\alpha}}\big(u+\tilde{A}_{j,i-1-\eta_{j-1}}\big)\tilde{H}_{\bs{\theta}}\big(W_{j,i-1-\eta_{j-1}}+X_{j,i-1-\eta_{j-1}}-u\big) \, \mathrm{d}u.
\end{align}
Once again we can conclude that all terms in (III) for customer indices $i = \eta_{j-1}+2, \cdots, \eta_j$ are dependent only on cycle $j$ variables, and must therefore be included in $\ell^{\circ}_j$. 
For $i = \eta_{j-1}+1$, observe that $A_{\eta_{j-1}+1}$ represents the time elapsed between the arrival of the last customer of cycle $j-1$ and the first customer of cycle $j$, and therefore it is not straightforward to replicate \eqref{eq:reg cycle:term 3} for $i = \eta_{j-1}+1$.
Since the integral on the LHS of \eqref{eq:reg cycle:term 3} also depends on some portion of cycle $j-1$, we split it into into two sub-integrals:
\begin{align}\notag
    \int_{0}^{\zeta_{j-1}-\tilde{A}_{\eta_{j-1}}} \lambda_{\alpha}\big(u+\tilde{A}_{i-1}\big)&\tilde{H}_{\bs{\theta}}\big(W_{i-1}+X_{i-1}-u\big) \, \mathrm{d}u \:+\\& \underbrace{\int_{\zeta_{j-1}-\tilde{A}_{\eta_{j-1}}}^{A_{\eta_{j-1}+1}} \lambda_{\alpha}\big(u+\tilde{A}_{i-1}\big)\tilde{H}_{\bs{\theta}}\big(W_{i-1}+X_{i-1}-u\big) \, \mathrm{d}u}_{\text{(IV)}}.\label{eq:reg cycle:term 3 split}
\end{align}
where the first sub-integral depends on regeneration cycle $j-1$, and must therefore be excluded by us. That is,
\begin{align*}
    T^{\text{exc}}_j = \int_{0}^{\zeta_{j-1}-\tilde{A}_{\eta_{j-1}}} \lambda_{\alpha}\big(u+\tilde{A}_{i-1}\big)\tilde{H}_{\bs{\theta}}\big(W_{i-1}+X_{i-1}-u\big) \, \mathrm{d}u.
\end{align*}
What this means is that we have a term $T^{\text{exc}}_j$ in \eqref{eqn:reg cycles:intuitive expression}, which is dependent on cycle $j-1$. But this also means that there is a similar term $T^{\text{exc}}_{j+1}$ which is dependent on cycle $j$, in the expression similar to \eqref{eqn:reg cycles:intuitive expression} but then corresponding to cycle $j+1$. We write that term down based on the first sub-integral from \eqref{eq:reg cycle:term 3 split} as follows:
\begin{align*}
    T^{\text{inc}}_j = \int_{0}^{\zeta_{j}-\tilde{A}_{\eta_{j}}} \lambda_{\alpha}\big(u+\tilde{A}_{\eta_j}\big)\tilde{H}_{\bs{\theta}}\big(W_{\eta_j}+X_{\eta_j}-u\big) \, \mathrm{d}u.
\end{align*}
Noting that $\zeta_{j}-\tilde{A}_{\eta_{j}} = (\zeta_j - \zeta_{j-1}) - (\tilde{A}_{\eta_{j}} - \zeta_{j-1} ) = R_j - \tilde{A}_{j, C_j}$, and using \eqref{eq:conversion}, this expression can be written as
\begin{align}
    T^{\text{inc}}_j = \int_{0}^{\zeta_j - \tilde{A}_{j,C_j}} \lambda_{\alpha}\big(u+\tilde{A}_{j,C_j}\big)\tilde{H}_{\bs{\theta}}\big(W_{j,C_j}+X_{j,C_j}-u\big) \, \mathrm{d}u.
\end{align}
(IV) depends on cycle $j$ only, as shown below. 
By \eqref{eq:conversion}, we have $A_{\eta_{j-1}+1} = \zeta_{j-1}-\tilde{A}_{\eta_{j-1}} + A_{j,1}$, so that (IV) reads as
\begin{align} \label{integral_2}
    \int_{\zeta_{j-1}-\tilde{A}_{\eta_{j-1}}}^{\zeta_{j-1}-\tilde{A}_{\eta_{j-1}} + A_{j,1}} \lambda_{\alpha}\big(u+\tilde{A}_{i-1}\big)\tilde{H}_{\bs{\theta}}\big(W_{i-1}+X_{i-1}-u\big) \, \mathrm{d}u. 
\end{align}
The first customer joining the system at least $\zeta_{j-1}-\tilde{A}_{\eta_{j-1}}$ time after $\tilde{A}_{\eta_{j-1}}$ finds the system empty, and thus has zero waiting time. 
Therefore, as a consequence of the Lindley recursion, $W_{\eta_{j-1}} + X_{\eta_{j-1}}-u \leq 0$ for any $u \geq \zeta_{j-1}-\tilde{A}_{\eta_{j-1}}$, so that $\tilde{H}_{\bs{\theta}}(W_{\eta_{j-1}}+X_{\eta_{j-1}}-u) = 1$. 
By performing the variable change $u \mapsto u - \big(\zeta_{j-1} -\tilde{A}_{\eta_{j-1}} \big)$, we conclude that \eqref{integral_2} has become
\begin{align}
    \int_{0}^{A_{j,1}} \lambda_{\bs{\alpha}}(u) \ \mathrm{d}u.
\end{align}


We conclude this section by defining an object that plays a crucial role in the assumptions that we impose for our main results to hold.
For every regeneration cycle $j \in \{1,2,\ldots,N_r\}$, define the vector
\begin{align}
    \bs{Z_j} = \Big(R_j, C_j, \big\{\tilde{A}_{j,i}\big\}_{i=1,2,\cdots,C_j}, \big\{W_{j,i}\big\}_{i=1,2,\cdots,C_j}, \big\{X_{j,i}\big\}_{i=1,2,\cdots,C_j} \Big)\label{data_vector}.
\end{align}
This means that $\bs{Z_j}$ is the collection of all data corresponding to the $j$-th regeneration cycle: cycle length, number of customers that joined the system during this cycle, their arrival times relative to the start time of the cycle, the waiting times they experience, and the upward jumps in virtual waiting times that they effect.
For a given parameter vector $\bs{\mu} \in \Theta$ and
\[\bs{z} = \big\{r, c, \{\tilde{a}_i\}_{i=1,2,\cdots,c}, \{w_i\}_{i=1,2,\cdots,c},\{x_i\}_{i=1,2,\cdots,c} \big\}\in{\mathscr Z}:=\bigcup_{m=1}^{\infty}{\mathbb N}\times {\mathbb N}\times {\mathbb R}^m_+ \times {\mathbb R}^m_+ \times {\mathbb R}^m_+,\] we define
\begin{align}
     q\big(\bs{z},\bs{\mu}\big) = &\sum_{i=1}^{c} \log\lambda_{\bs{\alpha}}\big(\tilde{a}_i\big) + \sum_{i=2}^{c} \log\tilde{H}_{\bs{\theta}}\big(w_{i-1}+x_{i-1}-(\tilde{a}_i - \tilde{a}_{i-1}) \big) - \int_{0}^{\tilde{a}_1} \lambda_{\bs{\alpha}}(u) \ \mathrm{d}u \notag
    \\
    & -\sum_{i=2}^{c} \int_{0}^{\tilde{a}_i - \tilde{a}_{i-1}} \lambda_{\bs{\alpha}}\big(u+\tilde{a}_{i-1}\big)\tilde{H}_{\bs{\theta}}\big(w_{i-1}+x_{i-1}-u\big) \ \mathrm{d}u  \notag
    \\
    & -\int_{0}^{r - \tilde{a}_c} \lambda_{\bs{\alpha}}\big(u+\tilde{a}_c\big)\tilde{H}_{\bs{\theta}}\big(w_c+x_c-u\big) \ \mathrm{d}u. \label{q(z,mu)}
\end{align}
Upon combining all above computations for the components (I), (II) and (III), (IV) and $T^{\text{inc}}_j$ we thus conclude that $\ell_j^\circ =q(\bs{Z}_j, \bs\mu)$ and, with the objects $q(\bs{Z}_1, \bs\mu), \cdots, q(\bs{Z}_{N_r}, \bs\mu)$ being i.i.d.,
\begin{align*}
    \ell_n^{\text{exact}} = \sum_{j=1}^{N_r} \ell_j^\circ=\sum_{j=1}^{N_r} q(\bs{Z}_j, \bs\mu).
\end{align*}

\section{ASYMPTOTIC PERFORMANCE OF ESTIMATOR}\label{section: strong consistency}

In Section \ref{section: Constructing i.i.d regeneration cycles}, we decomposed the sample path into distinct regeneration cycles. 
This allowed us to create, in our M$_t$/G/$s$+H setting with its own specific dynamics, i.i.d.\ objects. 
In this section, we extensively use this regenerative structure to prove Theorems \ref{theorem:consistency} and \ref{theorem:asymptotic_normality}. 
In the first subsection, we present the assumptions imposed, whereas in the second subsection is is argued that these assumptions are mild and natural. 
The last two subsections provide the proofs of consistency and asymptotic normality, respectively.

\subsection{Assumptions}
In Section \ref{section: Parametric estimation procedure}, we explained the estimation procedure, and stated our results regarding its asymptotic performance. 
The theorems hold under a series of assumptions, which we formally state in this subsection. 
We define $H_{\bs{\theta}}(\infty):=\lim_{x \rightarrow \infty}  H_{\bs{\theta}}(x).$

\begin{assumption*}
    Let $d: \mathbb{R}^{k+p} \times \mathbb{R}^{k+p} \mapsto \mathbb{R}_{+}$ be the $L^1$-metric. The following assumptions are imposed:
    \begin{itemize}
        \item[{\rm (A1)}] The parameter space $\Theta \equiv \Theta_{\bs\alpha}\times\Theta_{\bs\theta}\subset \mathbb{R}^{k+p}$ is a convex and compact set, such that the true parameter $\bs{\mu}_0$ lies in the interior of the set, i.e., $\bs{\mu}_0 \in \Theta^{\circ}$. 
        \item[{\rm (A2)}] There exist $\lambda_{\min}, \lambda_{\max} \in( 0,\infty)$ such that $\lambda_{\min} \leq \lambda_{\bs{\alpha}}(t) \leq \lambda_{\max}$ for all $\bs{\alpha} \in \Theta_{\alpha}$ and $t \in [0,1)$. 
        \item[{\rm (A3)}] There exists $\kappa > 0$ such that for all $t \in [0,1)$, $\lambda_{{\bs \cdot}}(t): \Theta_{\bs\alpha} \mapsto \mathbb{R}$ is $\kappa$-Lipschitz, i.e.,
        \begin{align*}
            \sup_{t \in [0,1)} \, \Big\vert \lambda_{\bs{\alpha}}(t) - \lambda_{\bs{\alpha}'}(t) \Big\vert \leq \kappa \ d\big(\bs{\alpha}, \bs{\alpha}' \big).
        \end{align*} 
        \item[{\rm (A4)}] $\lambda_{\max}\, \exptn(B)\,(1 - H_{\bs{\theta}}(\infty)) < 1$.
        \item[{\rm (A5)}] There exists $G_1 > 0$ such that for all $x \ge 0$, $\tilde{H}_{\bs{\theta}}(x): \Theta_{\bs\theta} \mapsto \mathbb{R}$ is $G_1$-Lipschitz, i.e.,
        \begin{align*}
            \sup_{x\ge 0}\,\Big\vert \tilde{H}_{\bs{\theta}} (x) - \tilde{H}_{\bs{\theta}'}(x)\Big\vert \leq G_1 \ d\big(\bs{\theta}, \bs{\theta'}\big).
        \end{align*}
        \item[{\rm (A6)}] There exists a linear function $G_2(\cdot):\mathbb{R}^{+} \mapsto \mathbb{R}^{+}$ such that, for all $x \ge 0$,
        \begin{align*}
            \Big\vert\log\tilde{H}_{\bs{\theta}} (x) -\log\tilde{H}_{\bs{\theta}'}(x)\Big\vert \leq G_2(x) \ d\big(\bs{\theta}, \bs{\theta'}\big).
        \end{align*}
        \item[{\rm (A7)}] There exists a linear function $G_3(\cdot): \mathbb{R}^{+} \mapsto \mathbb{R}^{+}$ such that, for all $x \ge 0$,
        \begin{align*}
            \Big\vert \log\tilde{H}_{\bs{\theta}}(x) \Big\vert \leq G_3(x).
        \end{align*}
        \item[{\rm (A8)}] For all $i \in \{1,2,\ldots, k\}$, $\bs{\alpha} \in \Theta_{\bs\alpha}$, and $t \in [0,1)$, the partial derivative $\frac{\partial}{\partial \alpha_i}\lambda_{\bs{\alpha}}(t)$ exists and is continuous.
        \item[{\rm (A9)}]
        For all $x\ge 0$, the gradient vector $\nabla_{\bs{\theta}} H_{\bs{\theta}}(x)$ and Hessian matrix $\nabla^2_{\bs{\theta}} H_{\bs{\theta}}(x)$ exist and are continuous with respect to $\bs{\theta}$.
        \item[{\rm (A10)}] The matrix 
        $\exptn\big[\nabla^2 q\big({\bs{Z}_1}, {\bs{\mu}_0}\big)\big]$ is invertible. 
    \end{itemize}
\end{assumption*}

\subsection{Discussion of Assumptions}\label{discass}
We provide some insights into the reason for imposing these assumptions and possible relaxations
\begin{itemize}
    \item[$\circ$] Convexity of the parameter space $\Theta$ is required because of the repeated application of the mean-value theorem in our proofs. 
    The other elements of (A1) are natural, and found commonly in statistical literature.
    \item[$\circ$] Through (A2) we impose the natural assumption that during times when the service system is operational, the arrival rate has a positive lower bound and a finite upper bound. 
    \item[$\circ$] Assumption (A3) essentially entails that upon slightly tweaking $\bs{\alpha}$, the variation in the arrival rate is small, which is crucial for the estimation of $\bs{\alpha}$. 
    Furthermore, this assumption also implies that $\log \lambda_{\bs\cdot}(t): \Theta_{\bs\alpha} \mapsto \mathbb{R}$ is $\kappa/\lambda_{\min}$-Lipschitz, because
    \begin{align*}
        \big\vert \log\lambda_{\bs{\alpha}}(t) - \log\lambda_{\bs{\alpha}'}(t) \big\vert &= \frac{\big\vert \log\lambda_{\bs{\alpha}}(t) - \log\lambda_{\bs{\alpha}'}(t) \big\vert}{\big\vert \lambda_{\bs{\alpha}}(t) - \lambda_{\bs{\alpha}'}(t)\big\vert} \times \big\vert \lambda_{\bs{\alpha}}(t) - \lambda_{\bs{\alpha}'}(t)\big\vert \leq \ \frac{\kappa}{\lambda_{\min}} d\big(\bs{\alpha}, \bs{\alpha}' \big).
        \end{align*}
    \item[$\circ$] We require (A4) to show the stability of the M$_t$/G/s+H queueing system. 
    Informally, the quantity $\lambda_{\max}\, \exptn(B)\,(1 - H_{\bs{\theta}}(\infty))$ can be interpreted as the amount of work brought to the system if customers would join regardless of the system's congestion level.
    This makes us believe that this assumption can be further relaxed to just requiring that 
    $\lambda_{\max} \exptn(B) (1 - H_{\bs{\theta}}(\infty)) < s,$ but this has turned out remarkably hard to prove.    
    From a practical standpoint, however, this is of limited relevance, because for virtually any relevant patience distribution one has that $H_{\bs\theta}(\infty) = 1$.
    \item[$\circ$] Assumption (A5) is of the same spirit as (A3), and is justified along the same lines. 
    \item[$\circ$] We require (A6) and (A7) while proving strong consistency of the maximum likelihood estimates. We have verified that these specific assumptions hold for standard distributions.
    \item[$\circ$] (A8) and (A9) are regularity assumptions that are common in statistical literature.
    \item[$\circ$] (A10) is also a standard assumption in the statistical literature and essentially requires that the gradient of the log-likelihood is not degenerate.
\end{itemize}

\subsection{Strong Consistency}
In Section \ref{section: Constructing i.i.d regeneration cycles}, we decomposed the observed sample path of the system into distinct regeneration cycles. 
We know that these cycles are i.i.d., and in the remainder of this section, we shall intensively make use of that property to establish asymptotic consistency of the estimators. 
At the end of Section \ref{section: Constructing i.i.d regeneration cycles}, we constructed random vectors which encompass all information pertaining to a regeneration cycle, and wrote the log-likelihood as a sum of i.i.d.\ functions of these random vectors. 
\begin{lemma} \label{D_Andrews_lemma_1}
     Let $d$ be a metric on $\Theta$. Let $\bs{\mu}, \bs{\mu'} \in \Theta$ where $\bs{\mu} = \big(\bs{\alpha}, \bs{\theta} \big)$ and $\bs{\mu'} = \big(\bs{\alpha}', \bs{\theta'} \big)$. Then under Assumptions {\em(A2), (A3), (A5), (A6)}, there exists a measurable function $B(\cdot): \mathscr{Z} \mapsto \mathbb{R}$ and a non-random function $h(\cdot): \mathbb{R} \mapsto \mathbb{R}$ such that for all~$j,$ \[\big\vert \,q\big(\bs{Z_j}, \bs{\mu'}\big) - q\big(\bs{Z_j}, \bs{\mu}\big) \,\big\vert \leq B\big(\bs{Z_j}\big)\,h\big(d\big(\bs{\mu}, \bs{\mu'}\big)\big),\] where $\exptn B\big(\bs{Z_j}\big) < \infty$ and $\lim_{x \rightarrow 0} h(x) = 0$.
\end{lemma}
\begin{proof}
    Deferred to Appendix \ref{section: appendix}.
\end{proof}

\begin{corollary}[Corollary of Lemma \ref{D_Andrews_lemma_1}] The following two equations hold:
    \begin{align}
    \sup_{N_r \in{\mathbb N}} \ \frac{1}{N_r} \sum_{j=1}^{N_r} \exptn B\big(\bs{Z_j}\big) &< \infty, \label{D_Andrews_condition_1}
    \\
    \frac{1}{N_r} \sum_{j=1}^{N_r} \Big(B\big(\bs{Z_j}\big) - \exptn B\big(\bs{Z_j}\big)\Big) &\overset{\text{a.s}}{\longrightarrow} 0\:\:\:\mbox{as}\:\:N_r\to\infty\,. \  \label{D_Andrews_condition_2}
    \end{align}
\end{corollary}
\begin{proof}
    The random variables $B\big(\bs{Z_j}\big)$ for $j = 1,2,\cdots, N_r$ are i.i.d., and by Lemma \ref{D_Andrews_lemma_1}, $\exptn B(Z_1) < \infty$, so that  (\ref{D_Andrews_condition_1}) follows trivially. For (\ref{D_Andrews_condition_2}), we rely on the strong law of large numbers.
\end{proof}

We now possess all the requisite machinery to prove Theorem \ref{theorem:consistency}.
\begin{proof} [Proof of Theorem \ref{theorem:consistency}] 
First observe that as $n \rightarrow \infty,$ we have that $n/N_r \overset{\text{a.s.}}{\longrightarrow} \exptn\,C_1$, which is finite due to part (b) of Theorem \ref{theorem:properties_of_system}, so that 
$N_r \rightarrow \infty$ is equivalent to $n \rightarrow \infty$.
Using  \cite[Theorem 3(b)]{andrews1992generic}, in combination with
Lemma \ref{D_Andrews_lemma_1} and Equations \eqref{D_Andrews_condition_1} and \eqref{D_Andrews_condition_2}, conclude that 
\begin{align}
    \sup_{\bs{\mu} \in \Theta} \ \Bigg|\frac{1}{N_r} \sum_{j=1}^{N_r} q\big(\bs{Z_j}, \bs{\mu}\big) - \exptn q\big({\bs{Z}_1}, \bs{\mu}\big) \Bigg| \overset{\text{a.s}}{\longrightarrow} 0, \label{U-SLLN}
\end{align}
which is equivalent to, as $n \rightarrow \infty$, 
\begin{align*}
    \sup_{\bs{\mu} \in \Theta} \ \Bigg\vert \frac{1}{n} \ell_n\big(\bs{\alpha}, \bs{\theta};{\bs A},{\bs W} \,\big\vert \,{\bs X}\big)\cdot \frac{n}{N_r} - \exptn q\big({\bs{Z}_1}, \bs{\mu}\big) \Bigg\vert \overset{\text{a.s.}}{\longrightarrow} 0.
\end{align*}
The remainder of the proof is similar to that of \cite[Theorem 1]{inoue2023estimating}.
Recalling that $n/N_r \overset{\text{a.s.}}{\longrightarrow} \exptn\, C_1$ as $n \rightarrow \infty,$ we have that
\begin{align*}
    \sup_{\bs{\mu}\in\Theta} \left\vert \frac{1}{n}\ell_n\big(\bs{\alpha}, \bs{\theta};{\bs A},{\bs W} \,\big\vert \,{\bs X}\big)-\ell(\bs{\mu}) \right\vert \overset{\text{a.s.}}{\longrightarrow} 0 ,
\end{align*}
where $\ell(\bs{\mu})=\exptn q\big({\bs{Z}_1}, \bs{\mu}\big)/\exptn\,C_1$ is a non-random function which is maximized at the true parameter ${\bs{\mu}_0}$.
The density of the effective interarrival times given in \eqref{eqn:density effective interarrival times}, and thereby the log-likelihood \eqref{eqn:log_likelihood:exact}, are uniquely determined by $\lambda_{\bs\alpha}(\cdot)$ and $H_{\bs\theta}(\cdot)$, which in turn are uniquely determined by $\bs\alpha$ and $\bs\theta$ respectively. 
In particular, this means that there exist neither $\bs{\alpha}_1, \bs{\alpha}_2 \in \Theta_{\alpha}$ such that $\lambda_{\bs{\alpha}_1}(\cdot) = \lambda_{\bs{\alpha}_2}(\cdot)$ almost everywhere, nor $\bs{\theta}_1, \bs{\theta}_2 \in \Theta_{\theta}$ such that $H_{\bs{\theta}_1}(\cdot) = H_{\bs{\theta}_2}(\cdot)$ almost everywhere. 
We therefore conclude that the model is identifiable in the Kullback-Leibler sense, i.e.,
\begin{align*}
    \ell(\bs\mu) - \ell(\bs{\mu}_0) < 0, \ \forall \bs{\mu} \neq \bs{\mu}_0.
\end{align*}
Hence we can conclude that, $\bs{\mu_n} \overset{\text{a.s.}}{\longrightarrow} {\bs{\mu}_0}$.
\end{proof}

\subsection{Asymptotic Normality}
In the previous subsection, we established strong consistency of the estimator $\hat{\bs{\mu}}_n$, showing it converges almost surely to the true parameter ${\bs{\mu}_0}$. 
The next logical step is to discuss the fluctuation of the estimator around ${\bs{\mu}_0}$. 
In this section we show, under the assumptions imposed, that the fluctuations $\big(\hat{\bs{\mu}}_n - {\bs{\mu}_0}\big)$ scaled by a factor $\sqrt{n}$ tend to zero-mean normally distributed random variable.
Further advantages of working with regeneration cycles will become evident in this section. 
Being able to write down the log-likelihood as a sum of i.i.d.\ functions of the data in each regeneration cycle, we will use existing knowledge for proving Theorem \ref{theorem:asymptotic_normality}.
For convenience, let us define 
\begin{align*}
    \mathcal{L}_{N_r}\big(\bs{\mu}\big) := \frac{1}{N_r} \ell_n\big(\bs{\alpha}, \bs{\theta};{\bs A},{\bs W} \,\big\vert \,{\bs X}\big)=\frac{1}{N_r} \sum_{j=1}^{N_r} q\big(\bs{Z_j},{\bs\mu}).
\end{align*}
Lemma \ref{lemma:E(q_dot)=0} is important with regards to proving Theorem \ref{theorem:asymptotic_normality}, as we shall see below.
\begin{lemma}
    Under Assumptions {\em (A2), (A3), (A6), (A7), (A8), and (A9)}, $\exptn \big[\nabla q\big({\bs{Z}_1}, {\bs{\mu}_0}\big)\big]= \bs{0}_{(k+p)\times 1}. $\label{lemma:E(q_dot)=0}
\end{lemma}
\begin{proof}
    Deferred to Appendix \ref{section: appendix}.
\end{proof}

\begin{proof}[Proof of Theorem \ref{theorem:asymptotic_normality}]
As $n \rightarrow \infty$, by Theorem \ref{theorem:properties_of_system}(a), $N_r \rightarrow \infty$. Therefore, by the central limit theorem in combination with Lemma~\ref{lemma:E(q_dot)=0},
\begin{align}
    \sqrt{N_r} \ \nabla\mathcal{L}_{N_r}\big({\bs{\mu}_0}\big) = \sqrt{N_r}\times\frac{1}{N_r}\sum_{i=1}^{N_r} \nabla q\big(\bs{Z_i}, {\bs{\mu}_0}\big) &= \sqrt{N_r}\Bigg(\frac{1}{N_r}\sum_{i=1}^{N_r} \nabla q\big(\bs{Z_i}, {\bs{\mu}_0}\big) - \exptn \big[\nabla q\big(\bs{Z_i}, {\bs{\mu}_0}\big)\big]\Bigg) \notag
    \\
    &\overset{\text{d}}{\longrightarrow} \ \mathcal{N}\Big(0, \text{Var}\big(\nabla q\big({\bs{Z}_1}, {\bs{\mu}_0}\big)\big) \Big), \label{asymptotic_normality:eqn:1}
\end{align}
where, using Lemma \ref{lemma:E(q_dot)=0} again, the variance term is given by
\begin{align}\label{asymptotic_normality:eqn:2}
    \text{Var}\big(\nabla q\big({\bs{Z}_1}, {\bs{\mu}_0}\big)\big) &= \exptn\bigg[\Big(\nabla q\big({\bs{Z}_1}, {\bs{\mu}_0}\big) - \exptn\nabla q\big({\bs{Z}_1}, {\bs{\mu}_0}\big)\Big)\Big(\nabla q\big({\bs{Z}_1}, {\bs{\mu}_0}\big) - \exptn\nabla q\big({\bs{Z}_1}, {\bs{\mu}_0}\big)\Big)^{\text{T}} \Big] \notag
    \\ 
    &= \exptn\Big[\nabla q\big({\bs{Z}_1}, {\bs{\mu}_0}\big) \nabla q\big({\bs{Z}_1}, {\bs{\mu}_0}\big)^{\text{T}} \Big] \notag
    \\
    &= -\exptn\Big[\nabla^2 q\big({\bs{Z}_1}, {\bs{\mu}_0}\big)\Big],
\end{align}
where the final equality has been explained in \cite[Chapter 18]{ferguson2017course}.
As $n \rightarrow \infty$, for all $\bs{\mu} \in \Theta$, we also have by the strong law of large numbers,
\begin{align}
    \nabla^2 \mathcal{L}_{N_r} \big(\bs{\mu}\big) = \frac{1}{N_r} \sum_{i=1}^{N_r} \nabla^2 q\big(\bs{Z_i}, \bs{\mu}\big) \overset{\text{a.s}}{\longrightarrow} \exptn\Big[\nabla^2 q\big({\bs{Z}_1}, \bs{\mu}\big)\Big]. \label{asymptotic_normality:eqn:3}
\end{align}
For some $c \in (0,1)$ and $\tilde{\bs{\mu}}_{n} = c{\bs{\mu}_0} + (1-c)\hat{\bs{\mu}}_{n}$, by the mean-value theorem, 
\begin{align*}
    \nabla \mathcal{L}_{N_r}\big(\hat{\bs{\mu}}_{n}\big) = \nabla \mathcal{L}_{N_r}\big(\bs{\mu}_0\big) + \nabla^2 \mathcal{L}_{N_r}\big(\tilde{\bs{\mu}}_{n}\big) \cdot \big(\hat{\bs{\mu}}_{n} - \bs{\mu}_0 \big).
\end{align*}
By Theorem \ref{theorem:consistency}, we know that $\hat{\bs{\mu}}_{n} \overset{\text{a.s.}}{\rightarrow} \bs{\mu}_0$, and by Assumption (A1), $\bs{\mu}_0 \in \Theta^{\circ}$. 
This means that there exists an almost surely finite $\tilde{n}$, such that for all $n > \tilde{n}$, $\hat{\bs{\mu}}_n \in \Theta^\circ$.
For $n > \tilde{n}$, since $\hat{\bs{\mu}}_{n} \in \Theta^{\circ}$ is a maximum likelihood estimate, we know that $\nabla \ell_n\big(\hat{\bs{\mu}}_{n};{\bs A},{\bs W} \,\big\vert \,{\bs X}\big) = 0$ for all $n> \tilde{n}$, and therefore, $\nabla \mathcal{L}_{N_r}\big(\hat{\bs{\mu}}_{n}\big)  \overset{\text{d}}{\longrightarrow} 0$.
Hence, the above equation can be rewritten as 
\begin{align}
    \nabla^2 \mathcal{L}_{N_r}\big(\tilde{\bs{\mu}}_n\big) \cdot \big(\hat{\bs{\mu}}_{n} - \bs{\mu}_0 \big) = - \nabla \mathcal{L}_{N_r}\big(\bs{\mu}_0\big).
\end{align}
Multiplying both sides by $\sqrt{N_r}$ we get
\begin{align*}
    \nabla^2 \mathcal{L}_{N_r}\big(\tilde{\bs\mu}_n\big) \cdot \sqrt{N_r}\big(\hat{\bs{\mu}}_{n} - \bs{\mu}_0\big) = -\sqrt{N_r} \nabla \mathcal{L}_{N_r}\big(\bs{\mu}_0\big).
\end{align*}
Since $\hat{\bs{\mu}}_{n} \overset{\text{a.s.}}{\longrightarrow} {\bs{\mu}_0}$ as $n \rightarrow \infty$, and $\big\vert \tilde{\bs{\mu}}_{n} - {\bs{\mu}_0}\big\vert \leq \big\vert \hat{\bs{\mu}}_{n} - {\bs{\mu}_0}\big\vert$, we can conclude that $\tilde{\bs{\mu}}_{n} \overset{\text{a.s.}}{\longrightarrow} {\bs{\mu}_0}$.
Using this fact and Equations (\ref{asymptotic_normality:eqn:1}), (\ref{asymptotic_normality:eqn:2}), (\ref{asymptotic_normality:eqn:3}) as $N_r \rightarrow \infty$, we conclude
\begin{align}\label{equation:asymptotic_normality:sqrt(N_r)_convergence}
    \sqrt{N_r}\big(\hat{\bs{\mu}}_{n} - \bs{\mu}_0\big) \overset{\text{d}}{\longrightarrow} & -\exptn\Big[\nabla^2 q\big({\bs{Z}_1}, {\bs{\mu}_0}\big)\Big]^{-1} \mathcal{N}\bigg(0,  -\exptn\Big[\nabla^2 q\big({\bs{Z}_1}, {\bs{\mu}_0}\big)\Big] \bigg) \notag
    \\
    \overset{\text{d}}{=} & \ \mathcal{N}\bigg(0, -\exptn\Big[\nabla^2 q\big({\bs{Z}_1}, {\bs{\mu}_0}\big)\Big]^{-1}\bigg) 
    \overset{\text{d}}{=}  \ \mathcal{N}\Big(0, I\big({\bs{\mu}_0}\big)^{-1} \Big).
\end{align}
Recalling that ${n}/{N_r} \overset{\text{a.s.}}{\longrightarrow} \exptn  C_1$, and multiplying (\ref{equation:asymptotic_normality:sqrt(N_r)_convergence}) on both sides with $\sqrt{n/N_r}$, we conclude
\begin{align}
    \sqrt{n} \big({\hat{\bs{\mu}}_n} - {\bs{\mu}_0}\big) \overset{\text{d}}{\longrightarrow} \mathcal{N}\Big(0, \ \exptn C_1\, I\big({\bs{\mu}_0}\big)^{-1} \Big)\,,
\end{align}
as $n\to\infty$.
\end{proof}

\section{GENERAL FRAMEWORK FOR JOINING \\BASED ON INCOMPLETE INFORMATION}\label{section: joining decisions based on incomplete information}




In Section \ref{section: Parametric estimation procedure}, we assured the reader that we performed calculations for the case $\Delta(t) = V(t)$ only for the sake of demonstration.
In this section, we will generalize Theorems \ref{theorem:consistency} and \ref{theorem:asymptotic_normality} for a generalized delay announcement $\Delta(t) = \psi\big(\mathcal{R}(t)\big)$.
The consistency and asymptotic normality proofs of this estimator once again rely on the underlying idea of dividing the sample path into distinct regeneration cycles. 
For regeneration cycle $j \in \{1,2,\ldots,N\}$, let the objects $\{\tilde{A}_{j,i}\}_{i=1,2,\ldots,C_j}$, $\zeta_j$, $R_j$, and $C_j$ carry the same meaning as before. 
Let $\bs{Z_j}$ denote the new data vector corresponding to the $j$-th regeneration cycle. As in Sections~\ref{section: Constructing i.i.d regeneration cycles} and~\ref{section: strong consistency}, we define, for compactness abbreviating $h(x):=\tilde{H}_{\bs{\theta}}(\psi(\mathcal{R}(x)))$,
\begin{align}\label{eqn:q:incomplete}
    q\big(\bs{Z_j}, \bs{\mu}\big) = &\sum_{i=1}^{C_j} \log\lambda_{\bs{\alpha}}\big(\tilde{A}_{j,i}\big) + \sum_{i=1}^{C_j} \log h(\tilde{A}_{j,i}) - \sum_{i=2}^{n} \int_{0}^{\tilde{A}_{j,i}} \lambda_{\bs{\alpha}}\big(u + \tilde{A}_{j,i-1}\big) h(u+\tilde{A}_{j,i-1}) \, \mathrm{d}u \notag
    \\
    & - \int_{0}^{A_{j,1}} \lambda_{\bs{\alpha}}(u)h(u)\,{\rm d}u -\int_{0}^{R_j - \tilde{A}_{j,C_j}} \lambda_{\bs{\alpha}}\big(u + \tilde{A}_{j,C_j}\big) h(u+\tilde{A}_{j,C_j})\, \mathrm{d}u.
\end{align}

\begin{assumption**}[A11] Either of the following additional assumption is imposed:
\begin{itemize}
        \item[{\rm (a)}]
        Let $\tilde{\Omega}$ be the collection of all possible states of our system (as explained at the beginning of Section \ref{section: joining decisions based on incomplete information}), and let $\tilde{V}: \tilde{\Omega} \mapsto \mathbb{R}_{+}$ be the virtual waiting time function, in the sense that, given a vector of residual service times of customers $\bs{r} \in \tilde{\Omega}$, $\tilde{V}(\bs{r})$ is the corresponding virtual waiting time.
        Then there exists a constant $\tilde{M} > 0$ such that for every $\bs{r} \in \tilde{\Omega}$, we must have $\psi(\bs{r}) \leq \tilde{M}\tilde{V}(\bs{r})$.

        \item[{\rm (b)}] $\psi\big(\mathcal{R}(t)\big) = \psi_1\big(L(t)\big)$, where $\psi_1$ is a linear function, and $\exptn\big[C_1^2\big] < \infty$.
    \end{itemize}
\end{assumption**}

\textbf{Discussion of Assumptions, continued.}
\begin{itemize}
    \item[$\circ$] (A11)(a) essentially says that given the virtual waiting time is $V(t)$, the delay proxy $\psi(\mathcal{R}(t))$ should not be larger than $\tilde{M}V(t)$. In particular, model 1 given earlier clearly follows (A11)(a) with the choice $\tilde{M}=1$.
    \item[$\circ$] (A11)(b) is required for delay proxies based on the number of customers in the system length, such as Models 2,3. The first part of the assumption reflects this. The second part assumes that the second moment of the number of customers in a cycle is finite. Verifying this condition can be a challenging problem in the context of M$_t$/G/$s$+H systems. 
    In theory, we can allow $\psi\big(\mathcal{R}(t)\big) = \psi_1\big(L(t)^{1+\eta}\big)$, but then we require $\exptn\big[C_1^{2+\eta}\big] < \infty$. 
\end{itemize}
Observe that model 4 satisfies neither (A11)(a) or (A11)(b). However, Assumption (A11) merely tries to provide sufficient conditions under which Proposition \ref{propn} holds. We do not believe that these are strictly necessary.


\begin{propn} \label{propn}
    Under assumptions {\rm (A1)--(A6)} and {\rm (A11)(a)} or {\rm (A11)(b)},  the statement of Theorem~\ref{theorem:consistency} holds in the generalized setting of this section. Furthermore, under assumptions {\rm (A1)--(A10)} and {\rm (A11)(a)} or {\rm (A11)(b)}, the statement of Theorem \ref{theorem:asymptotic_normality} holds in the generalized setting of this section, where $q(\bs{Z}_1, \bs\mu)$ is now given by \eqref{eqn:q:incomplete}.
\end{propn}
The proof of Proposition \ref{propn} is very similar to that of Theorems \ref{theorem:consistency}-\ref{theorem:asymptotic_normality}, as detailed in Appendix \ref{section: appendix}.


\section{NUMERICAL EXPERIMENTS}\label{section: numerical experiments}

In this section we numerically assess the performance of our estimation procedure through a series of experiments. 
We consider three mechanisms:
\begin{itemize}
\item[(a)] customers know their exact waiting time, 
\item[(b)] customers receive an estimate for their expected waiting time based on the queue length, and 
\item[(c)] customers simply observe the number of customers in the system, and join or balk accordingly. 
\end{itemize}
For each of these, we consider an arrival rate function that is a sinusoidal or a weighted sum of sinusoidals (where the ratios of the periods are rational numbers, making the aggregate arrival rate periodic), and exponential, hyper-exponential, Pareto, or geometric distributed patience. 
Maximization of the log-likelihood is potentially challenging, as there is no guarantee whatsoever on local maxima also being global maxima. 
To deal with this complication, we have used (the Python implementation of) the L-BFGS-B and Nelder-Mead algorithms. 
For a given observation of the arrival process, service times and patience levels, we simulate 5 or 6 different service systems, with 1, 2, 4, 8, and 16/32 servers. 
We repeat this procedure 100 times to generate empirical confidence intervals of the obtained estimates, which we visualize via box plots.
In each of the experiments we report the number of arrivals that has been used, which has typically been chosen such that the width of the corresponding confidence intervals is `sufficiently small'.

\subsection{Model I: Exact waiting time is known}
In this model, customers are provided with the exact value of the virtual waiting time upon arrival. This means that the information function is assumed to be $\Delta(t) = V(t)$, so that we are in the framework of Sections \ref{section: Model and preliminaries}--\ref{section: strong consistency}.

\subsubsection{Sinusoidal arrival rate and exponentially distributed patience}
In this example, our model of choice is an M$_t$/M/$s$+H system, with arrival rate, service time distribution and patience distribution given by
\begin{align*}
    \lambda_{\bs{\alpha}}(t) &= 50 + 20\sin\big(1 - 0.1t\big),
    \\
    M &\sim \text{Exp}(0.2),
    \\
    H_{\bs{\theta}}(x) &= 1 - e^{-0.5x}.
\end{align*}
The parameters to be estimated are therefore $\bs{\alpha}_0 = (50, 20, 1)$, and $\bs{\theta}_0 = 0.5$.  
We work with simulated data corresponding to $20,000$ total arrivals (i.e., balking and non-balking customers).
The resulting box plots are shown in Figure \ref{fig:box plots:model I: submodel 1}. 
Recalling Example \ref{example: best and worst instances for estimation}, where we discussed how the service capacity affects the  accuracy of the estimates, 
in this example we observe the same phenomenon: the accuracy of the $\bs{\alpha}$ estimates improves as the number of servers increases. 
On the other hand, observe from Figure~\ref{fig:box plots:model I: submodel 1: theta} that $\bs{\theta}$ is estimated poorly when the number of servers is relatively large: 
then the virtual waiting time becomes small, so that virtually any arriving customer joins, thus not providing any information on the patience distribution. 
In a more formal sense, this can also be observed from Equation \eqref{eqn:log_likelihood:exact}, by noting that the terms which contain $\theta$ start disappearing: if $W_{i-1}$ and $ X_{i-1}$ are virtually zero, then $W_{i-1} + X_{i-1} - A_i$ typically takes negative values, and therefore, $\log \tilde{H}_{\bs{\theta}}(W_{i-1}+X_{i-1}-A_i)$ is zero. 
A similar effect taking place inside the integral, all terms containing $\bs{\theta}$ start dissolving for such higher numbers of servers, thus prohibiting the optimization of the log-likelihood over $\bs{\theta}$. 
The patience parameter ${\bs \theta}$ is generally best estimated if the number of servers allows for a system where the virtual waiting takes a wide spectrum of values, which is typically the case when the number of servers is relatively low.
However, if the number of servers is really low, then the fraction of customers joining is low, which has a detrimental effect on the accuracy of the estimator of $\bs\theta$.
Due to the conflict between the accuracy of the estimators for $\bs{\alpha}$ and $\bs{\theta}$, scenarios in which the number of servers is `moderate' lead to the best overall performance.

\begin{figure}[ht]
    \begin{subfigure}{0.45\textwidth}
        \includegraphics[width=\linewidth]{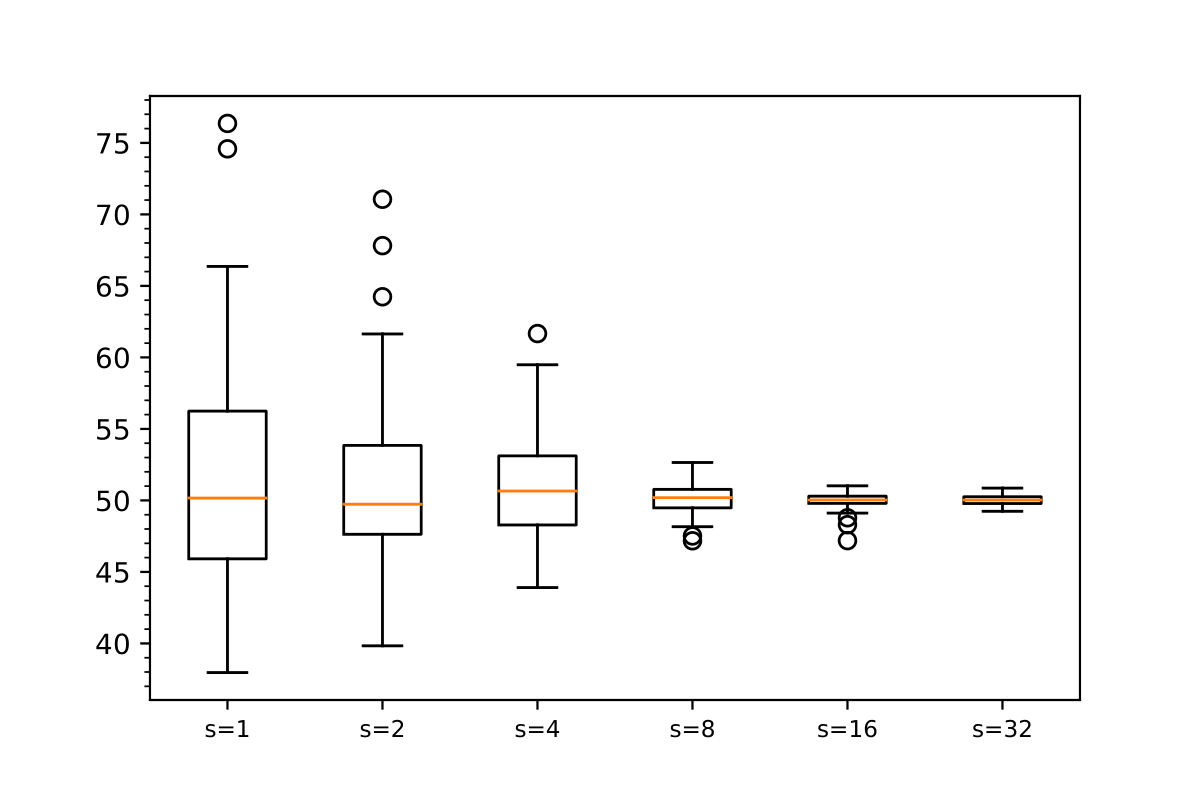}
        \caption{Confidence Intervals for $\alpha_1$ (True value = 50)} \label{fig:box plots:model I: submodel 1: a}
        \end{subfigure}\hspace*{\fill}
        \begin{subfigure}{0.45\textwidth}
        \includegraphics[width=\linewidth]{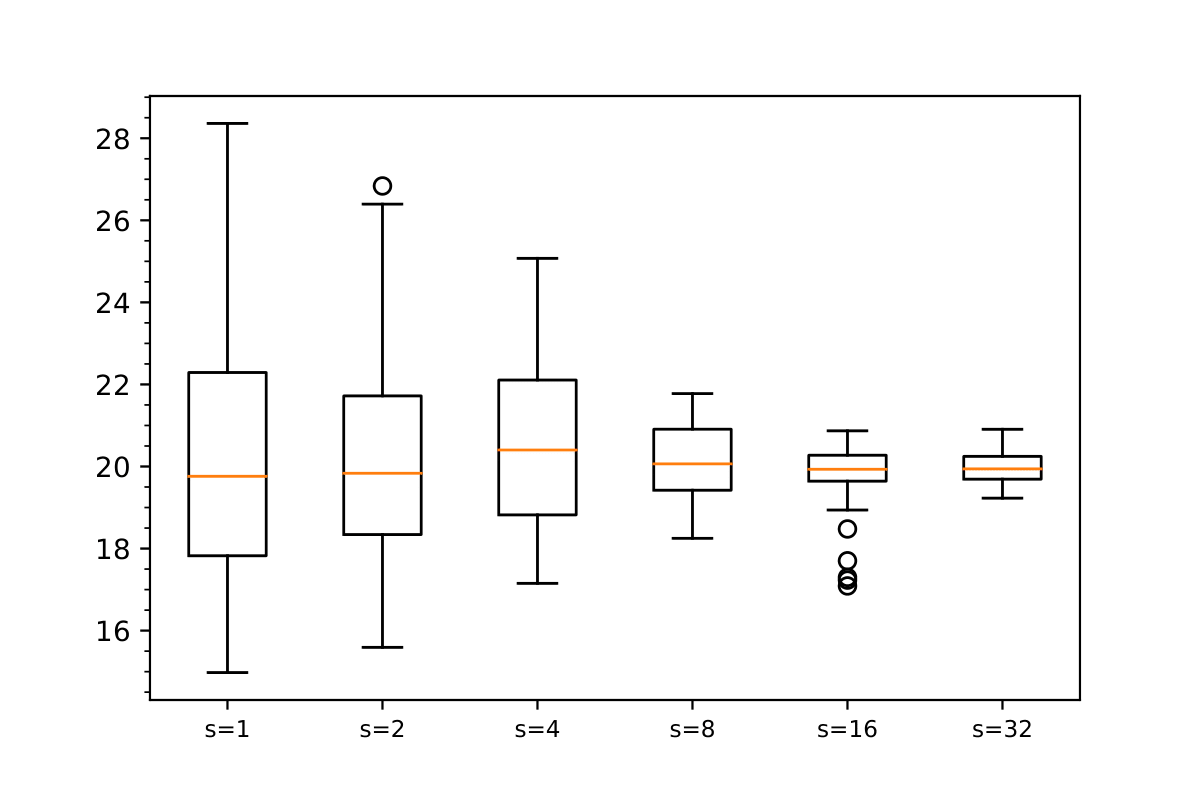}
        \caption{Confidence Intervals for $\alpha_2$ (True value = 20)} \label{fig:box plots:model I: submodel 1: b}
    \end{subfigure}
    
    \medskip
    \begin{subfigure}{0.45\textwidth}
        \includegraphics[width=\linewidth]{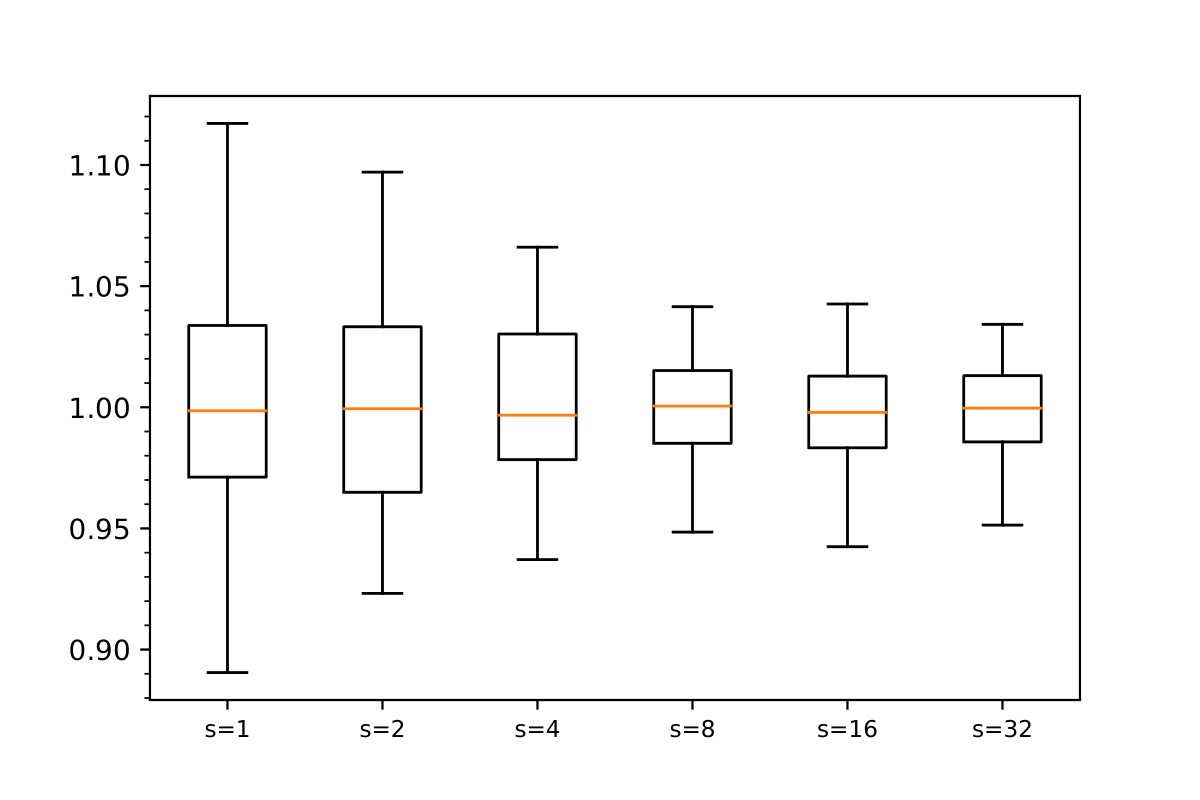}
        \caption{Confidence Intervals for $\alpha_3$ (True value = 1)} \label{fig:box plots:model I: submodel 1: c}
        \end{subfigure}\hspace*{\fill}
        \begin{subfigure}{0.45\textwidth}
        \includegraphics[width=\linewidth]{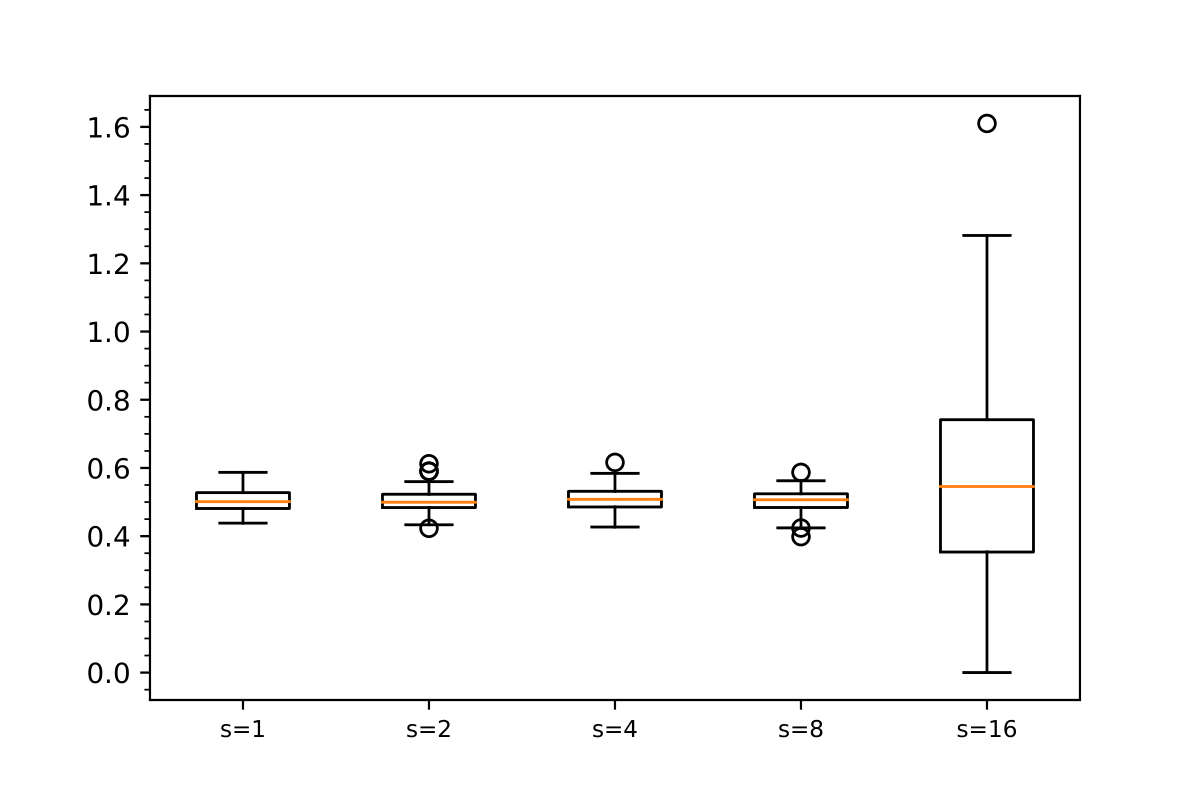}
        \caption{Confidence Intervals for $\theta$ (True value = 0.5)} \label{fig:box plots:model I: submodel 1: theta}
    \end{subfigure}

\caption{Box plots of Model I - Submodel 1 estimates} \label{fig:box plots:model I: submodel 1}
\end{figure}

\subsubsection{Weighted sinsuoidal arrival rate and hyperexponentially distributed patience}

In this example, our model of choice is an M$_t$/M/$s$+H system, with arrival rate, service distribution and patience distribution given by
\begin{align*}
    \lambda_{\bs{\alpha}}(t) &= 50 + 15\sin\big(4 - 0.1t\big) + 10\sin\big(1 - 0.5t\big),
    \\
    M &\sim \text{Exp}(0.4),
    \\
    H_{\bs{\theta}}(x) &= 1 - 0.8e^{-x} - 0.2e^{-0.1x}.
\end{align*}
The parameters to be estimated are therefore $\bs{\alpha}_0 = (50, 15, 10, 4, 1)$, and $\bs{\theta}_0 = (0.8, 1, 0.1)$. 
In this example, the number of total arrivals is $10^6$, leading to 
the box plots of Figure~\ref{fig:box plots:model I: submodel 2}.
Similar to the previous model, we observe that accuracy of the estimator of $\bs{\alpha}$ improves as the number of servers increases. 
However, the estimates corresponding to $\bs{\theta}$ now do not exhibit a clear pattern. 
When the number of servers is low (say, $s=1$), the estimates of $\theta_3$ is accurate, whereas the estimate of $\theta_2$ is poor. 
This is because virtual waiting times are typically high, entailing that all customers with patience sampled from the $\text{Exp}(1)$ distribution tend to get rejected due to their low patience values. 
Therefore, a large proportion of non-balking customers have patience sampled from $\text{Exp}(0.1)$, so that $\theta_3$ can be estimated accurately but $\theta_2$ not. 
This model requires a substantial amount of data for the estimation procedure for multiple reasons, one of which has been discussed in Example \ref{example: identifiability issues}. 
The other reason is that the log-likelihood has a rather irregular surface, which is hard to optimize over (in comparison to the previous model).

\begin{figure}[ht]
    \begin{subfigure}{0.45\textwidth}
        \includegraphics[width=\linewidth]{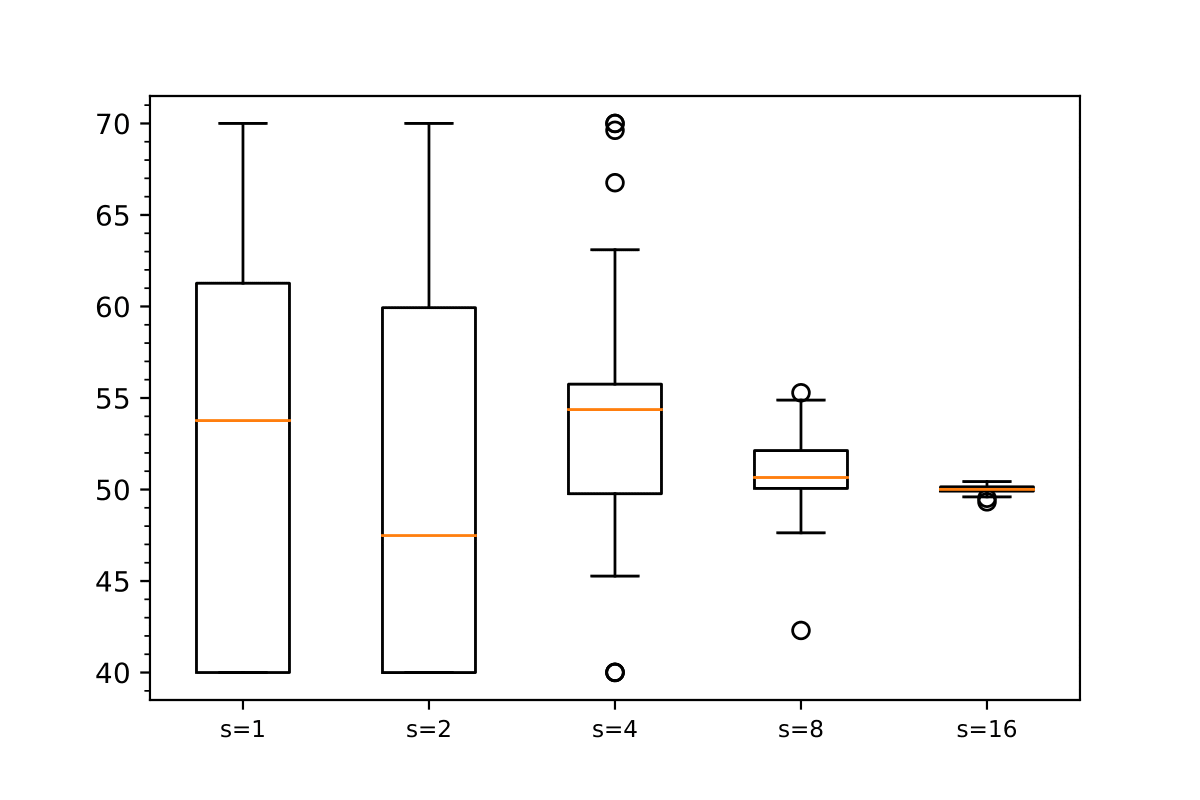}
        \caption{Confidence Intervals for $\alpha_1$ (True value = 50)} \label{fig:box plots:model I: submodel 2: a}
        \end{subfigure}\hspace*{\fill}
        \begin{subfigure}{0.45\textwidth}
        \includegraphics[width=\linewidth]{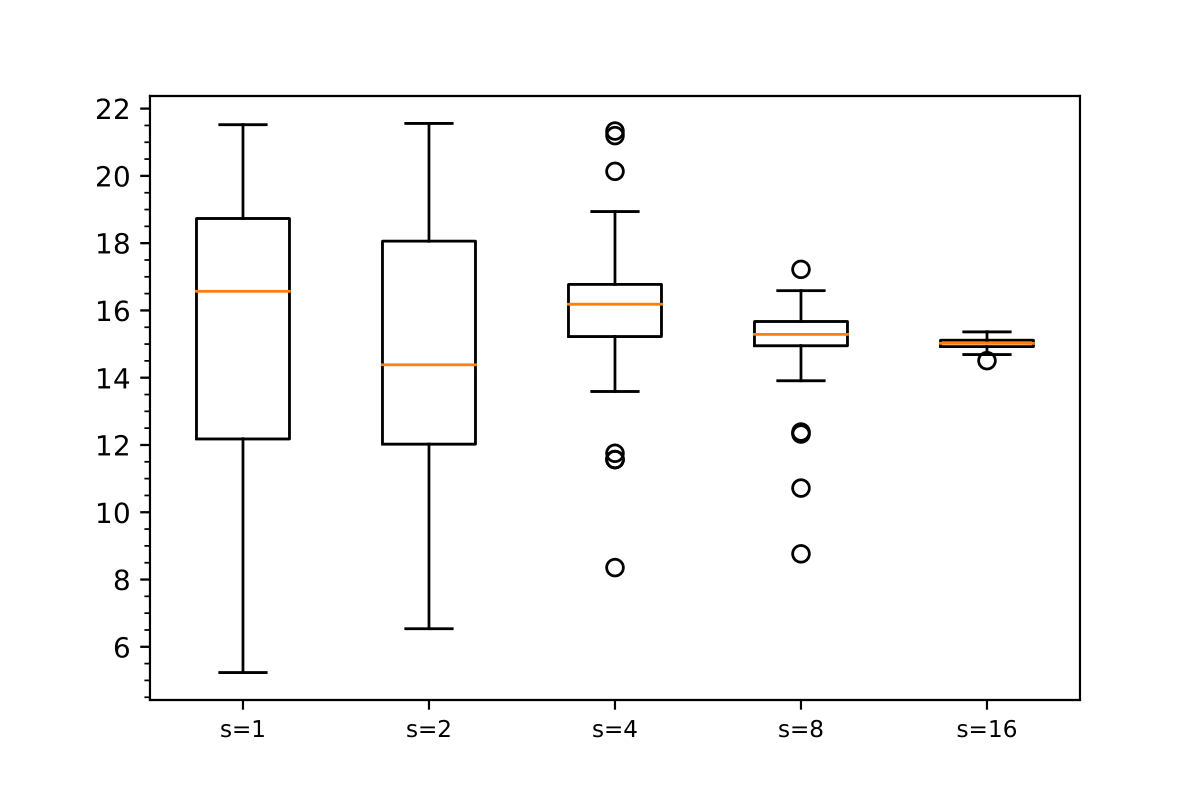}
        \caption{Confidence Intervals for $\alpha_2$ (True value = 15)} \label{fig:box plots:model I: submodel 2: b1}
    \end{subfigure}
    
    \medskip
    \begin{subfigure}{0.45\textwidth}
        \includegraphics[width=\linewidth]{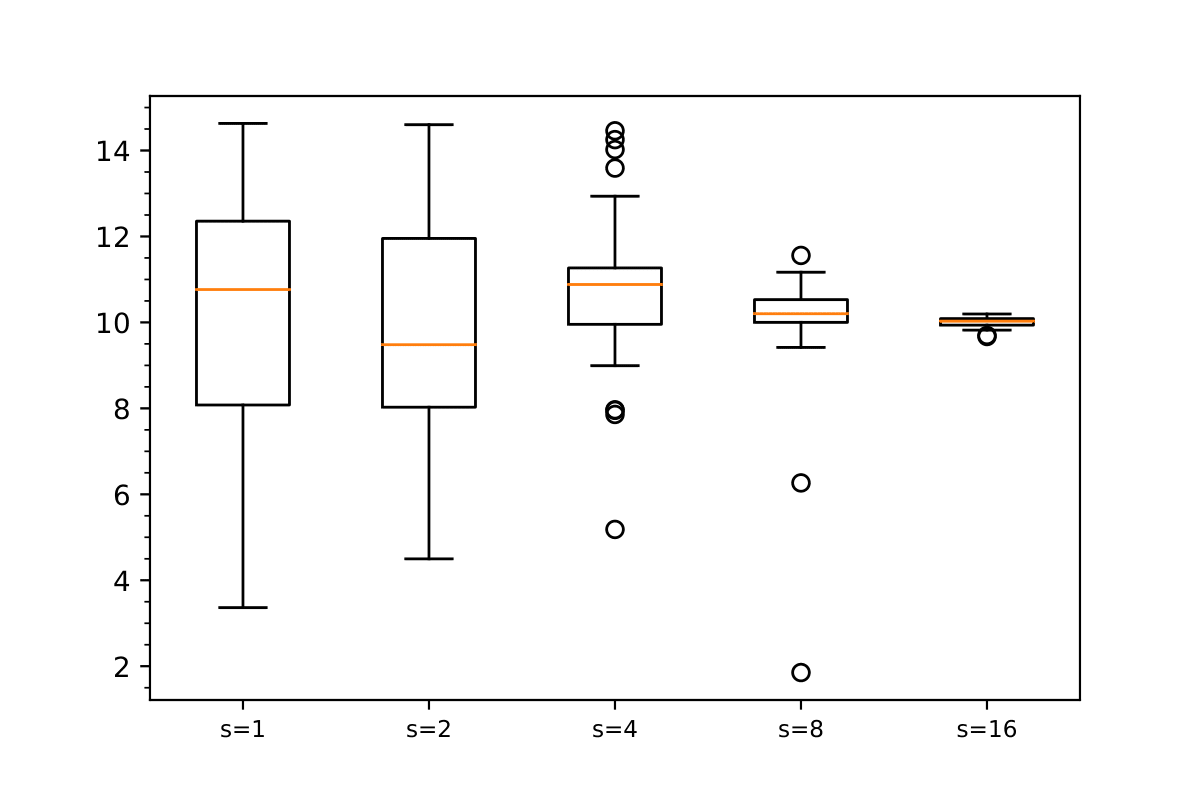}
        \caption{Confidence Intervals for $\alpha_3$ (True value = 10)} \label{fig:box plots:model I: submodel 2: b2}
        \end{subfigure}\hspace*{\fill}
        \begin{subfigure}{0.45\textwidth}
        \includegraphics[width=\linewidth]{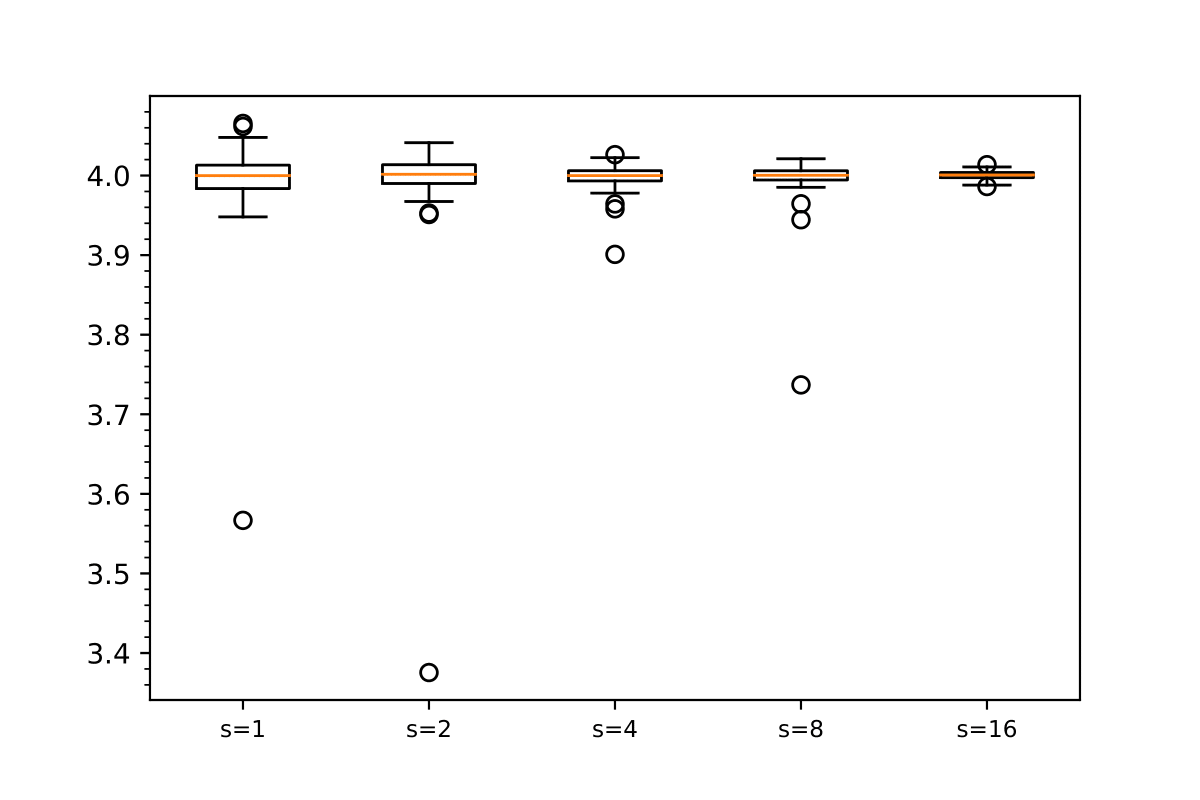}
        \caption{Confidence Intervals for $\alpha_4$ (True value = 4)} \label{fig:box plots:model I: submodel 2: c1}
    \end{subfigure}

    \medskip
    \begin{subfigure}{0.45\textwidth}
        \includegraphics[width=\linewidth]{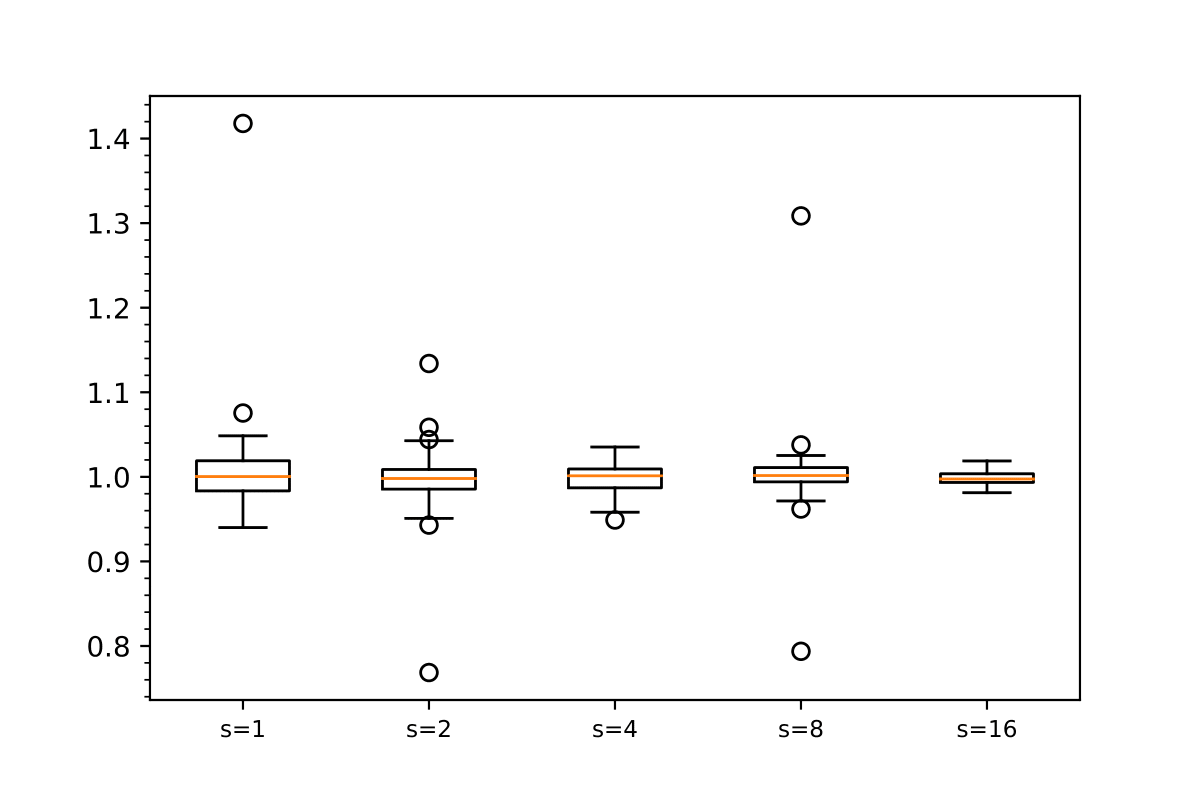}
        \caption{Confidence Intervals for $\alpha_5$ (True value = 1)} \label{fig:box plots:model I: submodel 2: c2}
        \end{subfigure}\hspace*{\fill}
        \begin{subfigure}{0.45\textwidth}
        \includegraphics[width=\linewidth]{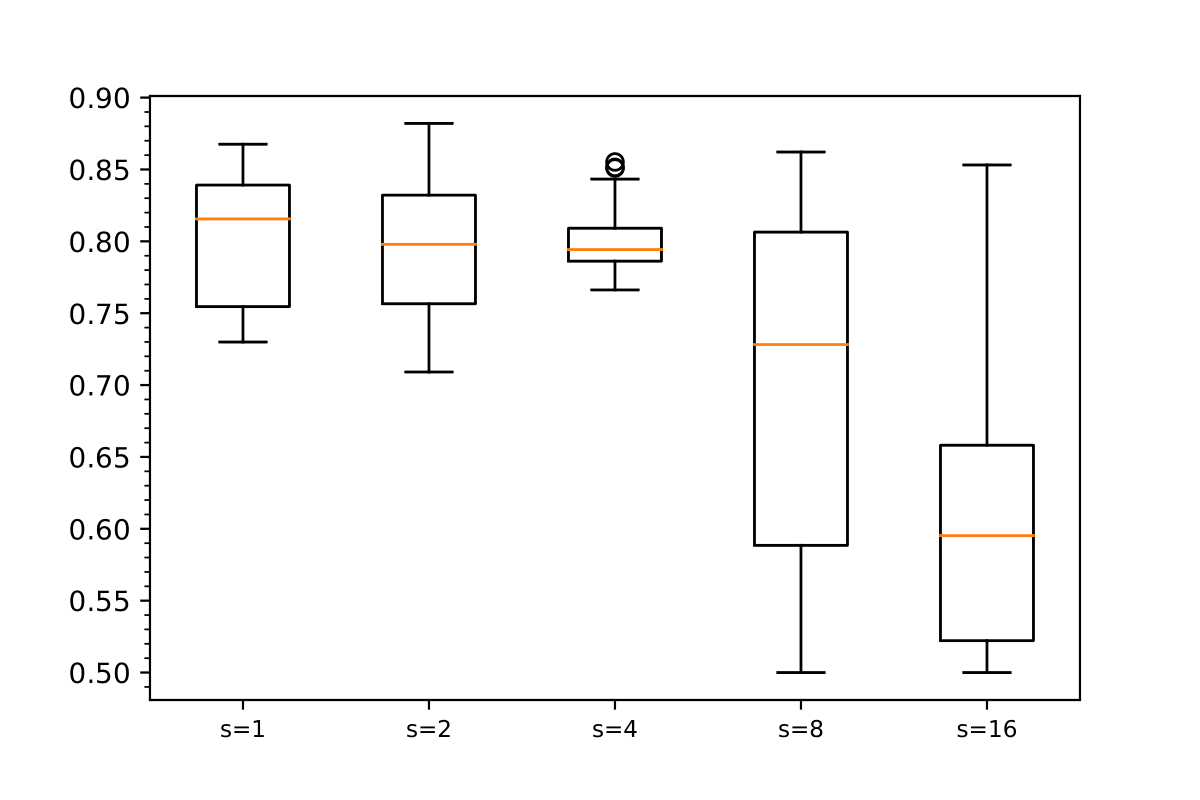}
        \caption{Confidence Intervals for $\theta_1$ (True value = 0.8)} \label{fig:box plots:model I: submodel 2: alpha}
    \end{subfigure}

    \medskip
    \begin{subfigure}{0.45\textwidth}
        \includegraphics[width=\linewidth]{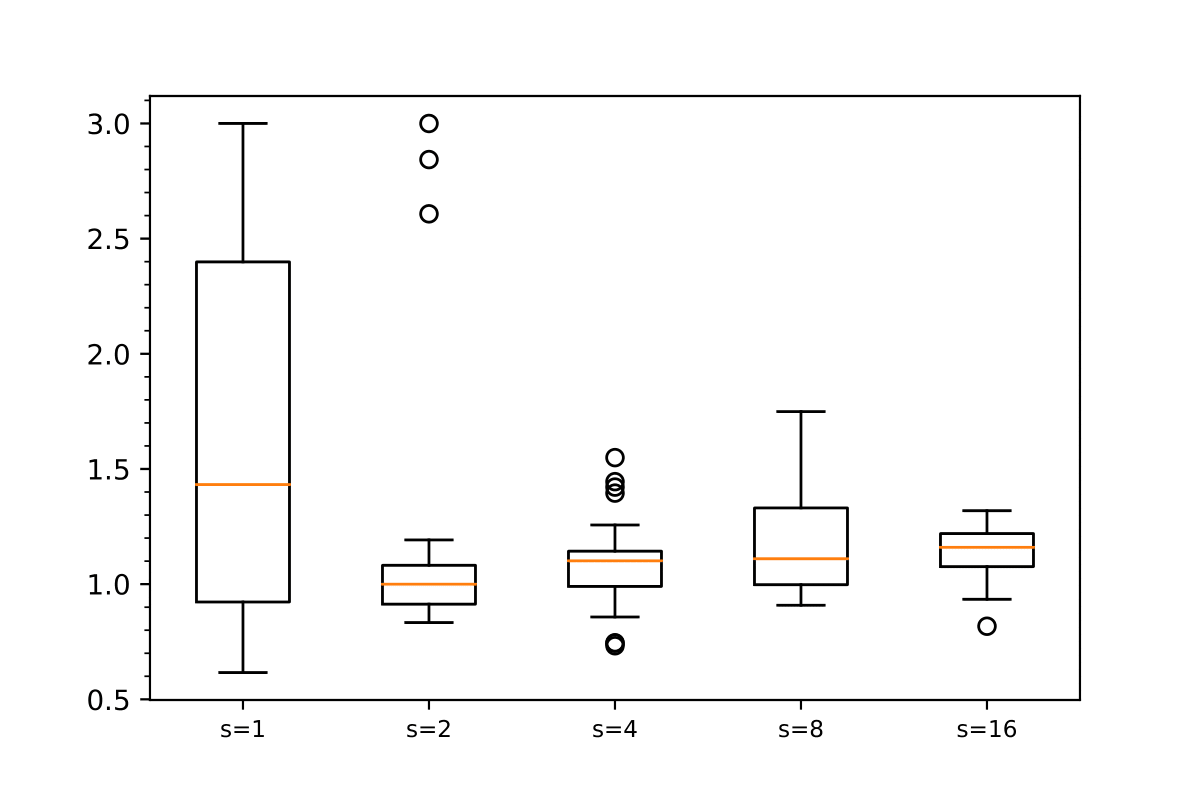}
        \caption{Confidence Intervals for $\theta_2$ (True value = 1)} \label{fig:box plots:model I: submodel 2: theta_1}
        \end{subfigure}\hspace*{\fill}
        \begin{subfigure}{0.45\textwidth}
        \includegraphics[width=\linewidth]{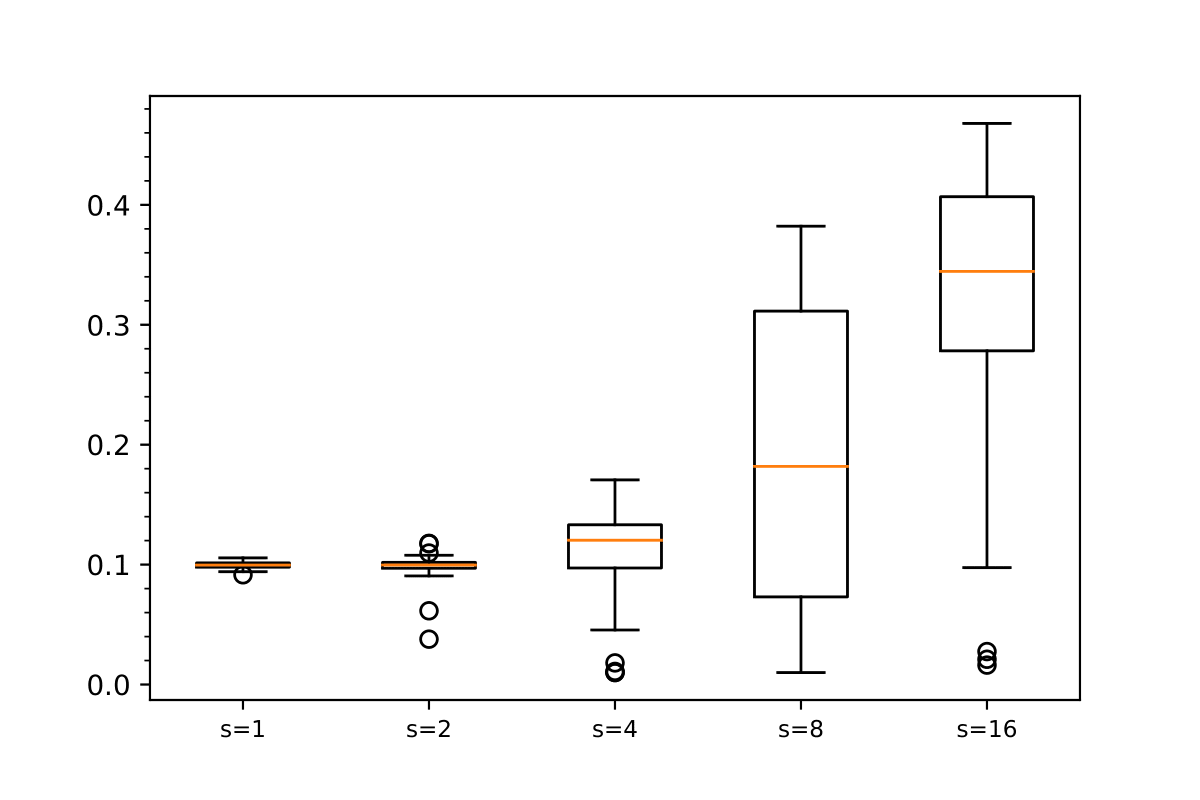}
        \caption{Confidence Intervals for $\theta_3$ (True value = 0.1)} \label{fig:box plots:model I: submodel 2: theta_2}
    \end{subfigure}

\caption{Box plots of Model I - Submodel 2 estimates} \label{fig:box plots:model I: submodel 2}
\end{figure}

\subsection{Model II: Estimated virtual waiting time based joining}

We now demonstrate the procedure working with incomplete information, as was presented in Section~\ref{section: joining decisions based on incomplete information}, in which the delay proxy at time $t$ can be any estimate of the waiting time, depending only on the information in the regeneration cycle time $t$ belongs to.
In this particular experiment we choose $\Delta(t) = (L(t)-s+1)^{+}\,{\mathbb E}B/s$, where $L(t)$ denotes the number of customers in the system at time~$t$. 
Our model of choice is an M$_t$/G/$s$+H system with arrival rate, service disribution and patience distribution given by
\begin{align*}
    \lambda_{\bs{\alpha}}(t) &= 50 + 20\sin(4-0.1t) + 10\sin(1-0.5t),
    \\
    G &\sim \text{Gamma}(0.1,4),
    \\
    H_{\bs{\theta}}(x) &= 1-\frac{1}{(1 + x)^2}.
\end{align*}
The paramaters to be estimated are therefore $\bs{\alpha}_0 = (50, 20, 10, 4, 1)$, and $\bs{\theta}_0 = (1,2)$. We use simulated data from 0.5$\times 10^6$ total arrivals (balking and non-balking).
Figure \ref{fig:box plots:model II: submodel 1} shows the box plots for the empirical confidence intervals of the estimates for $\bs\alpha$ and $\bs\theta$.
They follow a similar trend as in the one observed in previous examples.
\begin{figure}[ht]
    \begin{subfigure}{0.45\textwidth}
        \includegraphics[width=\linewidth]{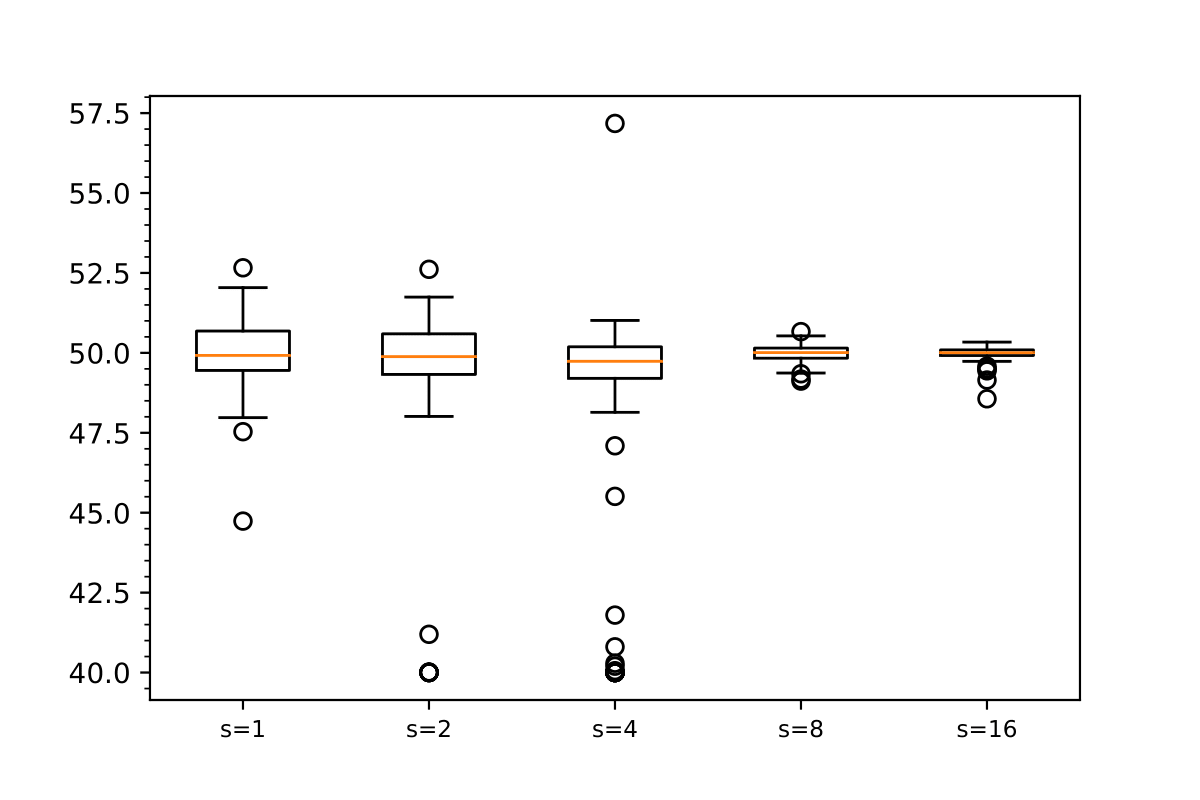}
        \caption{Confidence Intervals for $\alpha_1$ (True value = 50)} \label{fig:box plots:model II: submodel 1: a}
    \end{subfigure}\hspace*{\fill}
    \begin{subfigure}{0.45\textwidth}
        \includegraphics[width=\linewidth]{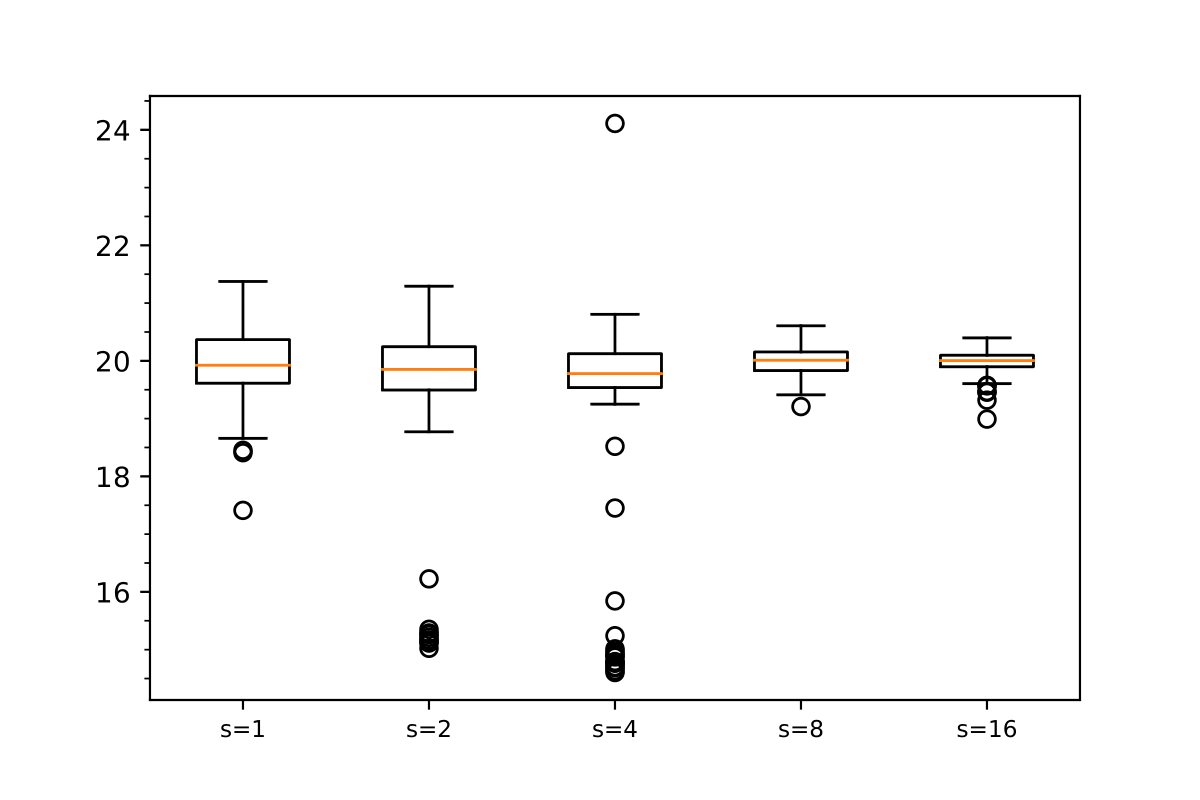}
        \caption{Confidence Intervals for $\alpha_2$ (True value = 20)} \label{fig:box plots:model II: submodel 1: b1}
    \end{subfigure}
    
    \medskip
    \begin{subfigure}{0.45\textwidth}
        \includegraphics[width=\linewidth]{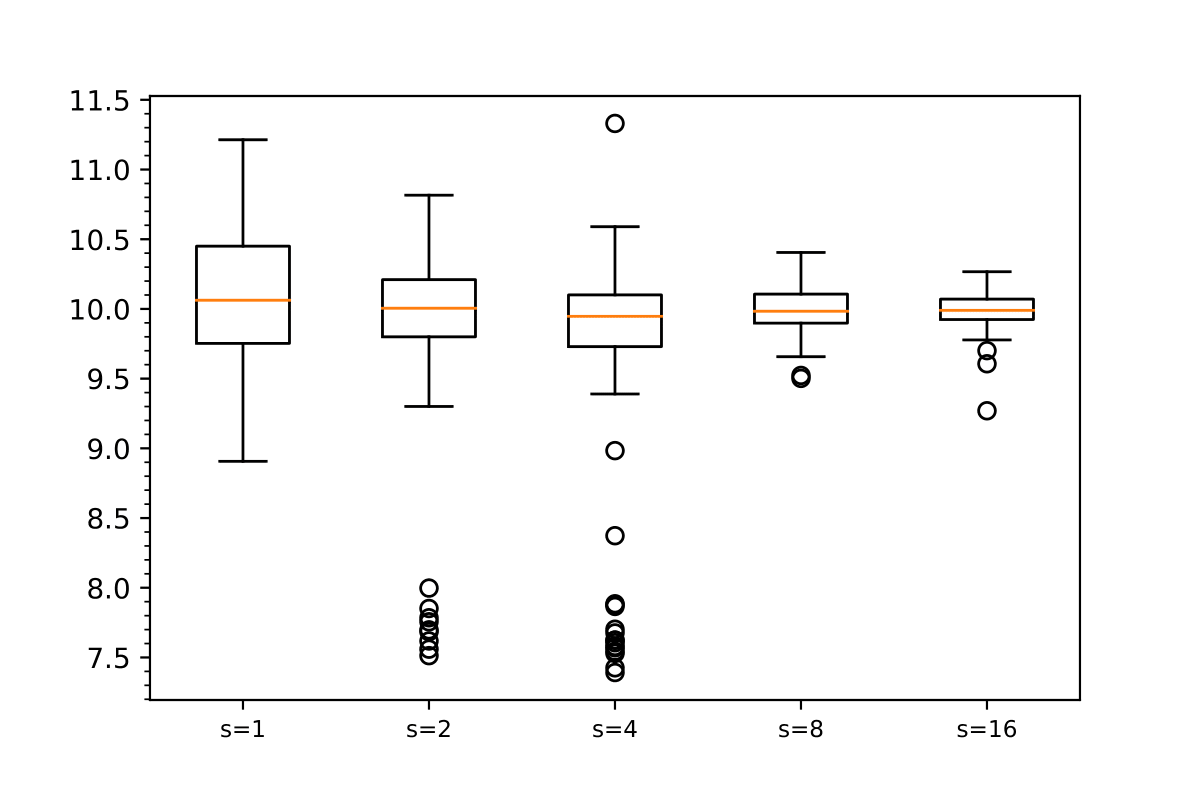}
        \caption{Confidence Intervals for $\alpha_3$ (True value = 10)} \label{fig:box plots:model II: submodel 1: b2}
    \end{subfigure}\hspace*{\fill}
    \begin{subfigure}{0.45\textwidth}
        \includegraphics[width=\linewidth]{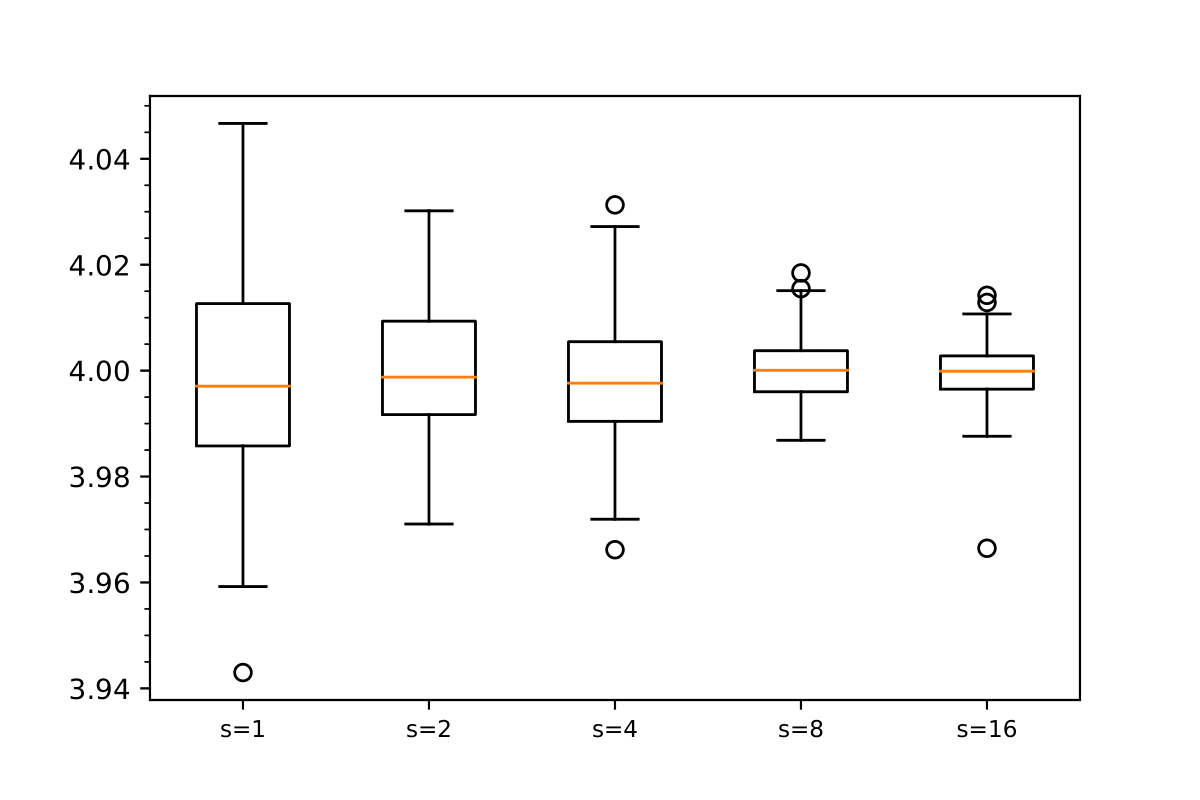}
        \caption{Confidence Intervals for $\alpha_4$ (True value = 4)} \label{fig:box plots:model II: submodel 1: c1}
    \end{subfigure}

    \medskip
    \begin{subfigure}{0.45\textwidth}
        \includegraphics[width=\linewidth]{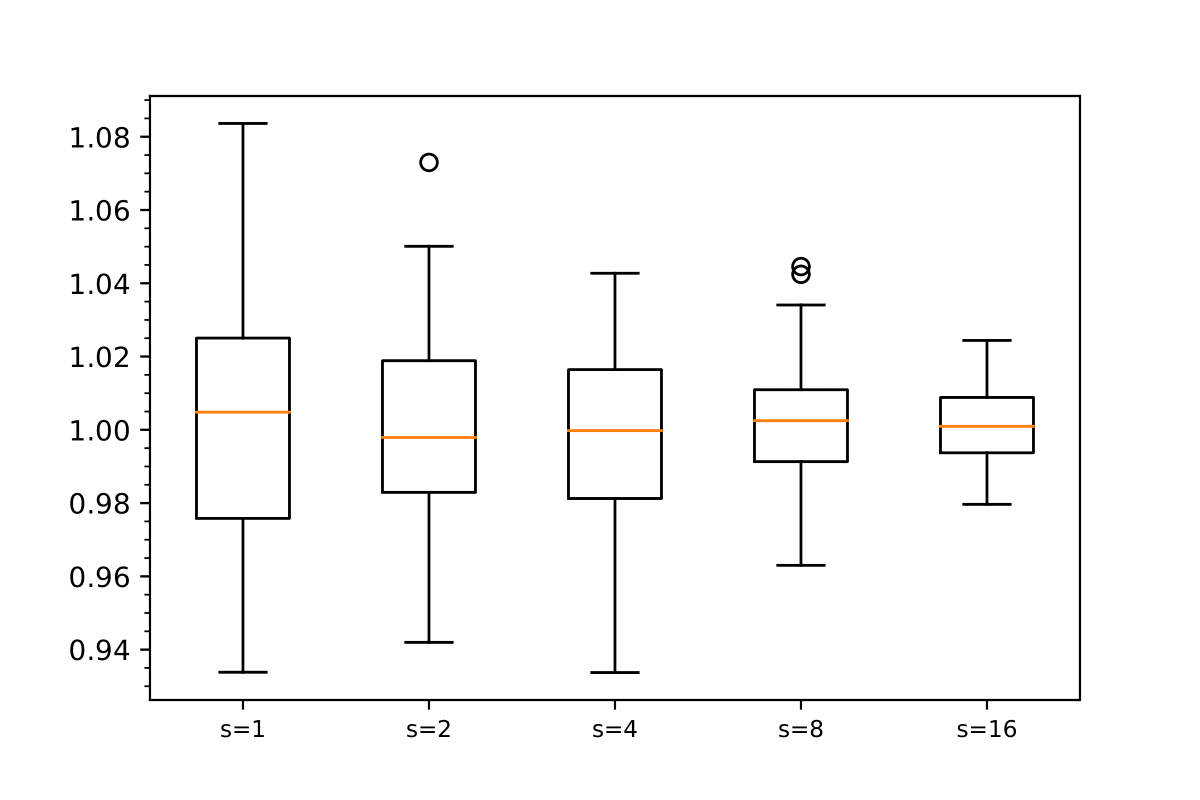}
        \caption{Confidence Intervals for $\alpha_5$ (True value = 1)} \label{fig:box plots:model II: submodel 1: c2}
    \end{subfigure}\hspace*{\fill}
    \begin{subfigure}{0.45\textwidth}
        \includegraphics[width=\linewidth]{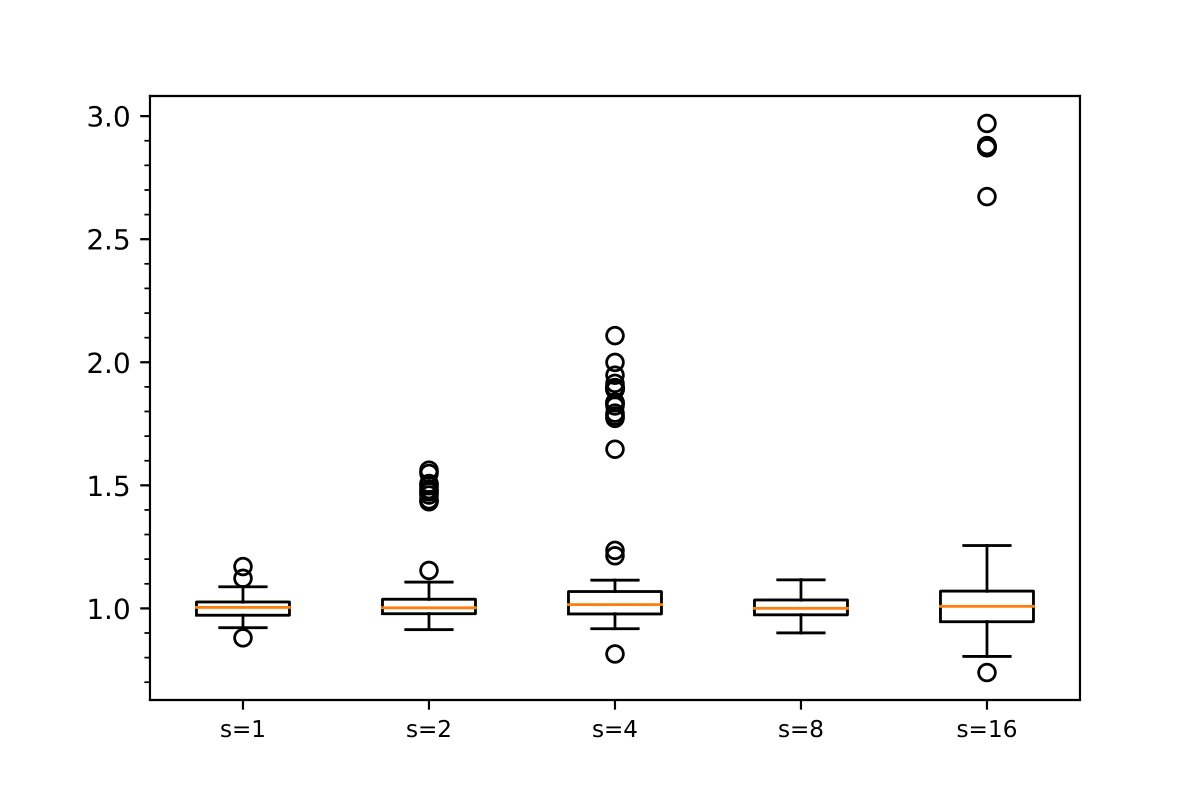}
        \caption{Confidence Intervals for $\theta_1$ (True value = 1)} \label{fig:box plots:model II: submodel 1: alpha}
    \end{subfigure}

    \medskip
    \begin{subfigure}{0.45\textwidth}
        \includegraphics[width=\linewidth]{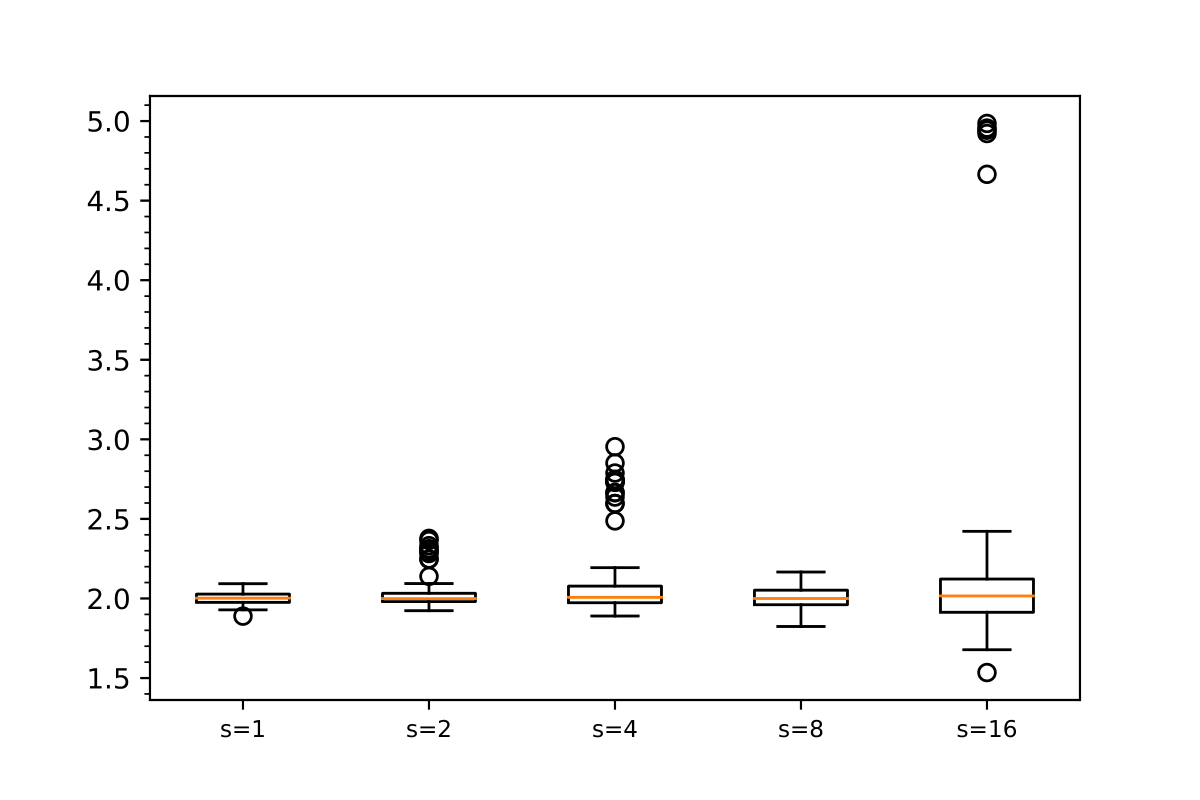}
        \caption{Confidence Intervals for $\theta_2$ (True value = 2)} \label{fig:box plots:model II: submodel 1: beta}
    \end{subfigure}\hspace*{\fill}

\caption{Box plots of Model II estimates} \label{fig:box plots:model II: submodel 1}
\end{figure}

\subsection{Model III: System size based joining}

In this model, an arriving customer only receives information about the total number of customers in the system at time $t$, i.e. $\Delta(t) = L(t)$. 
Or alternatively, an arriving customer receives information about the number of service completions that they must wait for, before their service begins. 
That is, $\Delta(t) = \big(L(t)-s + 1\big)^{+}$.
In particular, our model of interest is an M$_t$/G/$s$+H system with arrival rate, service distribution and patience distribution given by
\begin{align*}
    \lambda(\bs{\alpha}, t) &= 50 + 20\sin(4 - 0.1t) + 10\sin(1 - 0.5t),
    \\
    G &\sim \text{Gamma}(0.05,4),
    \\
    H_{\bs{\theta}} &\sim \text{Geo}(0.1).
\end{align*}
The paramaters to be estimated are therefore $\bs{\alpha}_0 = (50, 20, 10, 4, 1)$, and $\bs{\theta}_0 = 0.1$. In this example, we work with 40,000 total arrivals.
Box plots of empirical confidence intervals for the estimates of $\bs\alpha$ and $\bs\theta$ are shown in Figure \ref{fig:box plots:model III: submodel 1}.
Once again, we observe that they follow a similar trend as in the previous examples.

\begin{figure}[ht]
    \begin{subfigure}{0.45\textwidth}
        \includegraphics[width=\linewidth]{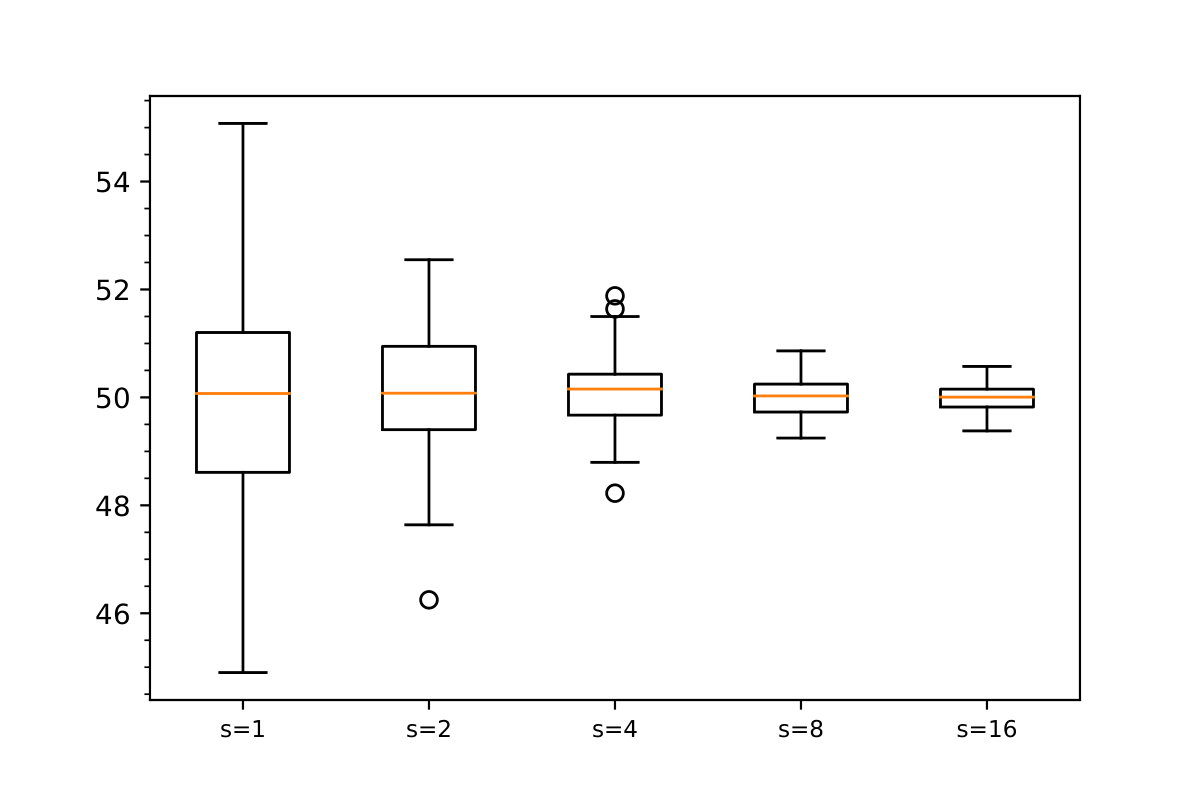}
        \caption{Estimates for $\alpha_1$ (True value = 50)} \label{fig:box plots:model III: submodel 1: a}
        \end{subfigure}\hspace*{\fill}
        \begin{subfigure}{0.45\textwidth}
        \includegraphics[width=\linewidth]{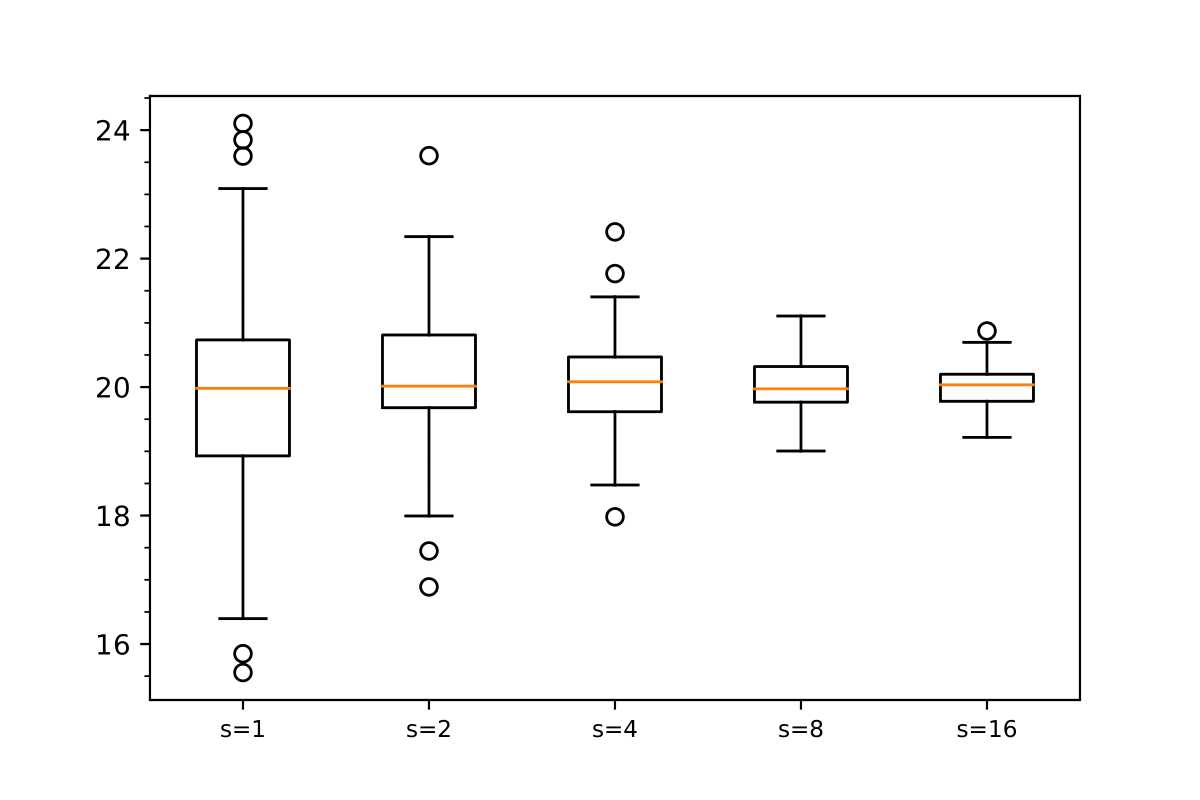}
        \caption{Estimates for $\alpha_2$ (True value = 20)} \label{fig:box plots:model III: submodel 1: b1}
    \end{subfigure}
    
    \medskip
    \begin{subfigure}{0.45\textwidth}
        \includegraphics[width=\linewidth]{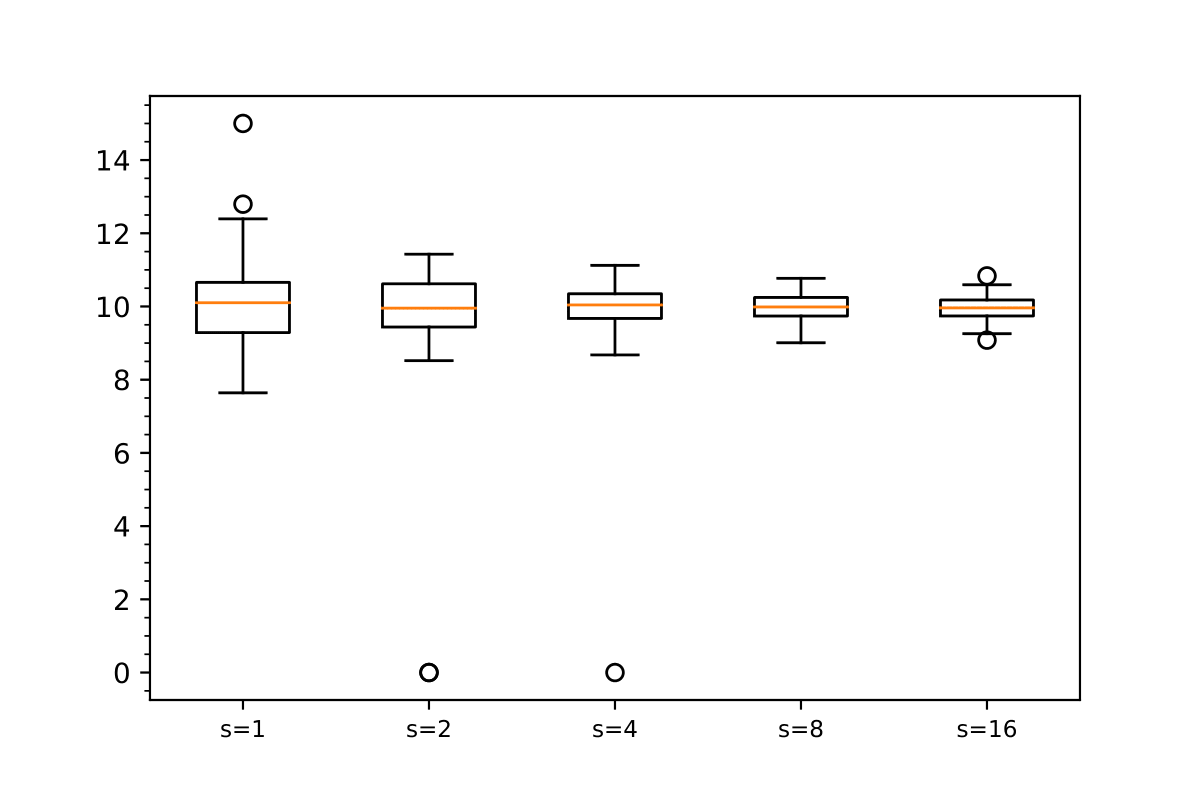}
        \caption{Estimates for $\alpha_3$ (True value = 10)} \label{fig:box plots:model III: submodel 1: b2}
        \end{subfigure}\hspace*{\fill}
        \begin{subfigure}{0.45\textwidth}
        \includegraphics[width=\linewidth]{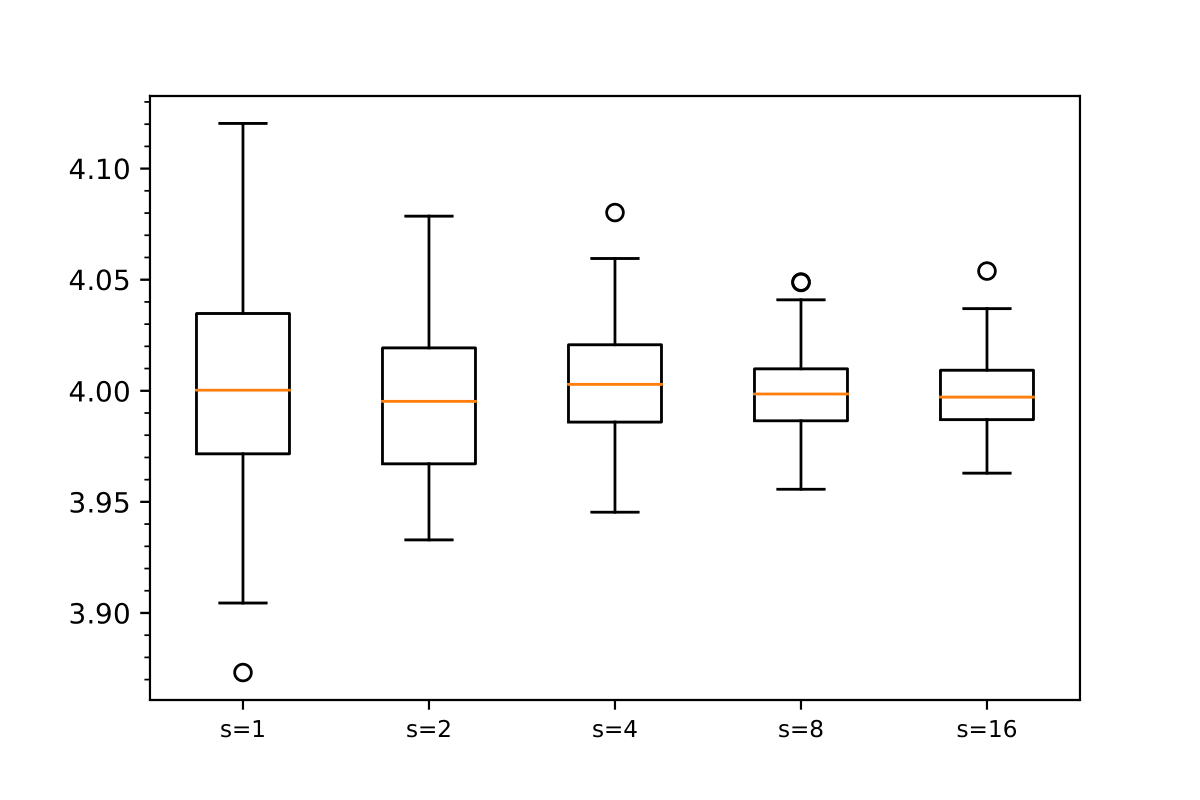}
        \caption{Estimates for $\alpha_4$ (True value = 4)} \label{fig:box plots:model III: submodel 1: c1}
    \end{subfigure}

    \medskip
    \begin{subfigure}{0.45\textwidth}
        \includegraphics[width=\linewidth]{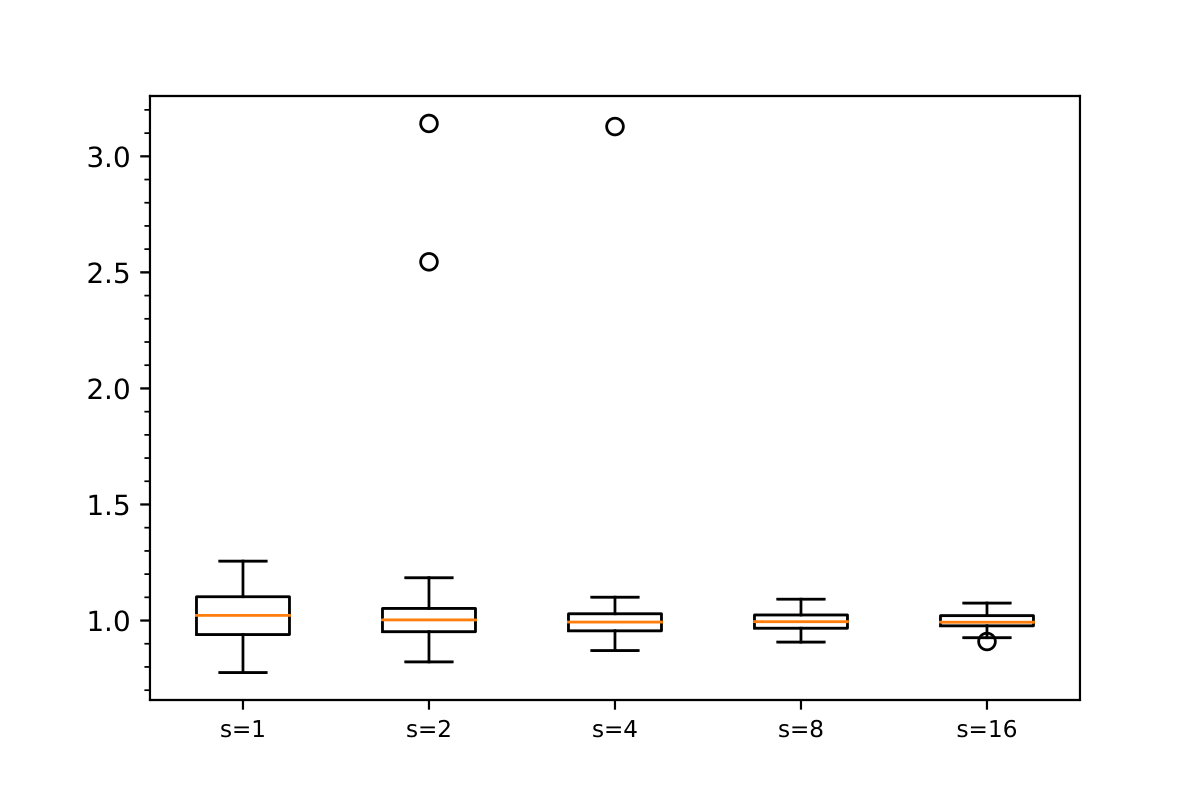}
        \caption{Estimates for $\alpha_5$ (True value = 1)} \label{fig:box plots:model III: submodel 1: c2}
        \end{subfigure}\hspace*{\fill}
        \begin{subfigure}{0.45\textwidth}
        \includegraphics[width=\linewidth]{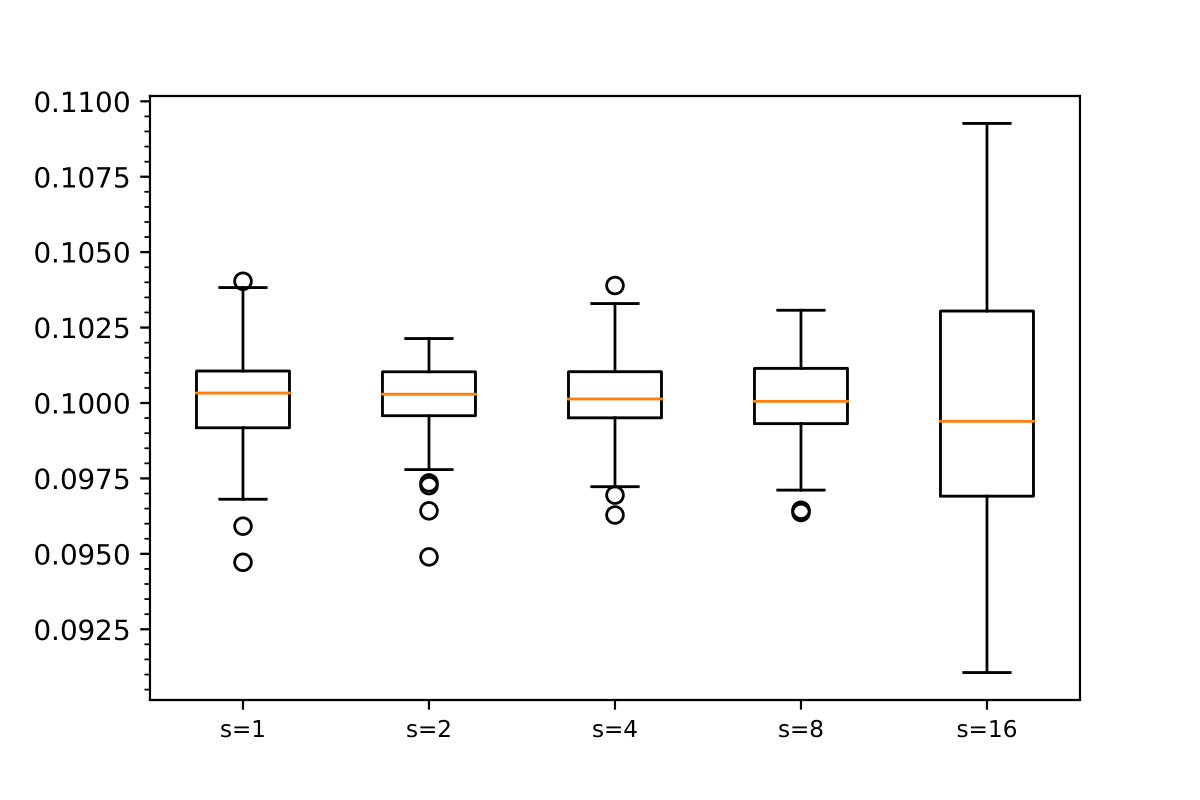}
        \caption{Estimates for $\theta$ (True value = 0.1)} \label{fig:box plots:model III: submodel 1: alpha}
    \end{subfigure}

\caption{Box plots of Model III estimates} \label{fig:box plots:model III: submodel 1}
\end{figure}

\section{CONCLUDING REMARKS}\label{section: conclusion}

In this paper, we considered a service system with a delay announcement mechanism.
Based on their patience level, customers decide to join the system or to balk.
As the underlying model we chose the versatile M$_t$/G/$s$+H queue, with a periodic time-dependent arrival rate $\lambda_{\bs{\alpha}_0}(t)$, and a patience distribution $H_{\bs{\theta}_0}(\cdot)$.
The goal of this paper was to estimate the parameter vectors $\bs{\alpha}_0$ and $\bs{\theta}_0$, even though the balking customers remain unobserved.
Through an MLE procedure, we devised estimators $\hat{\bs{\alpha}}_n$ and $\hat{\bs{\theta}}_n$ given an observed realization of the system corresponding to $n$ non-balking customers, which we proved to be consistent and asymptotically normal as $n\to\infty$.
Through numerical experiments, some of them working with delay proxies, we demonstrate the performance of the estimation procedure.
We also illustrate instances where estimation is easier and harder, with supporting explanations for the same.


Although the framework considered is rather general, several extensions could be thought of. 
An interesting generalization could be one in which the service times and patience levels are correlated, which could e.g.\ cover cases in which a customer with a large job size is naturally willing to wait longer than customers with smaller job sizes.


In the context of our paper the impatience mechanism can be seen as a decentralized control policy, in that it makes sure that the number of customers in the system does not become large. 
Thus, the material presented in this work is an example of an estimation problem based on partial information, in a system that operates under what could be called an `equalizing control'. 
In the introduction of our paper we came across instances in which (some of) the parameters were hard to identify, corresponding to situations in which the control succeeded in essentially equalizing the system occupancy, entailing that  relatively little information in our observations.
The complications arising when aiming to estimate parameters under `equalizing control measures' are discussed in detail in \cite{mendelson2023control}: the control tries to make the congestion level as constant as possible, whereas estimation would benefit from more fluctuations.

\bibliographystyle{plain}
\bibliography{bibliography.bib}

\newpage

\appendix

\begin{center}
    \underline{\textbf{\large{ONLINE APPENDIX}}}
\end{center}

\section{ANALYSIS OF M$_t$/G/$s$+H SYSTEM} \label{section: analysis of Mt/G/s+H system}

In Section \ref{section: Constructing i.i.d regeneration cycles}, we constructed a number of random variables pertaining to a regeneration cycle of an M$_t$/G/$s$+H system and then made some intuitive claims about them in Theorem \ref{theorem:properties_of_system}. 
In this appendix we prove this theorem.

\begin{proof}[Proof of Theorem \ref{theorem:properties_of_system}(a)]
    Let $\{\bs{X_n}\}_{n \geq 0}$ denote the vector of residual service times at times $0,1,2,\cdots$, i.e., $\bs{X_n} = \big(X_{n,1}, X_{n,2}, \cdots, X_{n,s} \big) \in \mathbb{R}_{+}^s$ where $X_{n,i}$ denotes the residual service time corresponding to server $i$ at time $n$. 
    Then, $\{\bs{X_n}\}_{n \geq 0}$ is a discrete-time Markov chain. Also define $\{\bs{L_{n}}\}_{n \geq 1}, \{\bs{T_{n}}\}_{n \geq 1} \in \mathbb{R}^s$ as follows: $\bs{L_n} = \big(L_{n,1}, L_{n,2}, \cdots, L_{n,s} \big)$ where $L_{n,i}$ denotes the  amount of workload server $i$ receives during time $(n-1, n]$, and $\bs{T_n} = \big(T_{n,1}, T_{n,2}, \cdots, T_{n,s} \big)$ where $T_{n,i}$ denotes the amount of service completed by server $i$ during time $(n-1, n]$. 
    If a server $i$ is never idle during $(n-1,n]$, then $T_{n,i} = 1$, otherwise $T_{n,i} < 1$.
    
    By Assumption (A4), there exists $1 > \delta > 0$ and $p \geq 2$ such that
    \begin{align*}
        \lambda_{\max}\exptn(B)\tilde{H}_{\bs{\theta}}(p-1) \leq 1-\delta.
    \end{align*}
    We define the Lyapunov function $\phi: \mathbb{R}^s_{+} \mapsto [1,\infty)$ by
    \begin{align}    
        \phi(\bs{x}) = 1 + \frac{1}{\delta} \max_{i = 1,2,\cdots, s} {x_i}. \label{equation:Lyapunov_function}
    \end{align}
    Let $\mathcal{C}$ denote the set $[0, k]^s$ for $k \gg p$, the exact lower bound for $k$ being evaluated later. 
    Suppose $\bs{X_n} \in \mathcal{C}$. 
    Then we show that $\exptn\big[\phi\big(\bs{X_{n+1}}\big) | \bs{X_n} \big]$ is bounded above by some fixed quantity, as follows:
    \begin{align}
        \exptn\big[\phi\big(\bs{X_{n+1}}\big) \ \big\vert \bs{X_n} \big] &= \exptn\bigg[1 + \frac{1}{\delta} \max_{i=1,2,\cdots,s} \big(X_{n,i} + L_{n+1,i} - T_{n+1,i}\big) \ \Big\vert \bs{X_n} \bigg] \notag
        \\
        &\leq 1 + \frac{1}{\delta} \max_{i=1,2,\cdots,s} X_{n,i} + \frac{1}{\delta}\exptn\Bigg[\sum_{i=1}^{s} L_{n+1,i} \ \Big\vert \bs{X_n} \Bigg] \notag
        \\
        &\leq 1 + \frac{k}{\delta} + \frac{\lambda_{\max}\exptn(B)}{\delta}.
    \end{align}
    Suppose now that $\bs{X_n} \in \mathbb{R}^{s} \setminus \mathcal{C}$. 
    Then, we consider two distinct cases as follows.
    
    \underline{\textbf{Case I}}: $\min_{i=1,2,\cdots,s} X_{n,i} \geq p$

    In this case,
    \begin{align}
        \exptn\big[\phi\big(\bs{X_{n+1}}\big) - \phi\big(\bs{X_n}) \ \big\vert \bs{X_n} \big] &= \ \exptn\bigg[\cancel{1} + \frac{1}{\delta} \max_{i=1,2,\cdots,s} \big(X_{n,i}+L_{n+1,i}-1\big) - \cancel{1} - \frac{1}{\delta} \max_{i=1,2,\cdots,s} X_{n,i} \ \Big\vert \bs{X_n} \bigg] \notag
        \\
        &\leq \ \frac{1}{\delta}\exptn\bigg[\cancel{\max_{i=1,2,\cdots,s} X_{n,i}} + \max_{i=1,2,\cdots,s} L_{n+1,i} -1 -\cancel{\max_{i=1,2,\cdots,s} X_{n,i}} \ \Big\vert \bs{X_n} \bigg] \notag
        \\
        &\leq \ \frac{1}{\delta} \bigg(-1 + \exptn \bigg[\sum_{i=1}^{s} L_{n+1,i} \ \Big\vert \bs{X_n} \bigg]\bigg) \notag
        \\
        &\leq \ \frac{1}{\delta} \Big(-1 + \lambda_{\max} \exptn(B) \tilde{H}_{\theta}(p-1)\Big) \notag
        \\
        &\leq \ \frac{1}{\delta} \big(-1 + 1-\delta \big) = -1.
    \end{align}
    
    \underline{\textbf{Case II}}: $ \min_{i=1,2,\cdots, s} X_{n,i} < p$
    
    Denote the total incoming workload during the interval $(n, n+1]$ by $W_{(n,n+1]}$. 
    Then,
    \begin{align}
        \exptn\big[\phi&(\bs{X_{n+1}}) - \phi(\bs{X_n}) | \bs{X_n} \big] \notag
        \\
        = \ & -\frac{1}{\delta} \prob\Big(W_{(n,n+1]} \leq k-p-1 \Big) + \exptn\Big[\big(\phi(\bs{X_{n+1}}) - \phi(\bs{X_n})\big)\indset\Big(W_{(n,n+1]} > k-p-1 \Big) \ \big\vert \bs{X_n}\Big]. \label{equation:finite_mean_length:V(X_{n+1})-V(X_n):total}
    \end{align}
    Explanation for the first term of \eqref{equation:finite_mean_length:V(X_{n+1})-V(X_n):total}: During the time interval $(n,n+1]$, the difference between the workloads of maximum and minimum workload servers is at least $k-p-1$.
    Therefore, if $W_{(n,n+1]} < k-p-1$, then it will always be assigned to some server which is not the maximum workload server. 
    In this case, $\phi(\bs{X_{n+1}}) - \phi(\bs{X_n}) = -1/\delta$.
    The second term of \eqref{equation:finite_mean_length:V(X_{n+1})-V(X_n):total} is majorized by
    \begin{align}
        \exptn\Big[\big(\phi&(\bs{X_{n+1}}) - \phi(\bs{X_n})\big)\indset\Big(W_{(n,n+1]} > k-p-1 \Big) \ \big\vert \bs{X_n}\Big] \notag
        \\
        \leq \ &\sum_{j=1}^{\infty} \frac{j}{\delta} \ \prob\Big(W_{(n,n+1]} \in \big(k-p+j-2, \ k-p+j-1\big) \ \big\vert \bs{X_n}\Big) \notag
        \\
        \leq \ &\frac{1}{\delta} \int_{k-p-1}^{\infty} (x-k+p+2) \ f_{W_{(n,n+1]} \vert \bs{X_n}}(x) \ \mathrm{d}x \notag
        \\
        \leq \ &\frac{1}{\delta} \int_{k-p-1}^{\infty} x \ f_{W_{(n,n+1]} \vert \bs{X_n}}(x) \ \mathrm{d}x \longrightarrow 0 \notag
    \end{align}
    as $k \rightarrow \infty$, because $\exptn\big[W_{(n,n+1]} \ \vert \bs{X_n}\big] \leq \lambda_{\max} \exptn[B] < \infty$.
    Therefore, we obtain
    \begin{align*}
        \lim_{k \rightarrow \infty} \exptn\Big[\big(\phi(\bs{X_{n+1}}) - \phi(\bs{X_n})\big)\indset\Big(W_{(n,n+1]} > k-p-1 \Big) \Big] \leq 0.
    \end{align*}
    Since $\prob\big(W_{(n,n+1]} \leq k-p-1 \big) \rightarrow 1$ as $k \rightarrow \infty$, we can conclude that
    \begin{align*}
        -\frac{1}{\delta} \prob\Big(W_{(n,n+1]} \leq k-p-1 \Big) + \exptn\Big[\big(\phi(\bs{X_{n+1}}) - \phi(\bs{X_n})\big)\indset\Big(W_{(n,n+1]} > k-p-1 \Big) \Big] \rightarrow -\frac{1}{\delta}
    \end{align*}
    as $k \rightarrow \infty$. Since $-1/\delta < -1$, there must exist $\widetilde{k} \in \mathbb{R}$ such that the above expression is less than $-1$. 
    Choosing $\mathcal{C} = \big[0, \widetilde{k}\big]^s$, we have verified that for all $\bs{X_n} \in \mathbb{R}^s \setminus \mathcal{C}$,
    \begin{align*}
        \exptn\big[\phi\big(\bs{X_{n+1}}\big) - \phi\big(\bs{X_n}\big) \ \big\vert \bs{X_n}] \leq -1.
    \end{align*}
    Therefore, by \cite[Proposition 6.11(a)]{benaim2022markov}, $\mathcal{C}$ is a recurrent set.
    Now, let $\mathcal{U} = \{(0,0,\cdots,0) \}$. Define the induced chain $\{\bs{Y_n}\}_{n \geq 1}$ as follows:
    \begin{align*}
        \bs{Y_n} = \bs{X_{\tau_{\mathcal{C}}^{(n)}}}.
    \end{align*}
    Choose $N = \widetilde{k}+1$, and $\bs{Y_0} = \bs{X_0} = (0,0,\cdots,0)$.
    $\bs{Y_1} \in \mathcal{C}$ implies that $\max_{i=1,2,\cdots,s} Y_{1,i} \leq \widetilde{k}$. 
    Under the scenario that no customers arrive for the next $\widetilde{k}$ units of time, every server's workload continues dropping at unit rate and therefore, there exists $i \in \{2, 3, \cdots, \widetilde{k}+1\}$ such that $\bs{Y_i} = (0,0,\cdots,0)$. 
    The probability of this event is
    \begin{align*}
        \prob\Big(\text{No arrivals during} \ \big(\tau_{\mathcal{C}}^{(1)}, \tau_{\mathcal{C}}^{(1)}+ \widetilde{k} \big) \Big) \geq e^{-\lambda_{\max}\widetilde{k}}.
    \end{align*}
    By \cite[Proposition 6.15(i)]{benaim2022markov}, we get that $\mathcal{U}$ is recurrent. 
    Furthermore, by \cite[Proposition 6.15(ii)]{benaim2022markov}, $\sup_{x \in \mathcal{C}} \exptn_{x}\big[\tau_{\mathcal{U}}\big] < \infty$. 
    Therefore, the expected return time to $(0,0,\cdots,0)$ starting at $(0,0,\cdots, 0)$ itself, is finite, i.e. $\exptn(R_1) < \infty$.
\end{proof}

\begin{proof}[Proof of Theorem \ref{theorem:properties_of_system}(b)] \
    The maximum amount of service that can be completed in regeneration cycle $1$ is bounded from above by $sR_1$, which corresponds to the scenario that each server works for the entire duration of the regeneration cycle. 
    On the other hand, the desired amount of service is given by $B_{1,1} + B_{1,2} + \cdots + B_{1,C_1}$. 
    This implies that $sR_1 \geq B_{1,1} + B_{1,2} + \cdots + B_{1,C_1}$. 
    Taking expectation on both sides and making use of the previous result, we conclude that
    \begin{align}
    \exptn\Bigg[\sum_{i=1}^{C_1} B_{1,i}\Bigg] < s\exptn\big[R_1\big] < \infty.
    \end{align}
    For $k \geq 1$, define $S_k := B_{1,1} + B_{1,2} + \cdots + B_{1,k}$. 
    Then,
    \begin{align*}
        \underbrace{\sum_{i=1}^{C_1} B_{1,i}}_{Z} = \sum_{k=1}^{\infty} S_k \indset\big\{C_1 = k\big\} = \lim_{K \rightarrow \infty} \ \underbrace{\sum_{k=1}^{K} S_k\indset\big\{C_1 = k\big\}}_{Z_K}.
    \end{align*}
    $0 \leq Z_1 \leq Z_2 \leq \cdots \leq Z$ and $\lim_{K \rightarrow \infty} Z_K = Z$. 
    Therefore, by the monotone convergence theorem,
    \begin{align*}
        \exptn&\Bigg[\sum_{i=1}^{C_1} B_{1,i}\Bigg] = \lim_{K \rightarrow \infty} \exptn \Bigg[\sum_{k=1}^{K} S_k \indset\big\{C_1 = k\big\}\Bigg] = \lim_{K \rightarrow \infty} \sum_{k=1}^{K} \exptn\Big[S_k \indset\big\{C_1 = k\big\} \Big] = \sum_{k=1}^{\infty} \exptn\Big[S_k \indset\big\{C_1 = k\big\}\Big]
        \\
        &= \sum_{k=1}^{\infty} \exptn\Bigg[\sum_{i=1}^{k} B_{1,i} \indset\big\{C_1 = k\big\} \Bigg] = \sum_{k=1}^{\infty} \sum_{i=1}^{k} \exptn\Big[B_{1,i}\indset\big\{C_1 = k\big\}\Big] \overset{\text{(i)}}{=} \sum_{i=1}^{\infty} \sum_{k=i}^{\infty} \exptn\Big[B_{1,i}\indset\big\{C_1 = k\big\} \Big]
        \\
        &= \sum_{i=1}^{\infty} \lim_{K \rightarrow \infty} \sum_{k=i}^{i+K} \exptn\Big[B_{1,i}\indset\big\{C_1 = k\big\} \Big] = \sum_{i=1}^{\infty} \lim_{K \rightarrow \infty} \exptn\Bigg[\sum_{k=i}^{i+K} B_{1,i}\indset\big\{C_1 = k\big\} \Bigg] \\&= \sum_{i=1}^{\infty} \lim_{K \rightarrow \infty} \exptn\Big[B_{1,i}\indset\big\{i \leq C_1 \leq i+K\big\}\Big]
        \\
        &\overset{\text{(ii)}}{=} \sum_{i=1}^{\infty} \exptn\bigg[\lim_{K \rightarrow \infty} B_{1,i}\indset\big\{i \leq C_1 \leq i+K\big\} \bigg] = \sum_{i=1}^{\infty} \exptn\Big[B_{1,i}\indset\big\{C_1 \geq i\big\} \Big]\\&  \overset{\text{(iii)}}{=} \sum_{i=1}^{\infty} \exptn\big[B_{1,i}\big] \prob\big(C_1 \geq i\big) = \exptn\big[B_{1,1}]\exptn\big[C_1\big].
    \end{align*}
    The equalities in the above display are justified as follows:
    (i) is allowed since all summands are positive and the summation is finite;
    (ii) again follows because of monotone convergence theorem;
    (iii) holds because the event $\big\{C_1 < i\big\}$ is independent of the random variable $B_{1,i}$. 
    Note that $\big\{C_1 < i\big\}$ depends on the arrival times, job sizes and patience levels pertaining to all balking and non-balking customers arriving until the $(i-1)$-st non-balking customer enters the system. 
    It also depends on the arrival times and patience levels of subsequent customers, but not on their job sizes. 
    In conclusion, we get
    \begin{align*}
        \exptn\Bigg[\sum_{i=1}^{C_1} B_{j,i}\Bigg] = \exptn\big[B_{1,1}\big] \exptn\big[C_1\big] < \infty,
    \end{align*}
    so that $ \exptn\big[C_1\big] < \infty$.
\end{proof}

\begin{proof}[Proof of Theorem \ref{theorem:properties_of_system}(c)]
    Let $\{t_n: n \geq 0\}$, $\{a_n: n \geq 0\}$ be defined as follows: $t_0 = 0$, and for $k \geq 0$,
    \begin{align*}
        a_k &= \text{Number of effective arrivals during } (0, t_k],
        \\
        t_{k+1} &= \min\Big\{\big\lceil t_k+ \big\rceil, \tilde{A}_{a_k+1} \Big\}. 
    \end{align*}
    Consider the process $\big\{Y_n: n \geq 0\big\}$ given by $Y_n = \big\{t_n, \bs{X(t_n)}\big\}$ where $\bs{X(t_n)}$ is the vector of residual service times at time $t_n$, i.e., $\bs{X}(t_n) \in \mathbb{R}^s$ stores the amount of residual workload corresponding to each of the $s$ servers. The evolution of this process can be described as a join-the-shortest workload system (which for the purposes here is identical to our system).
    Intuitively, the process $\big\{Y_n: n \geq 0\big\}$ captures the pending workload for each server after every effective arrival and at integral time points.
    Define
    \begin{align*}
        \tau = \inf\Big\{k \in \mathbb{N} \:\big\vert \ t_k \in \mathbb{N} \  \text{and} \ \bs{X(t_k)} = \bs{0}_{s \times 1} \Big\}.
    \end{align*}
    Then, the process $\{Y_n: n \geq 0\}$ regenerates at time $\tau$. Define a mapping 
    \begin{align*}
        f(\bs{Y_j}) = \min_{i=1,2,\cdots,s} X_i(t_j). 
    \end{align*}
    By the renewal-reward theorem, we get
    \begin{align} \label{equation:renewal_reward}
        \lim_{n \rightarrow \infty} \ \frac{1}{n} \sum_{j=1}^{n} f(\bs{Y_j}) = \frac{1}{\exptn\tau} \exptn\Bigg[\sum_{j=0}^{\tau-1} f(\bs{Y_j}) \Bigg].
    \end{align}
    By the ergodic theorem, the LHS of \eqref{equation:renewal_reward} is finite almost surely. Furthermore, $\exptn[\tau] < \infty$. Therefore,
    \begin{align*}
        \exptn\Bigg[\sum_{j=0}^{\tau-1} f(\bs{Y_j})\Bigg] < \infty.
    \end{align*}
    The above quantity represents the expected sum of virtual waiting times just after arrival instants as well as at integral time points, which is greater than the sum of virtual waiting times just before effective arrival instants.
    This implies that
    \begin{align*}
        \exptn\Bigg[\sum_{i=1}^{C_1} W_{1,i} \Bigg] \le \exptn\Bigg[\sum_{j=0}^{\tau-1} f(\bs{Y_j})\Bigg] < \infty.
    \end{align*} 
    Hence we conclude that the expected sum of waiting times of customers in a regeneration cycle is finite.
\end{proof}

\section{PROOFS OF AUXILIARY RESULTS} \label{section: appendix}
\begin{proof}[Proof for Lemma \ref{D_Andrews_lemma_1} in Section \ref{section: strong consistency}] By the triangle inequality, we can write 
\begin{align}
    &\Big\vert q\big(\bs{Z_j}, \bs{\mu'}\big) - q\big(\bs{Z_j}, \bs{\mu}\big) \Big\vert \leq \sum_{i=1}^{C_j} \underbrace{\Big\vert  \log\lambda_{\bs{\alpha}'}\big(\tilde{A}_{j,i}\big) - \log\lambda_{\bs{\alpha}}\big(\tilde{A}_{j,i}\big) \Big\vert}_{\text{(I)}} + \int_{0}^{A_{j,1}} \underbrace{\Big\vert  \lambda_{\bs{\alpha}'}(u) - \lambda_{\bs{\alpha}}(u) \Big\vert}_{\text{(II)}} \, \mathrm{d}u \ + \notag
    \\
    &\sum_{i=2}^{C_j} \underbrace{\Big\vert \log\tilde{H}_{\bs{\theta'}}\big(W_{j,i-1}+X_{j,i-1}-A_{j,i} \big) - \log\tilde{H}_{\bs{\theta}}\big(W_{j,i-1}+X_{j,i-1}-A_{j,i} \big) \Big\vert}_{\text{(III)}} \ + \notag
    \\
    &\sum_{i=2}^{C_j} \int_{0}^{A_{j,i}} \underbrace{\Big\vert \lambda_{\bs{\alpha}'}\big(u+\tilde{A}_{j,i-1}\big)\tilde{H}_{\bs{\theta'}}\big(W_{j,i-1}+X_{j,i-1}-u\big) - \lambda_{\bs{\alpha}}\big(u+\tilde{A}_{j,i-1}\big)\tilde{H}_{\bs{\theta}}\big(W_{j,i-1}+X_{j,i-1}-u\big)\Big\vert}_{\text{(IV)}} \, \mathrm{d}u \, + \notag
    \\
    &\int_{0}^{R_j - \tilde{A}_{j, C_j}} \underbrace{\Big\vert \lambda_{\bs{\alpha}'}\big(u+\tilde{A}_{j,C_j}\big)\tilde{H}_{\bs{\theta'}}\big(W_{j,C_j}+X_{j,C_j}-u\big) - \lambda_{\bs{\alpha}}\big(u+\tilde{A}_{j,C_j}\big)\tilde{H}_{\bs{\theta}}\big(W_{j,C_j}+X_{j,C_j}-u\big) \Big\vert}_{\text{(V)}} \, \mathrm{d}u.
\end{align}

Let us analyze (I), (II), (III), (IV), (V) separately. 
Upper bounds for (I) and (II) follow from Assumption (A3) and its corresponding discussion, the upper bound for
(III) follows due to (A6), and upper bounds for (IV) and (V) follow due to (A2), (A3) and (A5). Concretely,
\begin{align*}
    \textbf{(I):} \ &\Big\vert\log\lambda_{\bs{\alpha}'}\big(\tilde{A}_{j,i}\big) - \log\lambda_{\bs{\alpha}}\big(\tilde{A}_{j,i}\big) \Big\vert \leq \frac{\kappa}{\lambda_{\min}} \ d \big(\bs{\alpha}', \bs{\alpha}\big).
    \\
    \textbf{(II):} \ &\Big\vert \lambda_{\bs{\alpha}'}(u) - \lambda_{\bs{\alpha}}(u) \Big\vert \leq \kappa \ d\big(\bs{\alpha}', \bs{\alpha} \big).
    \\
    \textbf{(III):} \ &\Big\vert \log\tilde{H}_{\bs{\theta'}}\big(W_{j,i-1}+X_{j,i-1}-A_{j,i} \big) - \log\tilde{H}_{\bs{\theta}}\big(W_{j,i-1}+X_{j,i-1}-A_{j,i} \big) \Big\vert
    \\
    \leq \ & G_2\big(W_{j,i-1}+X_{j,i-1}-A_{j,i}\big) \ d\big(\bs{\theta}, \bs{\theta'}\big) \leq G_2\big(W_{j,i}\big) \ d\big(\bs{\theta}, \bs{\theta'}\big).
    \\
    \textbf{(IV):} \ & \Big|\lambda_{\bs{\alpha}'}\big(u+\tilde{A}_{j,i-1}) \tilde{H}_{\bs{\theta'}}\big(W_{j,i-1}+X_{j,i-1}-u \big) - \lambda_{\bs{\alpha}}\big(u+\tilde{A}_{j,i-1}) \tilde{H}_{\bs{\theta}}\big(W_{j,i-1}+X_{j,i-1}-u \big) \Big|
    \\
    \leq \ &\Big|\lambda_{\bs{\alpha}'}\big(u+\tilde{A}_{j,i-1}) \tilde{H}_{\bs{\theta'}}\big(W_{j,i-1}+X_{j,i-1}-u \big) - \lambda_{\bs{\alpha}'}\big(u+\tilde{A}_{j,i-1}) \tilde{H}_{\bs{\theta}}\big(W_{j,i-1}+X_{j,i-1}-u \big) \Big|
    \\
    & + \Big|\lambda_{\bs{\alpha}'}\big(u+\tilde{A}_{j,i-1}) \tilde{H}_{\bs{\theta}}\big(W_{j,i-1}+X_{j,i-1}-u \big) - \lambda_{\bs{\alpha}}\big(u+\tilde{A}_{j,i-1}) \tilde{H}_{\bs{\theta}}\big(W_{j,i-1}+X_{j,i-1}-u \big) \Big|
    \\
    \leq \ & \lambda_{\max} \Big|\tilde{H}_{\bs{\theta'}}\big(W_{j,i-1}+X_{j,i-1}-u \big) - \tilde{H}_{\bs{\theta}}\big(W_{j,i-1}+X_{j,i-1}-u \big) \Big| + \Big\vert\lambda_{\bs{\alpha}'}\big(u+\tilde{A}_{j,i-1}) - \lambda_{\bs{\alpha}}\big(u+\tilde{A}_{j,i-1}) \Big\vert
    \\
    \leq \ & \lambda_{\max} \ G_1 \ d\big(\bs{\theta}, \bs{\theta'}\big) + \kappa \ d\big(\bs{\alpha}', \bs{\alpha}\big).
    \\
    \textbf{(V):} \ & \text{Same computations as \textbf{(IV)} lead to}
    \\
    \leq \ & \lambda_{\max} \ G_1 \ d\big(\bs{\theta}, \bs{\theta'}\big) + \kappa \ d\big(\bs{\alpha}, \bs{\alpha}'\big).
\end{align*}
Therefore,
\begin{align*}
    \Big\vert q\big(\bs{Z_j}, \bs{\mu'}\big) - q\big(\bs{Z_j}, \bs{\mu}\big) \Big\vert \leq \ &C_j \frac{\kappa}{\lambda_{\min}} \ d\big(\bs{\alpha}, \bs{\alpha}'\big) + \kappa A_{j,1} \ d\big(\bs{\alpha}, \bs{\alpha}'\big) + \sum_{i=2}^{C_j} G_2\big(W_{j,i}\big) \ d\big(\bs{\theta}, \bs{\theta'}\big)
    \\
    & + \Big(\lambda_{\max} \ G_1 \ d\big(\bs{\theta}, \bs{\theta'}\big) + \kappa \ d\big(\bs{\alpha}, \bs{\alpha}'\big) \Big) \Big(A_{j,2}+A_{j,3}+ \cdots + A_{j,C_j} + R_j - \tilde{A}_{j,C_j} \Big)
    \\
    = \ & \Big( C_j\frac{\kappa}{\lambda_{\min}} + R_j \kappa \Big) \ d\big(\bs{\alpha}, \bs{\alpha}'\big) + \Bigg(\sum_{i=2}^{C_j} G_2\big(W_{j,i}\big) + \lambda_{\max}\big(R_j - A_{j,1}\big) G_1 \Bigg) \ d\big(\bs{\theta}, \bs{\theta'}\big).
\end{align*}
Separating, we get
\begin{align*}
     \Big\vert q\big(\bs{Z_j}, \bs{\mu'}\big) - q\big(\bs{Z_j}, \bs{\mu}\big) \Big\vert \leq \underbrace{d\big(\bs{\mu}, \bs{\mu'}\big)}_{h(x)=x} \times \underbrace{\Bigg(C_j\frac{\kappa}{\lambda_{\min}} + R_j \kappa + \sum_{i=2}^{C_j} G_2\big(W_{j,i}\big) + \lambda_{\max}\big(R_j - A_{j,1}\big) G_1 \Bigg)}_{B\big(\bs{Z_j}\big)}.
\end{align*}
By Assumption (A6) and Theorem \ref{theorem:properties_of_system}, it is clear that $\exptn\big[B(\bs{Z_j})\big] < \infty$.
\end{proof}

\begin{proof}[Proof of Lemma \ref{lemma:E(q_dot)=0} in Section \ref{section: strong consistency}]
For a fixed data vector ${\bs{Z}_1}$, we have
\begin{align*}
    q\big({\bs{Z}_1}, .\big) : \mathbb{R}^{k+p} \longrightarrow \mathbb{R}, \ 
    \nabla q\big({\bs{Z}_1}, .\big) : \mathbb{R}^{k+p} \longrightarrow \mathbb{R}^{k+p}.
\end{align*}
Because of Assumptions (A8), (A9), we can write
\begin{align*}
    \nabla q\big({\bs{Z}_1}, {\bs{\mu}_0}\big) = 
    \begin{bmatrix}
        \frac{\partial q\big({\bs{Z}_1}, {\bs{\mu}_0}\big)}{\partial \alpha_1}&
        \cdots&
        \frac{\partial q\big({\bs{Z}_1}, {\bs{\mu}_0}\big)}{\partial \alpha_k}&
        \frac{\partial q\big({\bs{Z}_1}, {\bs{\mu}_0}\big)}{\partial \theta_1}&\cdots&
        \frac{\partial q\big({\bs{Z}_1}, {\bs{\mu}_0}\big)}{\partial \theta_p}.   
    \end{bmatrix}^\top
\end{align*}
We need to prove that every component of $\exptn \Big[\nabla q\big({\bs{Z}_1}, {\bs{\mu}_0}\big) \Big]$ is 0. 
We prove 
\begin{align*}
    \exptn \bigg[\frac{\partial q\big({\bs{Z}_1}, {\bs{\mu}_0}\big)}{\partial \alpha_1}\bigg] = \exptn \bigg[\frac{\partial q\big({\bs{Z}_1}, {\bs{\mu}_0}\big)}{\partial \theta_1}\bigg] = 0;
\end{align*}
the other components then follow along the same lines.
By Assumptions (A2), (A6), (A7) and Theorem \ref{theorem:properties_of_system}, we can conclude that $q\big(\bs{Z_j}, {\bs{\mu}_0}\big)$ is integrable:
\begin{align*}
   \Big\vert q\big(\bs{Z_j}, {\bs{\mu}_0}\big) \Big\vert &\leq C_j \max\Big\{\big\vert\log\lambda_{\min}\big\vert, \big\vert \log\lambda_{\max}\big\vert \Big\} + \sum_{i=2}^{C_j} G_3\big(W_j\big) + \lambda_{\max} R_j
   \end{align*}
   implies
   \begin{align*}
   \exptn \Big\vert q\big(\bs{Z_j}, {\bs{\mu}_0}\big) \Big\vert &\leq \exptn\big[C_j\big] \max\Big\{\big\vert\log\lambda_{\min}\big\vert, \big\vert \log\lambda_{\max}\big\vert \Big\} + \exptn \bigg[\sum_{i=2}^{C_j} G_3\big(W_j\big) \bigg] + \lambda_{\max} \exptn[R_j] < \infty.
\end{align*}
By Assumption (A8), $q\big(\bs{Z_j}, \bs{\mu}\big)$ is continuously differentiable with respect to $\alpha_1, \alpha_2, \cdots, \alpha_k$ and $\theta_1, \theta_2, \cdots, \theta_p$. Now, we take the partial derivative of $q\big(\bs{Z_j}, \bs{\mu}\big)$ with respect to $\alpha_1$:
\begin{align*}
    \frac{\partial}{\partial \alpha_1} q\big(\bs{Z_j}, {\bs{\mu}_0}\big) = &\sum_{i=1}^{C_j} \frac{1}{\lambda_{{\bs{\alpha}_0}}\big(\tilde{A}_{j,i}\big)} \frac{\partial}{\partial \alpha_1}\lambda_{{\bs{\alpha}_0}}\big(\tilde{A}_{j,i}\big) - \int_{0}^{A_{j,1}} \frac{\partial}{\partial \alpha_1} \lambda_{{\bs{\alpha}_0}}\big(u\big) \ \mathrm{d}u
    \\
    & -\sum_{i=2}^{C_j} \int_{0}^{A_{j,i}} \frac{\partial}{\partial \alpha_1} \lambda_{{\bs{\alpha}_0}}\big(u+\tilde{A}_{j,i-1}\big)\tilde{H}_{\bs{\theta}}\big(W_{j,i-1}+X_{j,i-1}-u\big) \ \mathrm{d}u
    \\
    & + \int_{0}^{R_j - \tilde{A}_{j,C_j}} \frac{\partial}{\partial \alpha_1} \lambda_{{\bs{\alpha}_0}}\big(u+\tilde{A}_{j,C_j} \big) \tilde{H}_{\bs{\theta}}\big(W_{j,C_j} + X_{j,C_j} - u\big) \ \mathrm{d}u.
\end{align*}
By Assumption (A2), (A3),
\begin{align*}
    \bigg\vert \frac{\partial}{\partial \alpha_1} q\big(\bs{Z_j}, {\bs{\mu}_0}\big) \bigg\vert \leq \frac{\kappa}{\lambda_{\min}}C_j + \kappa R_j.
\end{align*}
Furthermore, by Theorem \ref{theorem:properties_of_system}(a), (b), 
\begin{align*}
    \exptn \Bigg[\frac{\kappa}{\lambda_{\min}} C_j + \kappa R_j \Bigg] = \frac{\kappa}{\lambda_{\min}} \exptn\big[C_j\big] + \kappa \exptn\big[R_j\big]< \infty.
\end{align*}
Therefore, by the dominated convergence theorem,
\begin{align*}
    \exptn \Bigg[\frac{\partial q\big({\bs{Z}_1}, {\bs{\mu}_0}\big)}{\partial \alpha_1} \Bigg] = \frac{\partial}{\partial \alpha_1} \exptn \big[q\big({\bs{Z}_1},{\bs{\mu}_0}\big)\big] = 0.
\end{align*}
Now we take the partial derivative of $q\big(\bs{Z_j}, \bs{\mu}\big)$ with respect to $\theta_1$.
\begin{align*}
    \frac{\partial}{\partial \theta_1} q\big(\bs{Z_j}, {\bs{\mu}_0}\big) = &\sum_{i=2}^{C_j} \frac{\partial}{\partial \theta_1} \log\tilde{H}_{{\bs{\theta}_0}}\big(W_{j,i-1}+X_{j,i-1}-A_{j,i}\big) 
    \\
    &- \sum_{i=2}^{C_j} \int_{0}^{A_{j,i}} \lambda_{{\bs{\alpha}_0}}\big(u+\tilde{A}_{j,i-1}\big) \frac{\partial}{\partial \theta_1}\tilde{H}_{{\bs{\theta}_0}}\big(W_{j,i-1}+X_{j,i-1}-u\big) \ \mathrm{d}u
    \\
    &+ \int_{0}^{R_j - \tilde{A}_{j,C_j}} \lambda_{{\bs{\alpha}_0}}\big(u+\tilde{A}_{j,C_j}\big) \frac{\partial}{\partial \theta_1} \tilde{H}_{{\bs{\theta}_0}}\big(W_{j,C_j}+X_{j,C_j}-u\big) \ \mathrm{d}u.
\end{align*}
By Assumptions (A2), (A5) and (A7),
\begin{align*}
    \bigg\vert \frac{\partial}{\partial \theta_1} q\big(\bs{Z_j}, {\bs{\mu}_0}\big) \bigg\vert \leq \sum_{i=2}^{C_j} G_2\big(W_{j,i}\big) + \lambda_{\max}\ G_1 \big(R_j - A_{j,1}\big).
\end{align*}
Furthermore, by Assumption (A6), and Theorem \ref{theorem:properties_of_system}(a), (c), 
\begin{align*}
    \exptn \Bigg[\sum_{i=2}^{C_j} G_2\big(W_{j,i}\big) + \lambda_{\max} G_1\big(R_j - A_{j,1}\big) \Bigg] = \exptn\Bigg[\sum_{i=2}^{C_j} G_2\big(W_{j,i}\big) \Bigg] + \lambda_{\max} \ G_1 \exptn\big[R_j - A_{j,1} \big] < \infty.    
\end{align*}
Therefore, again by the dominated convergence theorem, 
\begin{align*}
    \exptn \Bigg[\frac{\partial q\big({\bs{Z}_1}, {\bs{\mu}_0}\big)}{\partial \theta_1} \Bigg] = \frac{\partial}{\partial \theta_1} \exptn q\big({\bs{Z}_1},{\bs{\mu}_0}\big) = 0.
\end{align*}
The same argument applies to $\alpha_2, \cdots, \alpha_k$ and $\theta_2, \cdots, \theta_p$, so that
Therefore, $\exptn \big[\nabla q\big({\bs{Z}_1}, {\bs{\mu}_0}\big) \big] = 0$.
\end{proof}

\begin{proof}[Proof for Proposition \ref{propn} in Section \ref{section: joining decisions based on incomplete information}]
For our new objects $q(\bs{Z}_j, \bs{\mu})$, we re-prove Lemma \ref{D_Andrews_lemma_1}. We start with
\begin{align*}
    &\Big\vert q\big(\bs{Z_j}, \bs{\mu'}\big) - q\big(\bs{Z_j}, \bs{\mu}\big) \Big\vert \leq \sum_{i=1}^{C_j} \underbrace{\Big\vert  \log\lambda_{\bs{\alpha}'}\big(\tilde{A}_{j,i}\big) - \log\lambda_{\bs{\alpha}}\big(\tilde{A}_{j,i}\big) \Big\vert}_{\text{(I)}} + \int_{0}^{A_{j,1}} \underbrace{\Big\vert  \lambda_{\bs{\alpha}'}(u) - \lambda_{\bs{\alpha}}(u) \Big\vert}_{\text{(II)}} \, \mathrm{d}u \ + \notag
    \\
    &\sum_{i=2}^{C_j} \underbrace{\Big\vert \log\tilde{H}_{\bs{\theta'}}\big(\psi\big(\mathcal{R}(\tilde{A}_{j,i})\big)\big) - \log\tilde{H}_{\bs{\theta}}\big(\psi\big(\mathcal{R}(\tilde{A}_{j,i})\big)\big) \Big\vert}_{\text{(III)}} \ + \notag
    \\
    &\sum_{i=2}^{C_j} \int_{0}^{A_{j,i}} \underbrace{\Big\vert \lambda_{\bs{\alpha}'}\big(u+\tilde{A}_{j,i-1}\big)\tilde{H}_{\bs{\theta'}}\big(\psi\big(\mathcal{R}(u+\tilde{A}_{j,i-1})\big)\big) - \lambda_{\bs{\alpha}}\big(u+\tilde{A}_{j,i-1}\big)\tilde{H}_{\bs{\theta}}\big(\psi\big(\mathcal{R}(u+\tilde{A}_{j,i-1})\big)\big)\Big\vert}_{\text{(IV)}} \, \mathrm{d}u \, + \notag
    \\
    &\int_{0}^{R_j - \tilde{A}_{j, C_j}} \underbrace{\Big\vert \lambda_{\bs{\alpha}'}\big(u+\tilde{A}_{j,C_j}\big)\tilde{H}_{\bs{\theta'}}\big(\psi\big(\mathcal{R}(u+\tilde{A}_{j,C_j})\big)\big) - \lambda_{\bs{\alpha}}\big(u+\tilde{A}_{j,C_j}\big)\tilde{H}_{\bs{\theta}}\big(\psi\big(\mathcal{R}(u+\tilde{A}_{j,C_j})\big)\big) \Big\vert}_{\text{(V)}} \, \mathrm{d}u.
\end{align*}
Observe that the upper bounds for terms (I), (II), (IV) and (V) follow in the exact same way as from the proof of Lemma \ref{D_Andrews_lemma_1}. However, for term (III), under (A6) we have
\begin{align*}
    \Big\vert \log\tilde{H}_{\bs{\theta'}}\big(\psi\big(\mathcal{R}(\tilde{A}_{j,i})\big)\big) - \log\tilde{H}_{\bs{\theta}}\big(\psi\big(\mathcal{R}(\tilde{A}_{j,i})\big)\big) \Big\vert &\leq G_2\big(\psi\big(\mathcal{R}(\tilde{A}_{j,i})\big)\big) \ d\big(\theta, \theta'\big). 
\end{align*}
Further, we get
\begin{align*}
    G_2\big(\psi\big(\mathcal{R}(\tilde{A}_{j,i})\big)\big) &\leq G_2\big(\tilde{M}W_{j,i}\big), \ \text{under Assumption (A11)(a)}.
    \\
    G_2\big(\psi\big(\mathcal{R}(\tilde{A}_{j,i})\big)\big) &\leq G_2\big(\psi_1\big(L(\tilde{A}_{j,i})\big)\big), \ \text{under Assumption (A11)(b)}.
\end{align*}
Under (A11)(a), we therefore get
\begin{align*}
    \Big\vert q\big(\bs{Z_j}, \bs{\mu'}\big) - q\big(\bs{Z_j}, \bs{\mu}\big) \Big\vert \leq \underbrace{d\big(\bs{\mu}, \bs{\mu'}\big)}_{h(x)=x} \times \underbrace{\Bigg(C_j\frac{\kappa}{\lambda_{\min}} + R_j \kappa + \sum_{i=2}^{C_j} G_2\big(\tilde{M} W_{j,i}\big) + \lambda_{\max}\big(R_j - A_{j,1}\big) G_1 \Bigg)}_{B_1(Z_j)},
\end{align*}
while under (A11)(b), we get
\begin{align*}
    \Big\vert q\big(\bs{Z_j}, \bs{\mu'}\big) - q\big(\bs{Z_j}, \bs{\mu}\big) \Big\vert \leq \underbrace{d\big(\bs{\mu}, \bs{\mu'}\big)}_{h(x)=x} \times \underbrace{\Bigg(C_j\frac{\kappa}{\lambda_{\min}} + R_j \kappa + \sum_{i=2}^{C_j} G_2\big(\psi_1\big(L(\tilde{A}_{j,i})\big)\big) + \lambda_{\max}\big(R_j - A_{j,1}\big) G_1 \Bigg)}_{B_2(Z_j)}.
\end{align*}
Recalling that $G_2$ is linear together with \ref{theorem:properties_of_system}, it is clear that $\exptn\big[B_1(Z_j)\big] < \infty$.
Recalling that $G_2$ and $\psi_1$ are linear, and that $\exptn\big[C_1^2\big] < \infty$, together with Theorem \ref{theorem:properties_of_system}, it is clear that $\exptn\big[B_2(\bs{Z_j})\big] < \infty$.
This ensures that the Lemma \ref{D_Andrews_lemma_1} holds. 
Subsequently, the proof of Theorem \ref{theorem:asymptotic_normality} relies only on the objects $q(\bs{Z}_j, \bs\mu)$ which are i.i.d.\ and whose sum is the log-likelihood. 
We have managed to define such objects in Equation \eqref{eqn:q:incomplete}, and hence the proof carries over directly.
\end{proof}

\end{document}